\documentclass[a4paper]{amsart}
\usepackage[T1]{fontenc}
\usepackage[utf8]{inputenc}
\usepackage{lmodern}
\usepackage{enumerate}
\usepackage{amssymb,amsxtra}
\usepackage[all]{xy}
\usepackage{xcolor}
\usepackage{nicefrac,mathtools}
\usepackage{microtype}
\usepackage{amscd}
\usepackage{geometry}
\usepackage{mathbbol}

\usepackage[pdftitle={...},
 pdfauthor={Dami\'an Ferraro},
 pdfsubject={Mathematics}]{hyperref}
\usepackage{cite}

\numberwithin{equation}{section}
\theoremstyle{plain}
\newtheorem{theorem}[equation]{Theorem}
\newtheorem{lemma}[equation]{Lemma}
\newtheorem{proposition}[equation]{Proposition}
\newtheorem{corollary}[equation]{Corollary}
\theoremstyle{definition}
\newtheorem{definition}[equation]{Definition}
\newtheorem{notation}[equation]{Notation}
\theoremstyle{remark}
\newtheorem{remark}[equation]{Remark}
\newtheorem{example}[equation]{Example}
\newtheorem{question}[equation]{Question}

\newcommand{\bB}{\mathbb{B}}
\newcommand{\bC}{\mathbb{C}}

\newcommand{\bk}{\mathbb{k}}

\newcommand{\bL}{\mathbb{L}}
\newcommand{\bN}{\mathbb{N}}
\newcommand{\bR}{\mathbb{R}}
\newcommand{\bZ}{\mathbb{Z}}

\newcommand{\cA}{\mathcal{A}}
\newcommand{\cB}{\mathcal{B}}
\newcommand{\cC}{\mathcal{C}}
\newcommand{\cD}{\mathcal{D}}

\newcommand{\cL}{\mathcal{L}}

\newcommand{\cT}{\mathcal{T}}

\newcommand{\cY}{\mathcal{Y}}

\newcommand{\cX}{\mathcal{X}}

\newcommand{\intform}[1]{{#1}}
\newcommand{\intformexpl}[1]{\widetilde{#1}}
\newcommand{\intregrep}[1]{\Lambda^{#1}}
\newcommand{\intregrepEN}[1]{\lambda^{#1}}

\newcommand{\lt}{\mathtt{lt}}

\newcommand{\omax}{\otimes_{\max}}

\newcommand{\quotient}[3]{{#1}\left[{#2}/{#3}\right]}

\newcommand{\regrep}[1]{\Lambda^{#1}}
\newcommand{\regrepEN}[1]{\lambda^{#1}}
\newcommand{\rmu}{{r^{-1}}}

\newcommand{\smu}{{s^{-1}}}
\newcommand{\sbgp}{\leqslant}
\newcommand{\tmu}{{t^{-1}}}

\providecommand{\cspn}{\overline{\mathop{\rm span}}}
\providecommand{\Ind}{{\mathop{\rm Ind}}}
\providecommand{\red}{{\mathop{\rm r}}}
\providecommand{\spn}{{\mathop{\rm span}}}
\providecommand{\supp}{{\mathop{\rm supp}}}

\date{\today}

\begin{document}

\title[Induction, absorption and weak containment]{Induction, absorption and weak containment of *-representations of Banach *-algebraic bundles}
\author{Dami\'{a}n Ferraro}
\email{dferraro@litoralnorte.udelar.edu.uy}
\address{Departamento de Matemática y Estadística del Litoral, CENUR
	Litoral Norte, Universidad de la República\\Gral. Rivera 1350\\
	Salto, Uruguay 50000}
\date{\today}

\begin{abstract}
Given a Fell bundle $\cB=\{B_t\}_{t\in G}$ over a LCH group and a closed subgroup $H\subset G,$ we show that all the *-representations of $\cB_H:=\{B_t\}_{t\in H}$ can be induced to *-representations of $\cB$ by means of Fell's induction process; which we describe as induction via a  *-homomorphism $q^{\cB}_H\colon C^*(\cB)\to \bB(X_{C^*(\cB_H)}).$ 
The quotients $C^*_H(\cB):=q^\cB_H(C^*(\cB))$ are intermediate to $C^*(\cB)= C^*_G(\cB)$ and $C^*_\red(\cB)=C^*_{\{e\}}(\cB)$ because every inclusion of subgroups $H\subset K\subset G$ gives a unique quotient map $q^\cB_{HK}\colon C^*_K(\cB)\to C^*_H(\cB)$ such that $q^\cB_{HK}\circ q^\cB_K=q^\cB_H.$
All along the article we try to find conditions on $\cB,$ $G,\ H$ and $K$ (e.g. saturation, nuclearity or weak containment) that imply $q^\cB_{HK}$ is faithful.
One of our main tools is a blend of Fell's absorption principle (for saturated bundles) and a result of Exel and Ng for reduced cross sectional C*-algebras.
We also show that given an imprimitivity system $\langle T,P\rangle$ for $\cB$ over $G/H,$ if $H$ is open or has open normalizer in $G,$ then $T$ is weakly contained in a *-representation induced from $\cB_H$ (even if $\cB$ is not saturated).
Given normal and closed subgroups of $G,$ $H\subset K,$ we construct a Fell bundle $\cC$ over $G/K$ such that $C^*_\red(\cC)=C^*_H(\cB).$
We show that $q^\cB_H$ is faithful if and only if both  $q^{\cC}_{\{e\}}$ and $q^{\cB_K}_H$ are.
\end{abstract}

\maketitle

\tableofcontents

\section{Introduction}

The \textit{rigged spaces} introduced by Rieffel in  \cite{Rf74} are nowadays commonly known as (pre)Hilbert modules and are the heart of the  equivalence theory developed in \cite{RIEFFEL197451,Rf82Morita}; which we refer to as \textit{Morita-Rieffel equivalence} (following the suggestion of \cite{exel2000morita}).

In Rieffel's theory, an induction process from the C*-algebra $B$ to the C*-algebra $A$ is determined by a non degenerate *-homomorphism $\theta\colon A\to \bB(X_B),$ where $X_B$ is a (right) $B-$Hilbert module and $\bB(X_B)$ the C*-algebra of adjointable maps from $X_B$ to $X_B.$
Every *-representation on a Hilbert space $\pi\colon B\to \bB(Y)$  induces a *-representation $\Ind^\theta(\pi)\colon A\to \bB(X_B\otimes_\pi Y),$ where $X_B\otimes_\pi Y$ is the $\pi-$balanced tensor product and $\Ind^\theta(\pi)_a(\xi\otimes_\pi \eta)\equiv (\theta_a\otimes_\pi 1)(\xi\otimes_\pi \eta)=\theta_a\xi\otimes_\pi \eta.$

Mackey and Blattner's induction processes \cite{Mk52,blattner1961induced} generate (unitary) representations of a separable group $G$ out of representations of a closed subgroup $H \subset G.$
Rieffel described these processes as induction via a *-homomorphisms $\intregrep{H}\colon C^*(G)\to \bB(X_{C^*(H)})$ and Fell extended these ideas to the realm of Banach *-algebraic bundles \cite{MR0457620,FlDr88}.
Both Rieffel and Fell noticed that separability is not essential.

As it is implicit in the preceding paragraphs, we prefer to use right (over left) Hilbert modules and all our inner products are linear in the second variable, even those of Hilbert spaces.
When dealing with functions we sometimes write $T_b$ instead of $T(b),$ specially if a compact notation is needed.
In this work all the (topological) groups are assumed to be locally compact and Hausdorff (LCH).
Whenever $G$ is such an object, the expression $H\sbgp G$ means that $H$ is a closed subgroup of $G.$

When trying to understand Fell's induction process it is convenient to think of a Banach *-algebraic bundle $\cB=\{B_t\}_{t\in G}$ as group $G$ with a set of coefficients (fibre) $B_t$ associated to every $t\in G.$
Each fibre is a Banach space and there is an involution $a\mapsto a^*$ and a product $(a,b)\mapsto ab$ for the coefficients, both compatible with the product and involution of $G$ ($B_sB_t\subset B_{st}$ and $B_s^*\subset B_{\smu}$).
A Banach *-algebraic bundle $\cB=\{B_t\}_{t\in G}$ is a Fell bundle (or C*-algebraic bundle) if, for all $b\in \cB,$ $\|b^*b\|=\|b\|^2$ and $b^*b$ is positive in the C*-algebra $B_e$ ($e$ being the unit of $G$).
This is the case of the trivial one dimensional complex bundle $\cT_G=\{\bC\}_{t\in G}$ over $G.$

The representation theory of Banach *-algebraic bundles is studied in depth in \cite{FlDr88}, which is our main source of definitions and constructions (e.g. those of cross sections, inductive limit topology, integrated forms, induction and weak containment).
Sometimes we do not follow the notation of \cite{FlDr88}, like when we denote $C_c(\cB)$ the set of compactly supported continuous cross sections of $\cB.$
We write $L^1(\cB)$ and $C^*(\cB)$ for the $L^1-$ and $C^*-$cross sectional algebras of $\cB$ (the latter being the enveloping C*-algebra of the former).
The integration of *-representations gives a one to one correspondence between (non degenerate) *-representations of $\cB$ and $L^1(\cB).$
We use the same symbol for a *-representation and its (dis)integrated form.
For example, given a representation $T$ of $\cB,$ the expression $T_x$ stands for $T$ computed at $x\in \cB$ or the integrated form of $T$ computed at  $x\in L^1(\cB).$
Whenever it is convenient to make an explicit distinction between $T$ and its integrated form we denote $\widetilde{T}$ the latter or add the symbol $\widetilde{\ }$ somewhere in the notation.
Recall from \cite[VIII 16.4]{FlDr88} that $L^1(\cB)$ is reduced provided that $\cB$ is a Fell bundle. 
In this situation $L^1(\cB)$ is a *-subalgebra of $C^*(\cB)$ and we may think of integrated forms as defined in $C^*(\cB).$

Given a Banach *-algebraic bundle $\cB=\{B_t\}_{t\in G}$ 
and $H  \sbgp G,$ the reduction $\cB_H:=\{B_t\}_{t\in H}$ is a Banach *-algebraic bundle with the structure inherited from $\cB.$
As explained in  \cite{FlDr88}, only the $\cB-$\textit{positive} *-representations of $\cB_H$ are suitable to perform Fell's abstract and concrete induction processes.
By \cite[XI 9.26]{FlDr88}, when applied to the same $\cB-$positive *-representation $T$ of $\cB_H,$ the two processes give unitary equivalent  *-representations (both denoted $\Ind_H^\cB(T)$).

In \cite[XI 11.10]{FlDr88} Fell shows that every *-representation of $\cB_H$ is $\cB-$positive provided that $\cB$ is a saturated Fell bundle.
He also asks if this holds when the saturation hypothesis is removed, posing the ``positivity problem''; which we solve on the affirmative. 
Our solution has several consequences.
Firstly, for every Banach *-algebraic bundle $\cB$ over $G,$ all $H\sbgp G$ and any *-representation $T\colon \cB_H\to \bB(X);$ it follows that $T$ is $\cB-$positive if and only if $T_{b^*b}\geq 0$ for all $b\in \cB.$ 
Secondly, the representation theory of a Banach *-algebraic bundle $\cB$ (the induction process included) is equivalent to that of its bundle $C^*-$completion, which is a Fell bundle.
We may then continue this introduction working with Fell bundles.

In \cite{ExNg} Exel and Ng construct the reduced C*-algebra $C^*_\red(\cB)$ of a Fell bundle $\cB=\{B_t\}_{t\in G}$ as the image of a *-homomorphism $\regrepEN{\cB}\colon C^*(\cB)\to \bB(L^2_e(\cB)),$ $L^2_e(\cB)$ being a $B_e-$right Hilbert module.
It is a very nice exercise to verify that given a *-representation $T\colon B_e\equiv \cB_{\{e\}}\to \bB(X),$ the integrated form of the abstractly induced representation $\Ind_{\{e\}}^\cB(T)\equiv \Ind_e^\cB(T)$ is  $\Ind^{\regrepEN{\cB}}(T)\equiv \Ind^{\regrepEN{\cB}}(\intformexpl{T}) .$
Thus Fell's induction process $T\rightsquigarrow \Ind_e^\cB(T)$ is induction via $\regrepEN{\cB}.$

Following the work of Exel and Ng one can take any subgroup $H \sbgp G$ and construct a right $C^*(\cB_H)-$Hilbert module $L^2_H(\cB)$ and a *-homomorphism 
\begin{equation}\label{equ:firt mention to regrep}
\regrep{H\cB}\colon C^*(\cB)\to \bB(L^2_H(\cB))
\end{equation}
such that, for every *-representation $T\colon \cB_H\to \bB(X),$ $\Ind^{\regrep{H\cB}}(T)=\Ind_H^\cB(T);$ which we write $\Ind^{\regrep{H\cB}}(\widetilde{T})=\widetilde{\Ind}_H^\cB(T)$ if we want to emphasize that we are considering integrated forms.
This describes Fell's induction process, in general, as induction via a *-homomorphism.
Moreover, the constructions can be performed in such a way that:  $L^2_e(\cB)\equiv L^2_{\{e\}}(\cB);$ $\regrepEN{\cB}\equiv \regrep{\{e\}\cB};$ $L^2_G(\cB)\equiv C^*(\cB);$ $\intregrep{\cB}:=\intregrep{G\cB}$ is the canonical inclusion $C^*(\cB)\to \bB(C^*(\cB))$ and, finally, Rieffel's $\regrep{H}$ may be identified with $\regrep{H\cT_G}.$

The image $C^*_H(\cB):=\regrep{H\cB}(C^*(\cB))$ is a C*-algebra whose (non degenerate) *-representations are in a one to one correspondence with the *-representations of $\cB$ which are weakly contained in the family of *-representations induced from $\cB_H.$
It is then natural to say $\cB$ has the $H-$weak containment property ($H-$WCP) if $\regrep{H\cB}$ is faithful,  for this means that every *-representation of $\cB$ is weakly contained in a *-representation induced from $\cB_H.$

Notice that $\cB$  has the $\{e\}-$WCP (WCP for short) if and only if $\cB$ is amenable in the sense of \cite{Exel1997Amenability,ExNg} (i.e. $\intregrepEN{\cB}$ is faithful).
In \cite{Exel1997Amenability} Exel presented an approximation property (AP) for Fell bundles over discrete groups and showed it implies amenability.
Later, in \cite[Proposition 25.10]{Exlibro}, he proved that $C^*_\red(\cB)$ is nuclear provided that $B_e$ is nuclear and $\cB$ has the AP, the converse is due to Buss and Echterhoff \cite{buss2020amenability}.
Exel's AP turned out to be equivalent to a suitable translation of Anantharaman-Delaroche's notion of amenable W*-(and C*-)dynamical system \cite{ADaction1979,ADactionII1982,ADsystemes1987,abadie2021amenability} (AD-amenability).
These results indicate that, at least for Fell bundles over discrete groups, the AP is the ``correct'' notion of amenability (not the WCP).
Exel and Ng  generalised Exel's AP to Fell bundles over LCH groups and showed this new AP implies the WCP \cite{ExNg}.
Mc Kee and Pourshahami \cite{mckee2020amenable} consider Exel-Ng's AP in the realm of C*-dynamical systems and they relate it to AD-amenability.
To conform with the latest developments on (AD-)amenable actions on C*-algebras, when dealing with Fell bundles, we prefer to use the phrase \textit{weak containment} over \textit{amenable} or \textit{amenability}.

As shown in \cite{ExNg}, if $G$ is amenable (i.e. $\cT_G$ has the WCP) then every Fell bundle over $G$ has the AP and, consequently, it has the WCP.
This motivates the following.

\begin{question}\label{q:G has the HWCP then B has the WCP}
Let $G$ be a group and let $H \sbgp G$ be such that $G$ has the $H$-WCP. (i.e. $\intregrep{H}$ is faithful).
Does every Fell bundle over $G$ have the $H-$WCP?
\end{question}

Lau and Paterson's characterisation of amenability \cite[Corollary 3.2]{lau1991inner} implies that a group $G$ is amenable if and only if it is inner amenable and $C^*_\red(G)$ is nuclear.
Every discrete group $G$ is inner amenable, so it is amenable if and only if $C^*_\red(G)$ is nuclear.
Based on this we ask:

\begin{question}\label{q:inner amenability and nuclearity}
Suppose $\cB$ is a Fell bundle over an inner amenable group $G$ with $C^*_\red(\cB)$ nuclear. 
Does $\cB$ have the WCP?
If true, take a subgroup $H \sbgp G$ such that $C^*_H(\cB)$ is nuclear. 
Does $\cB$ have the $H-$WCP?
\end{question}

Fell's Imprimitivity Theorem for saturated bundles \cite[XI 14.18]{FlDr88} is a generalization of Mackey's Imprimitivity Theorem.
The former characterizes the $*-$representations $T\colon \cB\to \bB(X)$ which are induced by *-representations of $\cB_H$ as those being part of a system of imprimitivity $\langle T,P\rangle$ for $\cB$ over $G/H.$
Recall that this means $P$ is a regular $X$-projection-valued Borel measure on $G/H$ such that $T_bP(W)=P(tW)T_b$ for all $b\in B_t,$ $t\in G$ and Borel subset $W$ of $G/H.$

Every $*-$representation $S\colon \cB_H\to \bB(X)$ can be used to produce a system of imprimitivity $\langle \Ind_H^\cB(S),P^S\rangle$ for $\cB$ over $H,$ which is called the induced system \cite[XI 14.3]{FlDr88}.
In the group case (i.e. trivial bundles) every unitary representation $U\colon H\to \bB(X)$ induces a system of imprimitivity $\langle \Ind_H^G(U),P^U\rangle$ for $G$ over $G/H.$
So, given a *-representation $T\colon \cB\to \bB(X),$ one can form the system of imprimitivity $\langle T\otimes \Ind_H^G(U),1\otimes P^U\rangle$ on the Hilbert space $X\otimes Y.$ 
If $\cB$ is saturated, $T\otimes \Ind_H^G(U)$ may be described as an induced representation:

\begin{proposition}[Fell's Absorption Principle, {\cite[XI 13.9]{FlDr88}}]\label{prop:Fell AbsPrin}
	Let $\cB$ be a saturated Fell bundle over $G,$ $H \sbgp G,$ $U\colon H\to \bB(X)$ a unitary representation and $T\colon \cB\to \bB(Y)$ a non degenerate *-representation.
	Then $T\otimes \Ind_{H}^G(U)$ is unitary equivalent to $\Ind_H^\cB((T|_{\cB_H})\otimes U).$
\end{proposition}

The situation for non saturated bundles is not clear, but there are some clues on the existing bibliography.
For the special case $H=\{e\}$  and $U\colon H\to \bC$ the trivial representation, $\Ind_H^G(U)$ is the left regular representation $\regrepEN{}\colon G\to \bB(L^2(G)).$
The tensor products $T\otimes \Ind_H^G(U)=T\otimes \regrepEN{}$ are considered in \cite{ExNg}, where the integrated form of $T\otimes \regrepEN{}$ is denoted $\mu_{\regrepEN{},T}$ and $\phi\equiv T|_{B_e}.$
Putting together \cite[Lemma 2.12]{ExNg} and the ideas explained right before its statement, one can deduce the existence of a C*-isomorphism $
\varphi \colon \intform{\mu}_{\regrepEN{},T}(C^*(\cB))\to (\regrep{\cB}\otimes_\phi 1)(C^*(\cB))$
such that $\varphi \circ \intform{\mu}_{\regrepEN{},T}=\regrep{\cB}\otimes_\phi 1.$
After realizing (or recalling) that $\regrep{\cB}\otimes_\phi 1=\Ind_H^\cB(\phi)\equiv \Ind_H^\cB(T|_{B_e}),$ one gets

\begin{theorem}[Exel-Ng's Absorption Principle]\label{thm:EXNG absorption principle}
	Let $\cB$ be Fell bundle and $T\colon \cB\to \bB(X)$ a non degenerate *-representation.
	Then $T\otimes \regrepEN{} $ is weakly equivalent to $\Ind_{\{e\}}^\cB(T|_{B_e}) .$
	In particular, the integrated form of $T\otimes \regrepEN{}$ factors through a faithful *-representation of $C^*_\red(\cB)$ if (and only if) $T|_{B_e}$ is faithful.
\end{theorem}

The two absorption principles stated above motivate the following.

\begin{question}\label{q:absorption principle}
  Let $\cB$ be a Fell bundle over $G,$ $H \sbgp G,$ $T\colon \cB\to \bB(X)$ a non degenerate *-representation and $U\colon H\to \bB(Y)$ a unitary representation.
  Can we describe, up to weak containment, $T\otimes \Ind_H^G(U)$ as an induced representation?
\end{question}

As explained before, $T\otimes \Ind_H^G(U)$ is part of a system of imprimitivity, so one may also ask:

\begin{question}\label{q:system of imprimitivity weakly contained}
  Let $\cB$ be a Fell bundle over $G,$ $H \sbgp G$ and $T\colon \cB\to \bB(X)$ a non degenerate *-representation which is part of a system of imprimitivity $\langle T,P\rangle$ for $\cB$ over $G/H.$
  Is $T$ weakly equivalent (or weakly contained) in a *-representation induced from $\cB_H$?
\end{question}

Some other questions and problems will be posed along the article.
The ones we have presented are enough, we hope, to justify the division into sections we present below.

We start Section~\ref{sec:rep on Hilbert modules} by developing a with a minimal working theory of *-representations of Banach *-algebraic bundles on Hilbert modules.
We are particularly interested in being able to (dis)integrate *-representations of Fell bundles on Hilbert modules.
As an application we construct, for every Fell bundle $\cB=\{B_t\}_{t\in G},$ two *-homomorphisms
\begin{align*}
\psi\colon C^*(\cB)&\to \bB(C^*_\red(G)\otimes_{\max} C^*_\red(\cB)) &
\varphi\colon C_\red^*(\cB)& \to \bB(C^*_\red(G)\otimes_{\min} C^*_\red(\cB))
\end{align*}
whose disintegrated forms (also denoted $\psi$ and $\varphi,$ respectively) may be characterised by saying that $\psi( \regrep{\cB}_b )= \regrepEN{}_t\otimes \regrepEN{\cB}_b $ and $\varphi(\regrepEN{\cB}_b) =\regrepEN{}_t\otimes \regrepEN{\cB}_b$ for all $b\in B_t$ and $t\in G.$
We prove that $\varphi$ is always faithful and that $\psi$ is so if $G$ is inner amenable.
Finally, we use these maps to answer the first part of Question~\ref{q:inner amenability and nuclearity} (see Corollary~\ref{cor:inner amenability and nuclearity}).

The positivity problem is solved at the beginning of Section~\ref{sec:Banac bundles, completions and positivity}, which we continue by presenting a purely abstract (re-)construction of Fell's abstract induction process.
Given a Banach *-algebraic bundle $\cB$ over $G$ and $H \sbgp G,$ we show how to recover the induction process $T\rightsquigarrow \Ind_H^\cB(T)$ out of that of the bundle C*-completion $\cC$ of $\cB.$
We also define the $H-$cross sectional C*-algebra of $\cB,$ $C^*_H(\cB),$ and construct a quotient $q\colon C^*(\cB)\to C^*_H(\cB)$ in such a way that given any non degenerate *-representation $\pi\colon C^*_H(\cB)\to \bB(X),$ the disintegrated form of $\pi\circ q$ is weakly contained in a *-representation induced from $\cB_H.$
Finally, we show that $C^*_H(\cB)$ is C*-isomorphic to $C^*_H(\cC).$

In Section~\ref{sec:abstract induction revisited} we construct, given a Fell bundle $\cB$ over $G$ and $H \sbgp G,$ the (full) right $C^*(\cB_H)-$Hilbert module $L^2_H(\cB)$ and the *-representation $\regrep{H\cB}$ of \eqref{equ:firt mention to regrep}.
Our exposition will reveal that, given any *-representation $T\colon \cB_H\to \bB(X),$  $\Ind_H^\cB(T)$ is exactly
\begin{align*}
\regrep{H\cB}\otimes_{\intform{T}}1 & \colon \cB\to \bB(L^2_H(\cB)\otimes_{\intform{T}} X)&  b & \mapsto \regrep{H\cB}_b\otimes_{\intform{ T}}1,
\end{align*}
where the $T$ in the subindex $\otimes_T$ refers to the integrated form of $T.$
After this, we won't find much trouble in proving that $C^*_H(\cB)$ is isomorphic $\intregrep{H\cB}(C^*(\cB)).$
We then turn to construct a kind of \textit{universal system of imprimitivity} $\langle \regrep{H\cB},\psi^{H\cB}\rangle$ for $\cB$ over $G/H$ which we use to describe Fell's induced systems of imprimitivity in an abstract way.
This will be important later, when we try to answer to Question~\ref{q:system of imprimitivity weakly contained}.

The final part of Section~\ref{sec:abstract induction revisited} is dedicated to (abstract) induction in stages; which we may state by saying that given a Fell bundle $\cB$ over $G$ and subgroups $H \sbgp K \sbgp G,$ $\regrep{H\cB}$ is unitary equivalent to $\Ind_K^\cB(\regrep{H\cB_K}):=\regrep{K\cB }\otimes_{\intregrep{H\cB_K}}1.$
As we shall see, this implies the existence of a unique quotient map $q^\cB_{KH}\colon C^*_K(\cB)\to C^*_H(\cB)$ such that $q^\cB_{KH}\circ \intregrep{K\cB}=\intregrep{H\cB}$ ($q^\cB_H:=q^\cB_{GH}$ equals $\intregrep{H\cB}$).

We will say $\cB$ has the $HK-$WCP if $q^\cB_{KH}$ is faithful.
The $H-$WCP is the $HG-$WCP and the WCP for short is the $\{e\}G-$WCP.
Induction in stages implies (right away) that $\cB$ has the $HK-$WCP whenever $\cB_K$ has the $H-$WCP. 
So we ask the converse.

\begin{question}\label{q:HK wcp of B implies H wcp of BK}
   Does $\cB_K$ have the $H-$WCP provided that $\cB$ has the $HK-$WCP?
\end{question}

Only after a considerable amount of work, and restricting ourselves to the case where both $H$ and $K$ are normal in $G,$ will we be able to give an affirmative answer (Corollary~\ref{cor:qBKH iso iff qBKH iso}).

The main result of Section~\ref{sec:absorption principle} is our best answer to Question~\ref{q:absorption principle} and it is, basically, an absorption principle that (weakly) generalises those of Fell, Exel and Ng.
For a *-representation $T\colon \cB\to \bB(X)$ of a Fell bundle over $G$ and a unitary representation $U\colon H\to \bB(Y)$ of $H \sbgp G,$ our principle implies that $T\otimes\Ind_H^G(U)$ is weakly equivalent to a family of *-representations induced from the reductions of  $\cB$ to the conjugated groups $\{tHt^{-1}\colon t\in G\}.$
In particular, 
\begin{equation*}
\|(T\otimes \Ind_H^G(U))_f\|\leq \sup_{t\in G}\|\regrep{tH\tmu\cB}_f\|
\end{equation*}
for all $f\in C^*(\cB).$

We prove that if $\cB$ is saturated or has unitary multiplier of order $t,$ then there exists a natural C*-isomorphism between $C^*_H(\cB)$ and $C^*_{tHt^{-1}}(\cB)$ that transforms $\regrep{H\cB}$ into  $\regrep{tHt^{-1}\cB},$ as it is the case if $H$ is normal.
This isomorphism is used in conjunction with our absorption principle to give an incomplete answer to Question~\ref{q:G has the HWCP then B has the WCP} at the end of the Section~\ref{sec:absorption principle} (see Corollary~\ref{cor:icomplete answer to question HWC of group passes to bundle}).

In Section~\ref{sec:weak containment with respect to subgroups} we give conditions  on the subgroups $H \sbgp K \sbgp G$ for (the integrated form) 
\begin{equation}\label{equ:lambda HB_BK}
\intregrep{H\cB}|_{\cB_K}\colon C^*(\cB_K)\to \bB(C^*_H(\cB))\subset \bB(L^2_H(\cB))
\end{equation}
to factor (via $q^{\cB_K}_H$)
through a faithful representation of $C^*_H(\cB_K).$
We use this fact to give an incomplete answer to question~\ref{q:HK wcp of B implies H wcp of BK}).

Interchanging $H$ and $K$ in \eqref{equ:lambda HB_BK} we get the (integrated) *-representation
\begin{equation*}
\intregrep{K\cB}|_{\cB_H}\colon C^*(\cB_H)\to \bB(C^*_K(\cB))\subset \bB(L^2_K(\cB));
\end{equation*}
which is faithful if $H$ is either open, normal or has open normaliser in $G.$
These conditions are combined with the saturation of $\cB$ in Corollary~\ref{cor:amenability of G implies that of B} to give two situations where one can show the $HK-$WCP of $G$ to implies that of $\cB$ (see Question~\ref{q:G has the HWCP then B has the WCP}).

In the last part of the sixth section we study the $E-$crossed products of \cite{KaLAQu2013} in the context of Fell bundles.
Given an ideal $E$ of the Forier-Stieljes algebra $B(G)$ of a group $G,$ we define an $E-$cross sectional C*-algebra $C^*_E(\cB)$ for every Fell bundle $\cB$ over $G.$
We are particularly interested in the case where $E$ is the ideal $E_H$ determined by the kernel of $\intregrep{H}\colon C^*(G)\to C^*_H(G).$
We show that if either $H$ is normal in $G$ or $\cB$ is saturated, then $C^*_{E_H}(\cB)=C^*_H(\cB).$

Section~\ref{sec:induction and Morita} is dedicated to study the interplay between Fell's induction process and the induction of *-representations via the equivalence bundles of \cite{AbFrrEquivalence,abadie2019morita}.
The conclusions we get are used to give a partial answer to Question~\ref{q:system of imprimitivity weakly contained}.

In Section~\ref{sec:completions} we use the induction process to give a criteria to determine which C*-completions of the unit fibre $B_e$ of a Banach *-algebraic bundle $\cB=\{B_t\}_{t\in G}$ are suitable to construct a C*-completion of $\cB.$
We then take a Fell bundle $\cB$ over $G$ and normal subgroups of $G,$ $H \sbgp N \sbgp G,$ and consider the $L^1-$cross sectional bundle over $G/N$ derived from $\cB,$ $\quotient{L^1}{\cB}{N},$ which is a Banach *-algebraic bundle with unit fibre $L^1(\cB_N).$
The C*-completion $C^*_H(\cB_N)$ of the fibre $L^1(\cB_N)$ can be used to construct a C*-completion $\quotient{\cC^*_H}{\cB}{N}$ of $\quotient{L^1}{\cB}{N}.$
We show that $\cB$ has the ($H-$)WCP if and only if $\quotient{\cC^*}{\cB}{N}$ has the WCP and $\cB_N$ has the ($H-$)WCP.
This resembles (and may be used to show) the well known fact that $G$ is amenable if and only if both $G/N$ and $N$ are amenable.

We realize that $C^*(\cB)=C^*(\quotient{\cC^*}{\cB}{N})$ and show that $\intregrep{N\cB}\colon C^*(\cB)\to C^*_N(\cB)$ is (up to C*-conjugations)  $\regrepEN{\quotient{\cC^*}{\cB}{N}}\colon C^*(\quotient{\cC^*}{\cB}{N})\to C^*_\red(\quotient{\cC^*}{\cB}{N}).$
Then we use our partial answer to Question~\ref{q:inner amenability and nuclearity} to show that $\cB$ has the $N-$WCP whenever $G/N$ is inner amenable and $C^*_N(\cB)$ is nuclear.

All along the article we have included some results we consider to be interesting and illustrative of how to use our main conclusions. 
For example, in Theorem~\ref{thm:contiditional expectation open subgroup} we show that for every open subgroup $H$ of $G$ and every Fell bundle $\cB$ over $G,$ $C^*_\red(\cB_H)$ is a C*-subalgebra of $C^*_\red(\cB).$
Additionally, we construct a conditional expectation $P_\red\colon C^*_\red(\cB)\to C^*_\red(\cB_H).$
Latter, in Proposition~\ref{prop:inner amenable subgroup and downward WCP}, we show that if $H$ is (in addition) inner amenable and $C^*_\red(\cB)$ is nuclear, then $\cB_H$ has the WCP and $C^*_\red(\cB)=C^*_H(\cB).$

\section{Representations on Hilbert modules}\label{sec:rep on Hilbert modules}

As mentioned in the Introduction, we treat Hilbert spaces as right $\bC-$Hilbert modules and all our inner products are linear in the second variable.
We take from \cite{FlDr88} all the representation theory of groups and Banach *-algebraic bundles on Hilbert spaces, including the definition of the regional topology and the notion of weak containment of *-representations.
Putting together several results and definitions from sections VII, VIII 21 and XI 8.17-8.21 of \cite{FlDr88} one gets the following.

\begin{theorem}\label{thm:foundamental facts on weak containment}
	Let $\cA$ be a Banach *-algebraic bundle over a LCH group $G$ and consider a non degenerate *-representation $S\colon \cA\to \bB(X)$ and a family $\cT$ of non degenerate *-representations of $\cA.$
	Then the following claims are equivalent:
	\begin{enumerate}
		\item $S$ is weakly contained in $\cT$ ($S\preceq \cT$).
		By definition this means that given $f_1,\ldots,f_n\in L^1(\cA),$ $\varepsilon>0$ and $\xi_1,\ldots,\xi_m\in X;$ there exists a *-representation $T\colon \cA\to \bB(Y)$ which is a direct sum of finitely many elements of $\cT$ (with repetitions allowed)  and vectors $\eta_1,\ldots,\eta_m\in Y$ such that 
		\begin{align*}
		|\langle \xi_j,\xi_k\rangle -\langle \eta_j,\eta_k\rangle|& <\varepsilon & |\langle \xi_j,\intform{S}_{f_i}\xi_k\rangle -\langle \eta_j,\intform{T}_{f_i}\eta_k\rangle| & <\varepsilon
		\end{align*}
		for all $i=1,\ldots,n$ and $j,k=1,\ldots,m.$
		\item For every compact set $C\subset \cA,$ $\varepsilon>0$ and $\xi_1,\ldots,\xi_n\in X$  there exists a *-representation $T\colon \cA\to \bB(Y)$ which is a direct sum of finitely many elements of $\cT$ (with repetitions allowed)  and vectors $\eta_1,\ldots,\eta_m\in Y$ such that
		\begin{align*}
		|\langle \xi_j,\xi_k\rangle -\langle \eta_j,\eta_k\rangle|& <\varepsilon & |\langle \xi_j,S_{b}\xi_k\rangle -\langle \eta_j,T_{b}\eta_k\rangle|& <\varepsilon
		\end{align*}
		for all $b\in C$ and $j,k=1,\ldots,m.$
		\item The integrated form $\widetilde{S}$ is weakly contained in $\widetilde{\cT}:=\{\widetilde{T}\colon T\in \cT\}.$
		\item $\|\intform{S}_f\|\leq \sup\{ \|\intform{T}_f\|\colon T\in \cT\}$ for all $f\in L^1(\cA).$
	\end{enumerate}
\end{theorem}

Recall that two families of *-representations of $\cB,$ $\mathcal S$ and $\mathcal T,$ are said to be weakly equivalent ($\mathcal S \sim \mathcal T$) if $\mathcal S \preceq \mathcal T \preceq \mathcal S.$
We use the symbol $\approx$ to denote unitary equivalence of *-representations and we replace it with $=$ wherever we want to identify two unitary equivalent representations.

We will content ourselves by having an integration and disintegration theory for *-representations of Fell bundles  on Hilbert modules, and we elude any attempt of giving a notion of weak containment for such objects.

The C*-algebra of bounded adjointable operators on a right $A-$Hilbert module $Y_A$ will be denoted $\bB(Y_A).$
In particular, $\bB(A_A)\equiv \bB(A)$ is the multiplier algebra of $A.$
The minimal tensor product of C*-algebras will be denoted $\otimes$  and the maximal one $\otimes_{\max}.$

\begin{definition}\label{defi:representations of banac algebraic bundles}
	A *-representation of a Banach *-algebraic bundle $\cB=\{B_t\}_{t\in G}$ on the right $A-$Hilbert module $Y_A$ is a function $T\colon \cB\to \bB(Y_A)$ which is linear when restricted to any fibre; is multiplicative ($T_{ab}=T_aT_b$); preserves the involution ($T_{a^*}={T_a}^*$) and, for all $\xi,\eta\in Y_A,$ the map $\cB\to A,$ $b\mapsto \langle T_b\xi,\eta \rangle,$ is continuous.
	We say $T$ is non degenerate if the essential space of $T,$ defined as $Y^T_A:=\cspn \{T_b\xi\colon b\in \cB,\ \xi\in Y_A\},$ equals $Y_A.$
	A vector $\xi\in Y_A$ is cyclic for $T$ if $Y_A=\cspn\{T_b\xi\colon b\in \cB\}.$
\end{definition}

\begin{remark}
	If $\cB$ has an approximate identity in the sense of \cite[VIII 2.10]{FlDr88} (like it is the case if $\cB$ is a Fell bundle), by Cohen-Hewitt's Theorem \cite[V 9]{FlDr88} we have $Y^T_A=\{T_b\xi\colon b\in B_e,\ \xi\in Y_A\}.$
	In this case we define the essential part of $T$ as the *-representation $T'\colon \cB\to \bB(Y^T_A)$ such that $T'_b\xi=T_b\xi$ for all $b\in \cB$ and $\xi\in Y_A.$
	Note $T'$ is non degenerate. 
\end{remark}

The tensor product representations of the example below are the Hilbert module version of those constructed in \cite[VIII 9.16]{FlDr88} for Hilbert spaces.

\begin{example}
  Let $\cB$ be a Banach *-algebaric bundle over $G,$ $T\colon \cB\to \bB(X_A)$ a *-representation and $U\colon G\to \bB(Y_B)$ a unitary representation.
  The minimal (Hilbert module) tensor product $X_A\otimes Y_B$ is a right $A\otimes B-$Hilbert module in a natural way.
  Thinking of $X_A$ and $Y_B$ as if it were Hilbert spaces, it is not hard to show (following \cite[VIII 9.16]{FlDr88}) that the map $T\otimes U\colon \cB\to \bB(X_A\otimes Y_B)$ sending $b\in B_t $ to $T_b\otimes U_t$ is a *-representation which is non degenerate if $T$ is so.
The expressions $U\otimes T$ and $T\otimes U$ will be used indistinctly.
\end{example}

When we say $T\colon \cB\to \bB(Y)$ is a *-representation and the space $Y$ is not clear from the context, we will be assuming  that $Y=Y_\bC$ is a Hilbert space.
Otherwise $Y$ may be a Hilbert module or a C*-algebra (considered as a Hilbert module).

Given any *-representation $T\colon \cB\to \bB(Y_A),$ the restriction $T|_{B_e}\colon B_e\to \bB(Y_A)$ is contractive. 
Then, for all $b\in \cB,$ $ \|T_b\| = \|  T_{b^*b}\|^{1/2} \leq \|b^*b\|^{1/2} \leq  \|b\|.$

Note also that given any $\xi\in Y_A$ we have
\begin{equation*}
\lim_{a\to b} \|T_a\xi-T_b\xi\|^2 = \lim_{a\to b} \|\langle T_{a^*a}\xi,\xi\rangle + \langle T_{b^*b}\xi,\xi\rangle -\langle T_{b^*a} \xi,\xi\rangle -\langle T_{a^*b}\xi,\xi\rangle\| =0.
\end{equation*}
Thus $\cB\to Y_A,\ b\mapsto T_b\xi,$ is continuous.

\begin{proposition}\label{prop:integration of *-representation}
	For every *-representation $T\colon \cB\to \bB(Y_A),$ of a Banach *-algebraic bundle $\cB=\{B_t\}_{t\in G},$ there exists a unique *-representation $\psi \colon L^1(\cB)\to \bB(Y_A)$ such that $ \psi_f \xi =\int_G T_{f(t)}\xi\, dt $ for all $f\in C_c(\cB)$ and $\eta\in Y_A.$ 
	Moreover, 
	\begin{enumerate}
		\item  The essential spaces of $T$ and $\psi$ agree.
		\item $\xi\in Y_A$ is cyclic for $T$ if and only if it is cyclic for $\psi.$
		\item\label{item:recover integrated form} If given $r\in G,$ $b\in B_r$ and $f\in C_c(\cB)$ we define $bf,fb\in C_c(\cB)$ by $(bf)(s)=bf(r^{-1}s)$ and $(fb)(s)=f(sr^{-1})b\Delta_G(r)^{-1},$ then  $T_b\psi_f=\psi_{bf}$ and $\psi_fT_b=\psi_{fb}.$
	\end{enumerate}
\end{proposition}
\begin{proof}
	The comments preceding the statement imply $T$ is a Banach representation in the sense of \cite[VIII 8.2]{FlDr88}, then it is integrable (in the sense of \cite[VIII 11.2]{FlDr88}) to a representation of $\psi^0$ of $C_c(\cB)$ by \cite[VIII 11.3]{FlDr88}.
	The equality $\psi^0_f\xi=\int_G T_{f(t)}\xi\, dt$ follows from the construction of integrated forms of Fr\'echet representations \cite[VIII 11.3]{FlDr88}. 
	The claims about non degeneracy and cyclic vectors follow from \cite[VIII 11.10 \& 11.11]{FlDr88}.
	
	A straightforward adaptation of the arguments of \cite[VIII 11.4 \& 11.6]{FlDr88} reveal that $\psi^0$ is a $\|\ \|_1-$contractive *-representation of $C_c(\cB),$ so it admits a unique continuous extension $\psi\colon L^1(\cB)\to \bB(Y_A)$ that turns out to be a *-representation.
	Finally, claim (\ref{item:recover integrated form}) is just \cite[VIII 12.6]{FlDr88}. 
\end{proof}

\begin{definition}\label{defi:integrated form}
	The integrated form of the *-representation $T\colon \cB\to \bB(Y_A)$ is the *-representation $\psi\colon L^1(\cB)\to \bB(Y_A)$ given by the Proposition above.
	We use the same symbol to denote a *-representation and its integrated form, so we may write $T_f$ for $f\in L^1(\cB).$	
	If $\cB$ is a Fell bundle, we also denote $\intform{T}$ the unique extension of $\psi$ to the enveloping C*-algebra $C^*(\cB)$ of $L^1(\cB).$
\end{definition}

\begin{remark}[On the notation used for integrated forms of tensor product representations]\label{rem:tensor product of *-representations}
	If $T\colon \cB\to \bB(Y_A)$ and $\pi\colon A\to \bB(Z_C)$ are *-representations and we form the $\pi-$balanced tensor product $Y_A\otimes_\pi Z_C,$ then there exists a unique *-representation $T\otimes_\pi 1 \colon \cB\to \bB(Y_A\otimes_\pi Z_C)$ such that $(T\otimes_\pi 1)_a (\xi\otimes_\pi \eta)=(T_a \xi)\otimes_\pi \eta.$
	Let $\psi$ and $\psi'$ be the integrated forms of $T$ and $T\otimes_\pi 1,$ respectively.
	For every $f\in C_c(\cB),$ $\xi\in Y_A$ and $\eta\in Z_C$ one can use the fact that $Y_A\to Y_A\otimes_\pi Z_C,$ $\zeta\mapsto \zeta\otimes_\pi \eta,$ is linear and bounded to deduce that 
	\begin{equation*}
	(\psi_f\xi)\otimes_\pi \eta = (\int_G T_{f(t)}\xi\, dt)\otimes_\pi \eta = \int_G T_{f(t)}\xi\otimes_\pi \eta\, dt
	=\psi'_f(\xi\otimes_\pi \eta).
	\end{equation*}
	Hence $\psi\otimes_\pi 1=\psi'$ or $T\otimes_\pi 1=T\otimes_\pi 1$ because we use the same expression to denote a *-representation and its (dis)integrated form. 
\end{remark}

To disintegrate *-representations of Fell bundles on Hilbert modules we adapt the ideas of \cite[VIII 13.2]{FlDr88}.
The key to do this is the following extension of \cite[VI 19.11]{FlDr88}.

\begin{proposition}\label{prop:extension form ideal}
	Let $B$ be a Banach *-algebra, $I$ a (not necessarily closed) *-ideal of $B,$ $Y_A$ a Hilbert module and $\pi\colon I\to \bB(Y_A)$ a morphism of *-algebras.
	Then $\pi$ is contractive with respect to the norm of $I$ inherited from $B.$
	If $\pi$ is non degenerate, that is to say $Y_A=\cspn \{\pi(b)\xi\colon b\in I,\ \xi\in Y_A\},$ then it admits a unique extension $\pi'$ to a *-representation of $B.$
	In case $A=\bC$ and $\pi$ is degenerate, it also admits an extension to a *-representation of $B.$ 
\end{proposition}
\begin{proof}
	Given $a\in I,$ $\xi\in Y_A$ and a state $\varphi$ of $A$ define $p\colon B\to \bC$ by $p(b)=\varphi(\langle \xi,\pi(a^*ba)\xi\rangle).$
	Then $p$ is positive in the sense of \cite[VI 18]{FlDr88} and, by \cite[VI 18.14]{FlDr88}, it satisfies $p(b^*cb)\leq \|c\|p(b^*b)$ for all $c\in B$ and $b\in I.$
	Thus we obtain the following inequality for all $c,b\in I$
	\begin{equation}\label{equ:inequality for state and ideal}
	\varphi(\langle \pi(b)\pi(a)\xi,\pi(c)\pi(b)\pi(a)\xi\rangle )\leq \|c\|\varphi(\langle \pi(b)\pi(a)\xi,\pi(b)\pi(a)\xi\rangle ).
	\end{equation}
	
	The closure of $\pi(I)$ is a C*-subalgebra of $\bB(Y_A),$ so there exists a net $\{b_i\}_{j\in J}\subset I$ of self adjoint elements such that $\{\pi(b_i)\}_{i\in I}$ is bounded and $\lim_i \pi(c)\pi(b_i)=\lim_i \pi(b_i)\pi(c)=\pi(c)$ for all $c\in I.$
	Putting $c=aa^*$ and $b=b_i$ in \eqref{equ:inequality for state and ideal} and taking limit we obtain
	\begin{equation*}
	\|aa^*\|\pi(aa^*)-\pi(aa^*)^2\geq 0 \ \Rightarrow\ \|aa^*\|\|\pi(a)\|^2\geq \|\pi(a)\|^4 \ \Rightarrow\ \|a\|\geq \|\pi(a)\|
	\end{equation*}
	for all $a\in I.$
	Thus $\pi$ is a contraction.
	
	Assume $\pi$ is non degenerate.
	In such a case the uniqueness of $\pi'$ is immediate and, as in \cite[VI 19.11]{FlDr88}, its existence is equivalent to the fact that given $a\in B,$ $b_1,\ldots, b_n\in I$ and $\xi_i,\ldots,\xi_n\in Y_A$  it follows that
	\begin{equation}\label{equ:bound of extended *-representation}
	\| \sum_{j=1}^n\pi(ab_j)\xi_j \|\leq \|a\|\|\sum_{j=1}^n\pi(b_j)\xi_j\|.
	\end{equation}
	
	To prove the identity above take a faithful and non degenerate *-representation $\rho\colon A\to \bB(Z),$ form the tensor product $Y\otimes_\rho Z;$ and consider the *-representation $\pi\otimes_\rho 1\colon I\to \bB(Y\otimes_\rho Z)$ such that $(\pi\otimes_\rho 1)_b (\xi\otimes \xi)=\pi(b)\xi\otimes \xi.$
	Then $\pi\otimes_\rho 1$ is non degenerate and admits a unique extension $(\pi\otimes_\rho 1)'\colon B\to \bB(Y\otimes_\rho Z)$ by \cite[VI 19.11]{FlDr88}.
	
	If given $a\in B,$ $b_1,\ldots, b_n\in I,$ $\xi_i,\ldots,\xi_n\in Y_A$ and $\xi\in Z$ we set $u:=\sum_{j=1}^n\pi(b_j)\xi_j$  and $v:=\sum_{j=1}^n\pi(ab_j)\xi_j,$ then we have
	\begin{align*}
	\langle \xi, \rho(\langle v,v\rangle_A )\xi\rangle 
	& = \langle v\otimes \xi,v\otimes\xi\rangle
	= \| (\pi\otimes_\rho 1)'_a(u\otimes \xi) \|^2
	\leq \|a\|^2\| u\otimes \xi \|^2 \\
	& \leq \|a\|^2 \langle \xi,\rho(\langle u,u\rangle_A)\xi\rangle.    
	\end{align*}
	This implies $\langle v,v\rangle_A \leq \|a\|^2\langle u,u\rangle_A,$ which gives \eqref{equ:bound of extended *-representation} afer taking norms and square roots.
	
	In case $A=\bC$ and $\pi$ is non degenerate we may consider the essential space $Y_\pi:=\cspn \pi(I)Y$ of $\pi$ and consider the extension $\pi'\colon B\to \bB(Y_\pi)\subset \bB(Y)$ of $\pi\colon I\to \bB(Y_\pi)$ as a *-representation of $B$ on $\bB(Y).$
\end{proof}

\begin{proposition}\label{prop:disintegration of *-representations}
	If $\cB$ be a Fell bundle, then every non degenerate *-representation $\pi\colon L^1(\cB)\to \bB(Y_A)$ is the integrated form of a unique *-representation $T\colon \cB\to \bB(Y_A).$
\end{proposition}
\begin{proof}
	To prove the existence follow the ideas of \cite[VIII 13.2]{FlDr88} noticing that, by Proposition~\ref{prop:extension form ideal}, $\pi$ can be extended to the bounded multiplier algebra of $L^1(\cB).$
	Uniqueness follows from \cite[VIII 11.22]{FlDr88}.
\end{proof}

Let $\cB$ be a Banach *-algebraic bundle and denote $\iota\colon L^1(\cB)\to C^*(\cB)$ the canonical map form $L^1(\cB)$ to its enveloping C*-algebra.
Take any non degenerate *-representation $\pi\colon C^*(\cB)\to \bB(X).$
Then $\pi\circ \iota$ is the integrated form of a unique *-representation $T\colon \cB\to \bB(X)$ and $\pi(C^*(\cB))$ is the closure of $\intform{T}(C_c(\cB)).$
By Proposition~\ref{prop:integration of *-representation}, $T(\cB)$ is contained in the multiplier algebra of $\pi(C^*(\cB))$ and for each $b\in \cB$ the operator $T_b$ is the unique such that $T_b\pi(\iota(f))=\pi(\iota(bf))$ and $\pi(\iota(f))T_b=\pi(\iota(fb))$ for all $f\in C_c(\cB).$

If $\pi$ is chosen to be faithful and non degenerate and we identify $C^*(\cB)$ with the image of $\pi,$ using \cite[12.4-5]{FlDr88} one can view $T$ as a *-representation $\regrep{\cB}\colon \cB\to \bB(C^*(\cB))$ such that
\begin{align*}
\regrep{\cB}_b\iota(f)& =\iota(bf) & \iota(f)\regrep{\cB}_b& =\iota(fb) &  \intform{\Lambda}^\cB_f(\iota(g))&=\iota(f*g).
\end{align*}
for all $b\in \cB$ and $f,g\in C_c(\cB).$
In particular, $\intform{\Lambda}^\cB\colon L^1(\cB)\to \bB(C^*(\cB))$ is the composition of $\iota$ with the natural inclusion $C^*(\cB)\subset \bB(C^*(\cB)).$

\begin{definition}
	The universal representation of $\cB$ is $\regrep{\cB}\colon \cB\to \bB(C^*(\cB)).$
\end{definition}

Recall from Definition~\ref{defi:integrated form} that for every Fell bundle $\cB$ we consider the integrated form $\intform{\Lambda}^\cB$ to be defined in $C^*(\cB)$ and, consequently, it is precisely the natural inclusion $C^*(\cB)\subset \bB(C^*(\cB)).$

The discussion preceding the last definition above can now be used to get the following.

\begin{remark}
Consider a non degenerate *-representation $\pi\colon C^*(\cB)\to \bB(X)$ and its canonical extension $\overline{\pi}\colon \bB(C^*(\cB))\to \bB(X).$
Then $\overline{\pi}\circ \regrep{\cB}$ is a *-representation whose integrated form is $\pi.$
Moreover, the claim holds true if $\cB$ is a Fell bundle and $X$ is any Hilbert module.
\end{remark}

\begin{example}
  Consider a Fell bundle $\cB=\{B_t\}_{t\in G}$ and the $\intregrepEN{\cB}\colon C^*(\cB)\to \bB(L^2_e(\cB))$ of \cite[Definition 2.7]{ExNg}.
  We want to identify the disintegrated form $\overline{\intregrepEN{\cB}}\circ \regrep{\cB}$ of $\regrepEN{\cB}$ (which we denote $\regrepEN{\cB}$).
   Recall that $C_c(\cB)$ is a dense subspace of the $B_e-$Hilbert module $L^2_e(\cB)$ and
  \begin{align*}
  \langle f,g\rangle & =\int_G f(t)^*g(t)\, dt &   (fb)(s)&=f(s)b
  \end{align*}
  for all $f,g\in C_c(\cB),$ $b\in B_e$ and $s\in G.$
  By construction, $\intregrepEN{\cB}(f)g=f*g$ for all $f,g\in C_c(\cB).$
  
  The inclusion map from $C_c(\cB)$ to $L^2_e(\cB)$ is continuous in the inductive limit topology because $\|f\|_{L^2_e(\cB)}\leq \int_G\|f(t)\|^2\, dt$ for all $f\in C_c(\cB).$
  Besides, the proof of \cite[VIII 5.11]{FlDr88} reveals that $P:=\{f*g\colon f,g\in C_c(\cB)\}$ is dense in $C_c(\cB)$ with respect to the inductive limit topology, and so it is dense in $L^2_e(\cB).$
  Note this implies $\intregrepEN{\cB}$ is non degenerate.
  
  For all $b\in \cB$ and $f,g\in C_c(\cB),$
   \begin{equation*}
b(f*g)=(bf)*g=\intregrepEN{\cB}_{bf}g=\intregrepEN{\cB}_{\regrep{\cB}_b f}g = \overline{\regrepEN{\cB}}_b \regrepEN{\cB}_f g \equiv \regrepEN{\cB}_b\regrepEN{\cB}_f g.
   \end{equation*}
   Setting $h:=f*g$ we get $\| bh \|_{L^2_e(\cB)}\leq \|b\|\|h\|_{L^2_e(\cB)} $ and the inequality $\| bh \|_{L^2_e(\cB)}\leq\|b\|\|h\|_{L^2_e(\cB)}$ must hold for all $h\in C_c(\cB)$ because $P$ is dense in $L^2_e(\cB).$
   Then $\lambda^\cB_b\in \bB(L^2_e(\cB))$ is the unique operator mapping $f\in C_c(\cB)$ to $bf.$  
\end{example}

In Section 3 of \cite{Ab03} the construction of $\regrepEN{\cB}$ is performed the other way around (the non integrated form is constructed first).

The example above serves as a motivation for our definition of the Hilbert modules $L^2_H(\cB)$ associated to subgroups $H \sbgp G,$ and also for our construction of the *-representation $\regrep{H\cB}\colon \cB\to \bB(L^2_H(\cB))$ and its integrated form $\intregrep{H\cB}\colon C^*(\cB)\to \bB(L^2_H(\cB)).$

\subsection{Application: inner amenability and weak containment}\label{ssec:inner amenability and weak containment}

Here we address the first part of Question~\ref{q:inner amenability and nuclearity}, the answer being Corollary~\ref{cor:inner amenability and nuclearity}.

All we need about inner amenable groups can be found in \cite{lau1991inner,paterson2000amenability,YuanInnerInvariantMeans}.
The definition of inner amenable group we adopt is that of \cite[Chapter 2, Section 2.35 (H)]{paterson2000amenability}, so our inner amenable groups are the [IA] groups of \cite{YuanInnerInvariantMeans}.

The left and right regular representations of a LCH group $G$ in $\bB(L^2(G))$ will be denoted $\regrepEN{}$ and $\rho,$ respectively.
They are given by $\regrepEN{}_t(f)(s)=f(t^{-1}s)$ and $\rho_t(f)(s)=\Delta(t)^{1/2}f(st),$ where $\Delta$ is the modular funtion of $G.$
These representations are conjugated by the unitary $U\in \bB(L^2(G)),$ $U(f)(s)=\Delta(s)^{-1/2}f(s^{-1}).$
Since the ranges of $\lambda$ and $\rho$ commute, the function $\omega\colon G\to \bB(L^2(G)),$  $\omega_t = \regrepEN{}_t\rho_t,$ is a *-representation.

It is a well known fact that $G$ is amenable if and only if the trivial representation $\kappa_G$ is weakly contained in $\regrepEN{}.$
Similarly, $G$ is inner amenable $\Leftrightarrow$ $\kappa_G\preceq \omega.$
We will only use the direct implication of this last claim, which is a straightforward consequence of \cite[Theorem 1]{YuanInnerInvariantMeans}.

For Fell bundles over discrete groups the result below can be found in \cite{ara2013dynamical}.

\begin{theorem}\label{thm:full cross sectional algebra an inner amenable groups}
 Let $\cB=\{B_t\}_{t\in G}$ be a Fell bundle over an inner amenable group.
 Then the integrated form of
 \begin{align*}
  \psi& \colon \cB\to \bB(C^*_\red(G)\omax C^*_\red(\cB))& (b\in B_s)&\mapsto \psi_b:= \regrepEN{}_s\otimes_{\max}  \regrepEN{\cB}_b
 \end{align*}
 is faithful and non degenerate.
\begin{proof}
 Take a non degenerate *-representation $T\colon \cB\to \bB(X)$ with faithful integrated form.
 By Theorem~\ref{thm:EXNG absorption principle} there exists a faithful *-representation $\Omega^\cB$ making 
$$ \xymatrix{ C^*(\cB) \ar[dr]_{\intregrepEN{\cB}}\ar[rr]^{\regrepEN{}\otimes T}  & & \bB(L^2(G)\otimes X)\\ 
	&C^*_\red(\cB)\ar[ur]_{\Omega^\cB} & }$$
a commutative diagram.
Consequently, if $\overline{\Omega}^\cB$ is the extension of $\Omega^\cB$ to $\bB(C^*_\red(\cB))\subset \bB(L^2_e(\cB)),$ then $\overline{\Omega}^\cB\circ \regrepEN{\cB}=\regrepEN{}\otimes T$ (for integrated and non integrated forms).
 
 As pointed out earlier, the right regular representation of $G$ is unitary equivalent to the left one.
 Hence, there exists a non degenerate *-representation 
 \begin{equation*}
  \Omega^G\colon C^*_\red(G)\to \bB(L^2(G)\otimes X)
 \end{equation*}
 whose extension $\overline{\Omega}^G\colon \bB(C^*_\red(G))\to \bB(L^2(G)\otimes X)$ satisfies $\overline{\Omega}^G\circ \regrepEN{}=\rho\otimes 1.$
 
 A simple computation shows that the ranges of $\Omega^\cB$ and $\Omega^G$ commute, thus there exists a unique *-representation
 \begin{equation*}
  \Omega^\otimes :=\Omega^G\otimes_{\max}\Omega^\cB\colon C^*_\red(G)\otimes_{\max}C^*_\red(\cB)\to \bB(L^2(G)\otimes X)
 \end{equation*}
 sending $g\otimes_{\max}f$ to $\Omega^G_g\Omega^\cB_f $ for all $f\in C_c(\cB)$ and $g\in C_c(G).$
 Moreover, if $\overline{\Omega}^\otimes$ is the extension of $\Omega^\otimes$ to the multiplier algebra, then for all $r,s\in G$ and $b\in B_s:$
 \begin{equation*}
  \overline{\Omega}^\otimes( \regrepEN{}_r\omax\regrepEN{\cB}_b ) = \rho_r\regrepEN{}_s\otimes T_b.
 \end{equation*}
 This yields $ \overline{\Omega}^\otimes\circ \psi = \omega\otimes T$ for the non integrated forms of $\psi$ and $\omega\otimes T,$ which implies the same identity holds for the integrated forms.
 Then, for all $f\in C^*(\cB),$
 \begin{equation*}
\|(\omega\intform{\otimes}T)_f\|\leq \|\intform{\psi}_f\|\leq \|\intform{T}_f\|.
 \end{equation*}
 
 By \cite[VIII 21.24]{FlDr88} and the fact that $\kappa_G\preceq \omega$ we have $T\equiv \kappa_G\otimes T\preceq \omega\otimes T.$
 Then $\|\intform{T}_f\|\leq \|(\omega\intform{\otimes}T)_f\|$ and this implies $\|\intform{\psi}_f\|=\|\intform{T}_f\|=\|f\|_{C^*(\cB)}$ for all $f\in C^*(\cB).$
\end{proof}
\end{theorem}

Once again we follow Exel's work on Fell bundles over discrete groups, in this case \cite[Proposition 18.6]{Exlibro}, to prove the following.

\begin{theorem}\label{thm:homomorphism for reduced cross sectional algebra}
 Given a Fell bundle $\cB=\{B_t\}_{t\in G},$ the  *-representation 
 \begin{align*}
  \psi^\red & \colon \cB\to \bB(C^*_\red(G)\otimes C^*_\red(\cB)) & (b\in B_s)&\mapsto \psi^\red_b:=\regrepEN{}_s\otimes \regrepEN{\cB}_b
 \end{align*}
 is non degenerate and there is a unique *-representation $\varphi$ making 
 \begin{equation*}
\xymatrix{ C^*(\cB)\ar[rr]^{\intform{ \psi^r}} \ar[rd]_{\regrepEN{\cB}} &  &\bB(C^*_\red(G)\otimes C^*_\red(\cB))\\
& C^*_\red(\cB)\ar[ur]_{\varphi} & }
 \end{equation*}
 a commutative diagram.
 Moreover, $\varphi$ is faithful.
 \begin{proof}
  Let $q\colon C^*_\red(G)\otimes_{\max} C^*_\red(\cB)\to C^*_\red(G)\otimes C^*_\red(\cB)$ be the natural quotient map and denote  $\overline{q}$ its natural extension to the multiplier algebras.
  Using the $\psi$ of Theorem~\ref{thm:full cross sectional algebra an inner amenable groups} we get the *-representation  $\psi^\red:=\overline{q}\circ \psi\colon \cB \to \bB(C^*_\red(G)\otimes C^*_\red(\cB)).$
  
  For every $b\in B_s,$ $f\in C^*_\red(G)$ and $g\in C^*_\red(\cB)$ we have
  \begin{equation*}
\psi^\red_b(f\otimes g)=q(\regrepEN{}_s f\otimes_{\max}\regrepEN{\cB}_b g)= \intregrepEN{}_{sf} \otimes\intregrepEN{\cB}_{bg}=
(\regrepEN{}\otimes \regrepEN{\cB})_b (f\otimes g).
  \end{equation*}
  It is then clear that $\psi^\red_b = \regrepEN{}_s\otimes \regrepEN{\cB}_b.$
    
  To prove (the integrated form of) $\intform{\psi^\red}$ factors via a faithful *-representation $\varphi$ of $C^*_\red(\cB)$ we take a *-representation $S\colon \cB\to \bB(X)$ of $\cB$ with $S|_{B_e}$ faithful.
  Define $Y:=L^2(G)\otimes X$ and $T:=\lambda\otimes S.$
  By Theorem~\ref{thm:EXNG absorption principle} we may identify $C^*_\red(\cB)$ with $\intform{T}(C^*(\cB))$ and write $T_b\equiv \regrepEN{\cB}_b$ and $f\equiv \intform{T}_f$ for all $b\in \cB$ and $f\in C_c(\cB).$
  
   We view $C^*_\red(G)\otimes C^*_\red(\cB)$ and its multiplier algebra as concrete C*-subalgebras of $\bB(L^2(G)\otimes Y)$ by identifying $f\otimes g$ with  $ \regrepEN{}_f\otimes \intform{T}_g$ ($f\in C_c(G)$ and $g\in C_c(\cB)$).
   Thus $\regrepEN{}_s\otimes \regrepEN{\cB}_b \equiv \regrepEN{}_s\otimes T_b,$ for all $s\in G$ and $b\in B_s.$
   This means $\psi^r$ gets identified with $\lambda\otimes T,$ Theorem~\ref{thm:EXNG absorption principle} thus gives the existence of $\varphi$ and implies $\varphi$ is faithful.
 \end{proof}
\end{theorem}

Putting together the last two Theorems we get the following consequence which, for discrete $G,$ is a combination of \cite[Theorem 3.9]{ExNg} and \cite[Theorem 25.11]{Exlibro}.

\begin{corollary}\label{cor:inner amenability and nuclearity}
 Let $\cB=\{B_t\}_{t\in G}$ be a Fell bundle.
 If $G$ is inner amenable and the canonical quotient map $q\colon C^*_\red(G)\otimes_{\max} C^*_\red(\cB)\to C^*_\red(G)\otimes C^*_\red(\cB)$ is faithful, then $\cB$ has the weak containment property.
 \begin{proof}
 	Consider the maps $\psi$ and $\varphi$ given by Theorems~\ref{thm:full cross sectional algebra an inner amenable groups} and~\ref{thm:homomorphism for reduced cross sectional algebra}.
 	The diagram
  \begin{equation*}
   \xymatrix{  C^*(\cB)\ar[rrr]^{\intform{\psi}}\ar[d]_{\regrepEN{\cB}}& & & \bB(C^*_\red(G)\omax C^*_\red(\cB)) \ar[d]^{\overline{q}}\\
   C^*_\red(\cB)\ar[rrr]_{\varphi} & & &\bB(C^*_\red(G)\otimes C^*_\red(\cB))}
  \end{equation*}
  commutes because $\varphi\circ\regrepEN{\cB}$ and $\overline{q}\circ \intform{\psi} $ are the integrated forms of $\overline{\varphi}\circ \regrepEN{\cB} = \psi^\red$ and $\overline{q}\circ\psi=\psi^\red,$ respectively.
  The hypotheses guarantee that $\intform{\psi},\overline{q}$ and $\varphi$ are injective, then $\regrepEN{\cB}$ is a C*-isomorphism.
 \end{proof}
\end{corollary}

\section{C*-completions, positivity and induction}\label{sec:Banac bundles, completions and positivity}

Consider a Banach *-algebraic bundle $\cB=\{B_t\}_{t\in G}$ and a subgroup $H \sbgp G.$
As mentioned earlier, the reduction $\cB_H=\{B_t\}_{t\in H}$ is a Banach *-algebraic bundle.

The basic ingredients to perform Fell's abstract induction process (as described in \cite{FlDr88}) are the function $p\colon  C_c(\cB)\to C_c(\cB_H)$ and the action $C_c(\cB)\times C_c(\cB_H)\to C_c(\cB),$ $(f,u)\mapsto fu,$ given by 
\begin{align}
p(f)(t)&=\Delta_G(t)^{1/2}\Delta_H(t)^{-1/2}f(t) \label{equ: p}\\
fu(t)&= \int_H f(ts)u(s^{-1})\Delta_G(s)^{1/2}\Delta_H(s)^{-1/2}\, ds \label{equ: action to induce}
\end{align}
where $\Delta_G$ and $\Delta_H$ are the modular functions of $G$ and $H,$ respectively.

When considering more than one subgroup of $G$ we sometimes write $p^G_H$ or $p_H$ instead of $p.$

\begin{remark}\label{rem:about normal subgroups and restrictions}
	The function $p$ is precisely the restriction map $f\mapsto f|_H$ if and only if $\Delta_G|_H=\Delta_H.$
	This is the case if and only if the coset space $G/H=\{tH\colon t\in G\}$ has a non-zero $G-$invariant regular Borel measure \cite[III 13.16]{FlDr88} (like it is the case if $H$ is normal or open in $G$).
\end{remark}

It is shown in \cite[XI 8.4]{FlDr88} that 
\begin{align}\label{equ:ref properties of p}
p(f)^* &= p(f^*) & p(f)*u & = p(fu) & (f*g)u &= f*(gu) & f(u*v) & = (fu)v 
\end{align}
for all $f,g\in C_c(\cB)$ and $u,v\in C_c(\cB_H).$

Even though $p$ is not (always) a conditional expectation in the usual C*-algebraic sense, one may think it is and try to replicate the induction process via conditional expectations.
This is a generalization of the approach adopted in \cite[pp 179]{Rf74} and a particular case of the (abstract) induction process for *-algebras developed by Fell in \cite[XI 4]{FlDr88}; from where we take the notion of inducible *-representation.
The situation is particularly nice if $H$ is open in $G,$ see Theorem~\ref{thm:contiditional expectation open subgroup}.

Given a *-representation $T\colon \cB_H\to \bB(X),$ the restriction of $\intform{T}\colon L^1(\cB_H)\to \bB(X)$ to $C_c(\cB_H)$ is inducible via $p$ to a *-representation of $C_c(\cB)$ if and only if 
\begin{align}
\langle \xi,\intform{T}_{p(f^**f)}\xi \rangle & \geq 0 \label{equ:positivity}\\
\|g\|_1^2\langle \xi,\intform{T}_{p(f^**f)}\xi \rangle & \geq \langle \xi,\intform{T}_{p(f^*g^*g*f)}\xi \rangle \label{equ:induced is bounded}
\end{align}
for all $\xi\in X$ and $f,g\in C_c(\cB).$

Fell calls $T$ ``$\cB-$positive'' if it satisfies \eqref{equ:positivity} and in \cite[XI 9.26]{FlDr88} he shows that $\intform{T}$ is inducible via $p$ if and only if it is $\cB-$positive.
He does this by using $\cB-$positivity to construct the so called concretely induced (non integrated) representation $\Ind_{\cB_H\uparrow \cB}(T),$ which he uses to prove \eqref{equ:induced is bounded}.

If $T$ is $\cB-$positive, the (abstractly) induced Hilbert space is the closed linear span of tensors $f\otimes_T\xi$ ($f\in C_c(\cB)$ and $\xi\in X$).
The (abstractly) induced *-representation of $C_c(\cB)$ is denoted $\Ind^p_{C_c(\cB_H)\uparrow C_c(\cB)}(T)$  and the conditions
\begin{align*}
\langle f\otimes_T\xi,g\otimes_T\eta\rangle &= \langle \xi,\intform{T}_{p(f^**g)}\eta\rangle &
\Ind^p_{C_c(\cB_H)\uparrow C_c(\cB)}(T)_f(g\otimes_T\eta)&=(f*g)\otimes_T\eta,
\end{align*}
valid for all $f,g\in C_c(\cB)$ and $\xi,\eta\in X,$
determine $\Ind^p_{C_c(\cB_H)\uparrow C_c(\cB)}(T)$ up to a unitary equivalence.

Integrating $\Ind_{\cB_H\uparrow \cB}(T)$ one gets a *-representation of $L^1(\cB)$ that, when restricted to $C_c(\cB),$ is unitary equivalent to $\Ind^p_{C_c(\cB_H)\uparrow C_c(\cB)}(T)$ \cite[XI 9.26]{FlDr88}.
Thus one may think
\begin{equation}\label{equ:intro abstract and concrete induction agree}
\intform{\Ind}_{\cB_H\uparrow \cB}(T)|_{C_c(\cB)} = \Ind^p_{C_c(\cB_H)\uparrow C_c(\cB)}(T),
\end{equation}
and view the induction processes $\Ind_{\cB_H\uparrow \cB}$ and $\Ind^p_{C_c(\cB_H)\uparrow C_c(\cB)}$ as different faces of the same coin.

The result below solves the positivity problem posed in the last Remark of \cite[XI 11.10]{FlDr88}. 
We only prove the implication  (\ref{item:easy form of B positivity})$\Rightarrow$(\ref{item:double sum of inner product positive}) because all the rest of the Theorem is contained in \cite{FlDr88}; it is reproduced here to be used as a reference.

\begin{theorem}\label{thm:characterization of positive representations}
 Let $\cB=\{B_t\}_{t\in G}$ be a Banach *-algebraic bundle, $H \sbgp  G$ and also let $T\colon \cB_H\to \bB(Y)$ be a *-representation.
 Then the following are equivalent:
 \begin{enumerate}
  \item\label{item:T is B positive} $T$ is $\cB-$positive.
  \item\label{item:double sum of inner product positive} For every coset $\alpha\in G/H,$ every positive integer $n,$ all $b_1,\dots,b_n\in \cB_\alpha,$ and all $\xi_1,\ldots,\xi_n\in Y,$ $\sum_{i,j=1}^n \langle \xi_i,T_{b_i^*b_j}\xi_j\rangle \geq 0.$
  \item\label{item:easy form of B positivity} The restriction $T|_{B_e}$ is $\cB-$positive, that is to say $\langle \xi,T_{b^*b}\xi\rangle \geq 0$ for all $b\in \cB$ and $\xi\in Y.$
 \end{enumerate}
 Besides, the three conditions above hold if $\cB$ is a Fell bundle.
 \begin{proof}
  The equivalence between (1) and (2) is part of the content of \cite[XI 8.9]{FlDr88} and (2) clearly implies (3).
  Note also that if $\cB$ is a Fell bundle, then (3) does hold because $b^*b\geq 0$ in $B_e$ for all $b\in \cB$ and the restriction $T|_{B_e}$ is a *-representation of $B_e.$
  As indicated by Fell at the end of \cite[XI 11.10]{FlDr88}, to prove our statement it is enough to show that (2) does hold under the (only) assumption of $\cB$ being a Fell bundle.
  
  Assume $\cB$ is a Fell bundle, take a coset $\alpha\in G/H,$ $b_1,\ldots,b_n\in \cB_\alpha$ and $\xi_1,\ldots,\xi_n\in Y.$
 Let $t_1,\ldots,t_n\in \alpha$ be such that $b_j\in B_{t_j}$ (for every $j=1,\ldots,n,$) and set $\mathfrak{t}:=(t_1,\ldots,t_n).$
  Define the matrix space
 \begin{equation*}
  \mathbb{M}_\mathfrak{t}(\cB):=\{ (M_{i,j})_{i,j=1}^n\colon M_{i,j}\in B_{t_i^{-1}t_j},\ \forall\ i,j=1,\ldots,n \}
  \end{equation*}
  as in \cite[Lemma 2.8]{AbFrrEquivalence}.
  Then $\mathbb{M}_\mathfrak{t}(\cB)$ is a C*-algebra with usual matrix multiplication as product and $*-$transpose as involution.
  A quick way of showing this is by taking a *-representation $S\colon \cB\to \bB(Y)$ with $S|_{B_e}$ faithful and to identify $\mathbb{M}_\mathfrak{t}(\cB)$ with the concrete C*-algebra
  \begin{equation*}
   \{ (S_{M_{i,j}})_{i,j=1}^n\colon M_{i,j}\in B_{t_i^{-1}t_j},\ \forall\ i,j=1,\ldots,n\}\subset \bB(Y^n).
  \end{equation*}
  
  The matrix $M:=(b_i^*b_j)_{i,j=1}^n$ belongs to $\mathbb{M}_\mathfrak{t}(\cB)$ and, regarding $\cB$ as a $\cB-\cB-$equivalence bundle, it follows from the proof of \cite[Lemma 2.8]{AbFrrEquivalence} that $M$ is positive in $\mathbb{M}_\mathfrak{t}(\cB).$
  If $N\in \mathbb{M}_\mathfrak{t}(\cB)$ is the positive square root of $M,$ then all the entries of $N$ belong to $\cB_H$ and
  \begin{equation*}
   \sum_{i,j=1}^n\langle \xi_i,T_{b_i^*b_j}\xi_j\rangle 
     = \sum_{i,j,k=1}^n \langle \xi_i,T_{{N_{k,i}}^*N_{k,j}}\xi_j\rangle
     = \sum_{k=1}^n \langle \sum_{i=1}^n T_{N_{k,i}}\xi_i,\sum_{j=1}^nT_{N_{k,j}}\xi_j\rangle\geq 0;
  \end{equation*}
  proving that $T$ is $\cB-$positive. 
 \end{proof}
\end{theorem}

Having solved the positivity problem we now try to give a reconstruction of Fell's abstract induction process without using the concretely induced *-representations.

\subsection{Fell's abstract induction process (reforged)}\label{ssec:abstract induction}

Fix a Banach *-algebraic bundle $\cB=\{B_t\}_{t\in G},$ a subgroup $H \sbgp G$  and a $\cB-$positive *-representation $T\colon \cB\to \bB(X).$

By \cite[XI 4.5]{FlDr88} the function $C_c(\cB)\times C_c(\cB)\to \bB(X),$ $(f,g)\mapsto \intform{T}_{p(f^**g)},$ in an operator inner product in the sense of \cite[XI 1]{FlDr88} (recall our convention on inner products).
Essentially, this means that there exists a unique pre-inner product $\langle\ , \ \rangle_T$ on the algebraic tensor product $C_c(\cB)\odot X$ such that
\begin{equation*}
   \langle f\odot \xi,g\odot \eta\rangle_T=\langle \xi,\intform{T}_{p(f^**g)}\eta\rangle.
\end{equation*}

The Hilbert space induced by $T$ (via $p$) is, up to a unitary equivalence, a space $X^p_T$ for which there exists a linear map with dense range $C_c(\cB)\odot X\to X^p_T,\ f\odot \xi\mapsto f\otimes_T \xi,$ such that $\langle f\otimes_T \xi,g\otimes_T \eta\rangle =\langle \xi,\intform{T}_{p(f^**g)}\eta\rangle.$

We want to define a *-representation $S\colon \cB\to \bB(X^p_T)$ such that $S_b(f\otimes_T\xi)=(bf)\otimes_T\xi.$
To prove the existence of this $S$ we start by recalling the following crucial fact, which appears in \cite[XI 9.10]{FlDr88}.

\begin{lemma}
	Let $\cB=\{B_t\}_{t\in G}$ be a Banach *-algebraic bundle, $H \sbgp G$ and $T\colon \cB_H\to \bB(X)$ a $\cB-$positive *-representation.
	Then for every coset $\alpha \in G/H,$ every $c\in \cB,$ all $b_1,\ldots,b_n\in \cB_{\alpha}$ and $\xi_1,\ldots,\xi_n\in X,$
	\begin{equation*}
	\sum_{i,j=1}^n\langle \xi_i,T_{b_i^*c^*cb_j}\xi_j\rangle\leq \|c\|^2 \sum_{i,j=1}^n\langle \xi_i,T_{b_i^*b_j}\xi_j\rangle.
	\end{equation*}
\end{lemma}

In the proof of \cite[XI 8.9]{FlDr88} Fell uses the $\rho-$functions of \cite[III 13.2]{FlDr88} to show the equivalence of the claims (\ref{item:T is B positive}) and (\ref{item:double sum of inner product positive})  of Theorem~\ref{thm:characterization of positive representations}.
The same idea allow us to use the lemma above to prove the following result, whose proof is left to the reader.

\begin{theorem}\label{thm:B positivity implies inducibility}
	Let $\cB=\{B_t\}_{t\in G}$ be a Banach *-algebraic bundle, $H \sbgp G$ and $T\colon \cB_H\to \bB(X)$ a $\cB-$positive *-representation.
	Then for every $c\in \cB$ and $f\in C_c(\cB)$ 
	\begin{equation*}
	\intform{T}_{p((cf)^**(cf))}\leq \|c\|^2 \intform{T}_{p(f^**f)}
	\end{equation*} 
\end{theorem} 

Using the ideas of \cite[XI 4.5]{FlDr88} (of, basically, decomposing $T$ into a direct sum of $\cB-$positive cyclic *-representations) one can prove that for all $c\in \cB,$ $f_1,\ldots,f_n\in C_c(\cB)$ and $\xi_1,\ldots,\xi_n\in X,$
\begin{equation*}
 \sum_{i,j=1}^n \langle \xi_i,\intform{T}_{p((cf_i)^**(cf_j))}\xi_j\rangle \leq \|c\|^2 \sum_{i,j=1}^n \langle \xi_i,\intform{T}_{p(f_i^**f_j)}\xi_j\rangle.
\end{equation*}

The inequality above implies (for each $c\in \cB$) the existence of a unique operator $S_c\in \bB(X^p_T)$ mapping $f\otimes_T\xi$ to $cf\otimes_T\xi.$
It is straightforward to show that $\cB\to \bB(X^p_T),\ c\mapsto S_c,$ is linear on each fibre, multiplicative and preserves adjoints.
For example, by proving that $(cf)^* * g= f^* *(c^*g)$ for all $f,g\in C_c(\cB)$ and $c\in \cB$ one gets
\begin{equation*}
  \langle S_c (f\otimes_T\xi),g\otimes_T\eta\rangle
   = \langle \xi, \intform{T}_{p((cf)^**g)}\eta\rangle = \langle \xi, \intform{T}_{p(f^**(c^*g))}\eta\rangle
   = \langle f\otimes_T\xi,S_{c^*}g\otimes_T\eta\rangle,
\end{equation*}
and deduces ${S_c}^*=S_{c^*}.$

For every $f,g\in C_c(\cB)$ the function $\cB\to C_c(\cB_H),$ $c\mapsto p((cf)^**g),$ is $\|\ \|_1-$continuous.
Hence, the function $\cB\to \bB(X^p_T)$ mapping   $c$ to $\langle S_c f\otimes_T\xi,g\otimes_T\eta\rangle$ is continuous for every $\xi,\eta\in X.$
This implies $S$ is a *-representation of $\cB.$

\begin{definition}\label{defi:abstract induction}
   The *-representation of $\cB$ (abstractly) induced by $T$ via $p$ is
   \begin{equation*}
   \Ind_H^\cB(T)\colon \cB\to \bB(X^p_T),\ c\mapsto S_c.
   \end{equation*}
   The integrated form of $\Ind_H^\cB(T)$ will be denoted $\intform{\Ind}_H^\cB(T)$ or $\widetilde{\Ind}_H^\cB(T).$   
\end{definition}

We claim that the restriction of $\intform{\Ind}_H^\cB(T)$ to $C_c(\cB)$ is (unitary equivalent to) $\Ind_{C_c(\cB_H)\uparrow C_c(\cB)}(T).$
One can deduce this fact from \cite[XI 9.25-30]{FlDr88} but it takes some time to get used to the notation used there, so we prefer to give an outline of proof here.

Fix $f,g,h\in C_c(\cB)$ and note $F\colon G\to C_c(\cB),$ $F(t):=f^** (h(t)g),$ is continuous with respect to the inductive limit topology.
Moreover, $F$ is $\|\ \|_\infty-$continuous and there exists a compact set $K\subset G$ that contains the support of $F(t)$ for all $t\in G.$

Let $C_K(\cB)$ be the set of continuous sections of  $\cB$ with support contained in $K.$
Then $F\colon G\to C_K(\cB)$ is continuous and has compact support, so it is integrable.
For every $t\in G$ the function $C_K(\cB)\to B_t,$ $u\mapsto u(t),$ is continuous and linear, then
\begin{equation*}
  \left(\int_G F(s)\, ds \right)(t) = \int_G F(s)(t) \, ds
  =\int_G \underbrace{\int_G f^*(r)h(s)g(\smu\rmu t)\,dr}_{F(s)(t)} ds = (f^**(h*g))(t).
\end{equation*}
Thus $\int_G F(s)\, ds = f^**(h*g).$

For every $\xi,\eta\in X$ the function $C_K(\cB)\to \bC,\ u\mapsto \langle \xi,\intform{T}_{p(u)}\eta\rangle,$ is linear and $\|\ \|_\infty-$continuous.
Then 
\begin{multline*}
\langle f\otimes_T\xi, h*g\otimes_T\eta\rangle =
\langle \xi,\intform{T}_{p(f^**(h*g))}\eta\rangle
 = \int_G  \langle \xi,\intform{T}_{p(f^**(h(s)g))}\eta\rangle\, ds\\
 = \int_G  \langle f\otimes_T \xi,S_{h(t)} g\otimes_T\eta\rangle\, ds
 = \langle f\otimes_T \xi,\intform{S}_{h} g\otimes_T\eta\rangle.
\end{multline*} 
This shows $\intform{S}_h g\otimes_T\eta =h*g\otimes_T\eta$ and implies 
\begin{equation*}
  \intform{\Ind}_H^\cB(T)_h = \Ind_{C_c(\cB_H)\uparrow C_c(\cB)}^p(T)_h
\end{equation*}
for all $h\in C_c(\cB).$

\begin{remark}
After the identification above, \cite[XI 9.26]{FlDr88} says $\Ind_H^\cB(T)$ is the abstract version of Fell's concretely induced *-representation $\Ind_{\cB_H\uparrow \cB}(T)$ (so they are unitary equivalent).
\end{remark}

\begin{remark}[Fell]\label{rmk:induce a restriction}
	If $S\colon \cB\to \bB(X)$ is a non degenerate *-representation, then $T:=S|_{\cB_H}$ is $\cB-$positive and there is a unique unitary $U\colon X^p_T\to X$ mapping $f\otimes_T \xi$ to $\intform{S}_f\xi$ for all $f\in C_c(\cB)$ and $\xi\in X.$
	Moreover, $U$ intertwines $\Ind_H^\cB(T)$ and $S.$
	Consequently, $\Ind_H^\cB(S|_{\cB_H})=S$ (up to a unitary equivalence).
\end{remark}

It is a very good exercise of translation (between the abstract and concrete induction processes) to give an ``abstract'' version of the proof of the Lemma in \cite[XI 11.3]{FlDr88}; which states that 
\begin{equation}\label{equ:inducing increases norm}
\|T_b\|\leq   \|\Ind_H^\cB(T)_b\|
\end{equation}
for all $b\in B_e.$
We will use the inequality above in Proposition~\ref{prop:can change a banach algebraic bundle by its completion}. 

\subsubsection{The inducible C*-algebra}\label{sssec:inducible algebra}

Let $\cB=\{B_t\}_{t\in G}$ be a Banach *-algebraic bundle and $H \sbgp G.$
Consider a state $\varphi$ of $C^*(\cB_H)$ and its GNS-construction $(X_\varphi,\varPsi_\varphi,\xi_\varphi),$ $\xi_\varphi$ being the cyclic vector of the *-representation $\varPsi_\varphi\colon C^*(\cB_H)\to \bB(X_\varphi).$

There exists a unique *-representation $T^\varphi\colon \cB_H\to \bB(X_\varphi)$ whose integrated form is the composition of the natural map $\iota\colon L^1(\cB_H)\to C^*(\cB_H)$ with $\varPsi_\varphi.$
In fact $T^\varphi$ is cyclic with cyclic vector $\xi_\varphi.$
We say $\varphi$ is $\cB-$positive if $T^\varphi$ is $\cB-$positive and denote $S^+_{\cB,H}$ the set of $\cB-$positive states of $C^*(\cB_H).$

The universal $\cB-$positive *-representation of $C^*(\cB_H),$ $\varPsi^{\cB}_H,$ is the direct sum of all the $\varPsi_\varphi$ with $\varphi\in S^+_{\cB,H}.$
We set $C^*_{\Ind^p}(\cB_H):=\varPsi^\cB_H(C^*(\cB_H))$ and call it the concrete $\cB-$inducible cross sectional C*-algebra of $\cB_H.$
Any C*-algebra C*-isomorphic to $C^*_{\Ind^p}(\cB_H)$ will be called the (or a)  $\cB-$inducible cross sectional C*-algebra of $\cB_H$ and, when no confusion can arise, it will be denoted $C^*_{\Ind^p}(\cB_H).$
For example, if $\cB$ is a Fell bundle, then $\varPsi^\cB_H$ is the universal representation of $C^*(\cB_H)$ and we write $C^*(\cB_H)=C^*_{\Ind^p}(\cB_H)=\varPsi^\cB_H(C^*(\cB_H)).$

Clearly, the direct sum $T^{\cB,H}$ of all the $T^\varphi$ with $\varphi\in S^+_{\cB,H}$ is a $\cB-$positive *-representation of $\cB_H.$
The integrated form $\intform{T}^{\cB,H}$ and $\varPsi^\cB_H$ are related by the identity $\intform{T}^{\cB,H}=\varPsi^\cB_H\circ \iota.$

Let $T\colon \cB_H\to \bB(X)$ be any non degenerate $\cB-$positive *-representation of $\cB_H.$
By decomposing $T$ into a direct sum of cyclic *-representations we are in fact decomposing it into a direct sum of cyclic $\cB-$positive *-representations.
All the members of this direct sum are (unitary equivalent to some) of the form $T^\varphi,$  so $T$ is weakly contained in $T^{\cB,H}$ (i.e.  $\|\intform{T}_f\|\leq \|\intform{T}^{\cB,H}_f\|$ for all $f\in C_c(\cB)$).
Then there exists a unique *-representation $\rho^T\colon C^*_{\Ind^p}(\cB_H)\to \bB(X)$ such that $\rho^T\circ \varPsi^\cB_H\circ \iota =\intform{T}.$
Reciprocally, given a non degenerate *-representation $\rho\colon C^*_{\Ind^p}(\cB_H)\to \bB(Y)$ the composition $\rho\circ \varPsi^\cB_H\circ \iota$ is a non degenerate *-representation which is the integrated form of a unique *-representation $T^\rho\colon \cB_H\to \bB(Y).$
For all $f\in L^1(\cB_H)$ we have $\|\intform{T^\rho}_f\|=\| \rho\circ \varPsi^\cB_H\circ \iota (f)\|\leq \|\varPsi^\cB_H\circ \iota (f)\|=\|T^{ {\cB,H}}_f\|.$
Since $T^{\cB,H}$ is $\cB-$positive \cite[XI 8.21]{FlDr88} implies $T^\rho$ is so.

Uniqueness of (dis)integrated forms implies $T=T^{\rho^T}$ and $\rho=\rho^{T^\rho},$ thus $T\rightsquigarrow T^\rho$ establishes a one to one correspondence between non degenerate $\cB-$positive *-representations of $\cB_H$ and non degenerate *-representations of $ C^*_{\Ind^p}(\cB_H).$
Fell's induction process $T\rightsquigarrow\Ind_H^\cB(T)$ may be thought of as a way of transforming *-representations of $C^*_{\Ind^p}(\cB_H)$ into *-representations of $C^*(\cB).$
Part of what we want to do is to describe this process as induction via a $*-$homomorphism $\theta\colon C^*(\cB)\to \bB(Z_{C^*_{\Ind^p}(\cB_H)}).$

\begin{remark}\label{rem:describe Fell induction as Rieffel induction Banach star}
Given a *-homomorphism 	$\theta\colon C^*(\cB)\to \bB(Z_{C^*_{\Ind^p}(\cB_H)})$ and a non degenerate $\cB-$positive *-representation $T\colon \cB_H\to\bB(X),$ the *-representation 
\begin{equation*}
\Ind^\theta(\rho^T)\circ \iota_G \equiv (\theta\circ\iota_G)\otimes_{\rho_T}1 \colon L^1(\cB)\to \bB(Z_{C^*_{\Ind^p}(\cB_H)}\otimes_{\rho^T}X)
\end{equation*}
is the integrated form of a unique non degenerate *-representation $S^{\theta,T}\colon \cB\to \bB(Z_{C^*_{\Ind^p}(\cB_H)}\otimes_{\rho^T}X).$
By uniqueness of integrated forms, $S^{\theta,T}=\Ind_H^\cB(T)$ (up to unitary equivalence) if and only if $\intform{\Ind}_H^\cB(T)\circ \iota_G = (\theta\circ\iota_G) \otimes_{\rho_T}1.$
Our goal is to work out a $\theta$ such that this  identity holds for all $T.$
\end{remark}

\subsection{Induction and the bundle C*-completion}\label{ssec:inducion and bundle completion}
The bundle C*-completion $\cC=\{C_t\}_{t\in G}$ of a Banach *-algebraic bundle $\cB= \{B_t\}_{t\in G}$ is constructed in  \cite[VIII 16.7]{FlDr88}, the key fact being that $\cC$ is a Fell bundle whose representation theory is equivalent to that of $\cB.$
The result below is the improvement one gets from \cite[XI 12.16]{FlDr88} after solving the positivity problem on the affirmative.

\begin{proposition}\label{prop:can change a banach algebraic bundle by its completion}
 Let $\cB=\{B_t\}_{t\in G}$ be a Banach *-algebraic bundle, $H \sbgp  G,$ $\cC$  the bundle C*-completion of $\cB$ and $\rho\colon \cB\to \cC$ the canonical morphism of \cite[VIII 16.7]{FlDr88}.
 Then 
 \begin{enumerate}
 \item For every $\cB-$positive *-representation $T\colon \cB_H\to \bB(Y)$ there exists a unique *-representation $T^\rho\colon \cC_H\to \bB(Y)$ such that $T^\rho\circ (\rho|_{\cB_H})=T.$
 Reciprocally, for every *-representation $S\colon \cC_H\to \bB(Z)$ the composition $T:=S\circ (\rho|_{\cB_H})\colon \cB_H\to \bB(Z)$ is a $\cB-$positive *-representation and $T^\rho =S.$
 Also, up to a unitary equivalence,
 \begin{equation}\label{equ:induced representations and completions}
 \Ind_H^\cB(T) =  \Ind_H^\cC(T^\rho)\circ \rho.
 \end{equation}
 \item If $\iota\colon L^1(\cB_H)\to C^*(\cB_H)$ is the canonical *-homomorphism, then there exists a unique *-homomorphism $\tau\colon C^*_{\Ind^p}(\cB_H)\to C^*(\cC_H)$ such that such that $\tau\circ \varPsi^\cB_H\circ \iota (f)=\rho\circ f$ for all $f\in C_c(\cB_H).$
 Moreover, $\tau$ is a C*-isomorphism.
 \end{enumerate}
 \begin{proof}
 The first claim is, basically, a combination of the ideas presented in \cite[XI 12.6 \& VIII 16]{FlDr88}.
 Later, in Proposition~\ref{prop:induciton and completions}, we will prove a slightly more general fact, so we omit the proof (please note we are not falling into a circular argument).
 
  Let $S\colon \cC_H\to \bB(X)$ be a non degenerate *-representation with faithful integrated form.
  The composite $T:=S\circ (\rho|_{\cB_H})$ is $\cB-$positive and a simple computation reveals that $\intform{S}_{\rho\circ f} = \intform{T}_f$ for all $f\in C_c(\cB);$ which in turn implies $\|\rho\circ f\|_{C^*(\cC_H)}=\|\intform{S}_{\rho\circ f}\|=\|\intform{T}_f\|\leq \|\varPsi^\cB_H(\iota (f))\|.$
  So, we may define a contractive morphism of *-algebras
  \begin{equation*}
  	\tau_0\colon \varPsi^\cB_H\circ \iota (C_c(\cB_H))\to C^*(\cC_H)\qquad \varPsi^\cB_H\circ\iota (f)\mapsto \rho\circ f,
  \end{equation*}
  which we can continuously extend to a *-homomorphism $\tau\colon C^*_{\Ind^p}(\cB_H)\to C^*(\cC_H).$
  Notice the range of $\tau_0$ is dense in $C^*(\cC_H),$ so $\tau$ is surjective.
  
  By the construction of $S,$ $(T^{\cB,H})^\rho$ is weakly contained in $S.$
  Hence, for every $f\in C_c(\cB_H):$
  \begin{equation*}
    \|\varPsi^\cB_H\circ \iota (f)\| = \|\intform{T}^{\cB,H}_f\|\leq \|\intform{S}_{\rho\circ f}\|=\|\rho\circ f\|_{C^*(\cC_H)} \leq  \|\varPsi^\cB_H\circ \iota (f)\|.
  \end{equation*}
  Thus $\tau_0$ is an isometry and $\tau$ a C*-isomorphism.
 \end{proof}
\end{proposition}

The procedure of passing to the bundle C*-completion can be quite drastic.

\begin{example}\label{exa:only zero completion}
 Let $\cB=\bC\times \bZ_2=\{B_0,B_1\}$ be the trivial bundle with constant fiber $\bC$ over the additive group $\mathbb{Z}_2=\{0,1\}.$
 Define the involution by $(\lambda,r)^*:=(\overline{\lambda},r)$ and the product by
 \begin{equation*}
  (\lambda,r)(\mu,s)=\begin{cases}
                      (\lambda\mu,r+s) \mbox{ if }r\neq 1 \mbox{ or }s\neq 1,\\
                      (-\lambda \mu,0) \mbox{ if }r=s=1.
                     \end{cases}
 \end{equation*}
 We leave to the reader the verification of the fact that $\cB$ is a Banach *-algebraic bundle with these operations.
 If $T\colon \cB\to \bB(Y)$ is a *-representation, then either $T|_{B_0}\colon B_0\to \bB(Y)$ is faithful or either it is the zero representation. 
 The first case is excluded because $(-1,0)=(1,1)^*(1,1)$ is negative in $B_0,$ so we must have $T|_{B_0}=0$ and this forces $\|T_b\| = \|T_{b^*b}\|^{1/2}=0$ for all $b\in \cB.$
 Hence, the bundle C*-completion $\cC$ of $\cB$ is the trivial bundle $\{0\}\times \bZ_2$ but, quite interestingly, $B_0$ is a C*-algebra and the identity $\|b^*b\|=\|b\|^2$ holds for all $b\in \cB.$
 Note that $\cC_{\{e\}}$ \underline{is not} the bundle C*-completion of $\cB_{\{e\}},$ implying that the processes of taking reductions (to a subgroup) and C*-completions \underline{do not} commute.
\end{example}

Let $H,G,\cB$ and $\cC$ be as in Proposition~\ref{prop:can change a banach algebraic bundle by its completion}.
Since every $*-$representation of $\cB$ is $\cB-$positive, we may identify $C^*(\cB)=C^*_{\Ind^p}(\cB_G)$ with $C^*(\cC_G)\equiv C^*(\cC)$ and  $C^*_{\Ind^p}(\cB_H)$ with $C^*(\cC_H).$

Suppose we can construct a *-homomorphism $\theta\colon C^*(\cC)\to \bB(Y_{C^*(\cC_H)})$ that describes Fell's induction process for $\cC$ (see Remark~\ref{rem:describe Fell induction as Rieffel induction Banach star}).
By Proposition~\ref{prop:can change a banach algebraic bundle by its completion}, $\theta$ would also describe Fell's abstract induction process for $\cB.$
Thus, at least from a theoretical point of view, it suffices to consider Fell bundles.

\subsection{Induced cross sectional C*-algebras}\label{ssec:induced algebras}

Let $\cB=\{B_t\}_{t\in G}$ be a Banach *-algebraic bundle and $H \sbgp G.$
Take a non degenerate $\cB-$positive *-representation $T\colon \cB_H\to \bB(X)$ that weakly contains any other such *-representation of $\cB_H$ (like the $T^{\cB,H}$ of Section~\ref{sssec:inducible algebra}).

Using the abstract induction process one realises that, given any *-representation $S\colon \cB_H\to \bB(Y),$ $S$ and its non degenerate part induce the same *-representation.
So, to compare $\|\Ind_H^\cB(S)_f\|$ and $\|\Ind_H^\cB(T)_f\|,$ one may assume $S$ is non degenerate.

To say $S$ is weakly contained in $T$ means the sequence of finite direct sums of $T$ converges to $S$ in the regional topology \cite{FlDr88}.
Then \cite[XI 12.2 \& 12.4]{FlDr88} imply $\Ind_H^\cB(S)$ is weakly contained in $\Ind_H^\cB(T),$ giving the bound 
\begin{equation*}
   \| \intform{\Ind}_H^\cB(S)_f \|\leq \|\intform{\Ind}_H^\cB(T)_f\|
\end{equation*}
for every $f\in L^1(\cB).$
Thus $\|\intform{\Ind}_H^\cB(T)_f\|$ is the maximum of 
\begin{equation*}
F(\cB,H,f):= \{ \|\intform{\Ind}_H^\cB(S)_f\|\colon S\mbox{ is a }\cB-\mbox{positive *-representation of }\cB_H\}\subset [0,+\infty)
\end{equation*}
for every $f\in L^1(\cB).$

We define the C*-seminorm $\|\ \|_H^\cB$ of $L^1(\cB)$ by 
\begin{equation*}
\|f\|^\cB_H:=\sup F(\cB,H,f)= \|\intform{\Ind}_H^\cB(T)_f\|.
\end{equation*}

\begin{definition}\label{defi:of H cross sectional algebra}
 The induced $H-$cross sectional C*-algebra of $\cB,$ $C^*_H(\cB),$ is the completion of the quotient $L^1(\cB)/\{f\in L^1(\cB)\colon \|f\|^\cB_H=0\}$ with respect to quotient norm.
 The canonical morphism of *-algebras from $L^1(\cB)$ to $C^*_H(\cB)$ will be denoted $\iota_H$ or $\iota^\cB_H$ and the corresponding map from $C^*(\cB)$ to $C^*_H(\cB)$ will be denoted $q^\cB_H.$
 In case $G=H$ we write $\iota$ instead of $\iota_G.$
\end{definition}

\begin{remark}
	Given closed subgroups of $G,$ $H$ and $K,$ we have $\|\ \|^\cB_H\leq \|\ \|^\cB_K$ if and only any *-representation of $\cB$ induced from (a $\cB-$positive *-representation) of $\cB_H$ is weakly contained in a *-representation induced from $\cB_K.$
\end{remark}

\begin{remark}
The relation $\|\ \|^\cB_H\leq \|\ \|^\cB_K$ does not imply that $H\subset K,$ see Proposition \ref{prop:condition for CHB=CtHtmuB}.
\end{remark}

The maps $q^\cB_H$ from the definition above and the $q^\cB_{HK}$ we construct below will be used extensively in the rest of the article, specially when working with the notion of weak containment.

\begin{remark}[and definition]\label{rmk:composition of q maps}
We know from \cite[XI 12.15]{FlDr88} (reproduced here as Corollary~\ref{cor:induction in stages, the Fell way}) that *-representations can be induced in stages.
Then, given inclusions $H\sbgp K\sbgp G,$ we have $\|\ \|^\cB_H\leq \|\ \|^\cB_K;$ which is equivalent to the existence of a unique *-homomorphism $q^\cB_{HK}\colon C^*_K(\cB)\to C^*_H(\cB)$ such that $q^\cB_{HK}\circ q^\cB_{K}=q^\cB_H.$
Notice that $q^\cB_{HK}$ is a surjective and that the following are equivalent:
\begin{enumerate}[(1)]
 \item $q^\cB_{HK}$ is an isomorphism.
 \item $\|\ \|^\cB_H = \|\ \|^\cB_K.$
 \item Every *-representation of $\cB$ induced from $\cB_K$ is weakly contained in a *-representation induced from $\cB_H.$
\end{enumerate}
\end{remark}

\begin{remark}\label{rmk:C_G(B)=C(B)}
	If we use Remark~\ref{rmk:induce a restriction} with $H=G$ we get that any *-representation $T\colon \cB_H\to\bB(Y)$ is $\cB-$positive and $\intform{\Ind}_H^\cB(T)$ is unitary equivalent to $\intform{T}.$
	Thus $C^*_G(\cB)\equiv C^*(\cB)$ and $q^\cB_G\colon C^*(\cB)\to C^*_G(\cB)$ is the identity map.
\end{remark}

\begin{remark}\label{rem:BH seminorms are norms if B is a Fell bundle}
   Suppose $\cB$ is a Fell bundle.
   We know from \cite[VIII 16.4]{FlDr88}
   that $L^1(\cB)$ is reduced (i.e. has a faithful *-representation).
   A close examination of the proof reveals that if $\pi\colon B_e\to \bB(X)$ is the universal representation, then the integrated form $\Ind_{\{e\}}^\cB(\pi)\colon L^1(\cB)\to \bB(X^\pi_p)$ is faithful.
   Hence, $\|\ \|^\cB_{\{e\}}$ is a norm and this implies $\|\ \|^\cB_{H}$ is a norm for every $H \sbgp G.$
\end{remark}

More often that not we treat $C^*_H(\cB)$ as an isomorphism class of C*-algebras.
So the expression $C^*_H(\cB)=A$ means that a particular C*-algebra $A$ is C*-isomorphic to the particular construction of $C^*_H(\cB)$ given in the definition above.
 
\begin{remark}\label{rem:induced and faithful integrated form}
	Given a *-representation $T\colon \cB_H\to \bB(X),$ $\Ind_H^\cB(T)$ weakly contains any *-representation of the form $\Ind_H^\cB(S)$ if and only if $\|f\|^\cB_H=\|\Ind_H^\cB(T)_f\|$ for all $f\in C_c(\cB).$
	In this situation we may identify $\intform{T}_f$ with $q^\cB_H(f)$ for all $f\in L^1(\cB)$ and $C^*_H(\cB)$ becomes the (norm) closure of $\intform{\Ind}_H^\cB(T) (C_c(\cB)).$
	This is the case if $T$ weakly contains any other non degenerate $\cB-$positive *-representation   of $\cB_H.$
\end{remark}

\begin{remark}\label{rem:H=e then H algebra is the reduced one}
As pointed out in the introduction, if $\cB$ is a Fell bundle, then for every *-representation $\pi\colon B_e\to \bB(X)$ we have $\regrepEN{\cB} \otimes_\pi 1=\intform{\Ind}_e^\cB(\pi).$
Taking $\pi$ non degenerate and faithful we get $\|f\|^\cB_{\{e\}}=\|\intform{\Ind}_e^{\cB}(\pi)_f\|=\|\regrepEN{\cB}_f\otimes_\pi 1\|=\|\regrepEN{\cB}_f\|$ for all $f\in L^1(\cB).$
Then $C^*_\red(\cB)=C^*_{\{e\}}(\cB).$
\end{remark}

\begin{remark}\label{rem:H cross sectional algebras of bundle an completion}
 If $\cC$ is the bundle C*-completion of $\cB,$  by the proof of Proposition~\ref{prop:can change a banach algebraic bundle by its completion} the $T$ of Remark~\ref{rem:induced and faithful integrated form} can be taken of the form $S\circ (\rho|_{\cB_H});$ with $S\colon \cC_H\to \bB(X)$ having faithful and non degenerate integrated form.
 By \eqref{equ:induced representations and completions} $C^*_H(\cB)$  is the norm closure of 
 \begin{equation*}
 \intform{\Ind}_H^\cB(T) (C_c(\cB))
  = \intform{\Ind}_H^{\cC}(S)(\{\rho\circ f\colon f\in C_c(\cB)\}).
 \end{equation*}
 The *-algebra on the right hand side above is norm dense in $ \intform{\Ind}_H^{\cC}(S)(C_c( \cC)).$
 So, after a faithful *-representation, $C^*_H(\cB)=C^*_H(\cC).$
\end{remark}

The preceding remarks together with the last paragraph of Section~\ref{ssec:inducion and bundle completion} reduce the study of the induction process of a Banach *-algebraic bundle to that of its bundle C*-completion.
Thus we may continue the study of the induction process working with Fell bundles.

\section{A *-homomorphism that describes abstract induciton}\label{sec:abstract induction revisited}

Let $\cB=\{B_t\}_{t\in G}$ be a Fell bundle, $H \sbgp G$ and consider the function $p\colon C_c(\cB)\to C_c(\cB_H)$ and the action of $C_c(\cB_H)$ on $C_c(\cB)$ described in \eqref{equ: p} and \eqref{equ: action to induce}.
By Theorem~\ref{thm:characterization of positive representations},  the $C_c(\cB_H)-$valued inner product
\begin{equation*}
C_c(\cB)\times C_c(\cB)\to C_c(\cB_H),\ (f,g)\mapsto p(f^* * g)
\end{equation*}
is positive definite with respect to the C*-completion $C_c(\cB_H)\subset C^*(\cB_H).$
Besides, the condition $p(f^**f)=0$ implies $f^**f(e)=\int_G f(t)^*f(t)\, dt=0$ and this happens if and only if $f=0.$
Properties \eqref{equ:ref properties of p} put us in the situation of \cite[Lemma 2.16]{Raeburn1998morita}, so 
\begin{align*}
\|\ \|_H& \colon C_c(\cB)\to \bR & f&\mapsto \|f\|_{H}:=\|p(f^**f)\|_{C^*(\cB_H)}^{1/2}
\end{align*}
is a norm on $C_c(\cB)$ and the $\|\ \|_{H}-$completion of $C_c(\cB)$ is a right $C^*(\cB_H)-$Hilbert module which we denote 	 $L^2_H(\cB)$ and call the $H-$Hilbert module of $\cB.$

The next Proposition is motivated by the results of \cite{ExNg}.
The products $bf$ we use below are those of Proposition~\ref{prop:integration of *-representation}.

\begin{proposition}\label{prop:rep on inducing module}
	Let $\cB=\{B_t\}_{t\in G}$ be a Fell bundle and $H \sbgp G.$
	Then $L^2_H(\cB)$ is $C^*(\cB_H)$-full and there exists a unique *-representation $\regrep{H\cB}\colon \cB\to \bB(L^2_H(\cB))$ such that $\regrep{H\cB}_b (f)=bf$ for all $b\in \cB$ and $f\in C_c(\cB).$
	Moreover,
	\begin{enumerate}
		\item $\regrep{H\cB}$ is non degenerate.
		\item The integrated form  $\intregrep{H\cB}\colon C^*(\cB)\to \bB(L^2_H(\cB))$ is the unique *-representation such that  $\intregrep{H\cB}_f(g)=f*g$ for all $f,g \in C_c(\cB).$
		\item The inclusion $C_c(\cB)\to L^2_H(\cB)$ is continuous in the inductive limit topology.
	\end{enumerate}
	\begin{proof}
		The map $p\colon C_c(\cB)\to C_c(\cB_H)$ is continuous in the inductive limit topology and surjective by \cite[II 14.8]{FlDr88}.
		Besides, $\{f^**g\colon f,g\in C_c(\cB)\}$ spans a subspace of $C_c(\cB)$ which is dense in the inductive limit topology (because Fell bundles have strong approximate units).
		Putting all together we obtain that $\{p(f^**g)\colon f,g\in C_c(\cB)\}$ spans a subspace of $C_c(\cB_H)$ which is dense in the inductive limit topology, implying that the inner products of $L^2_H(\cB)$ span a dense subspace of $C^*(\cB_H).$
		
		We will need to use claim (3) in the rest of the proof, so we prove it first.
		Take a compact set $D\subset G$ and denote $C_D(\cB)$ the set of elements of $C_c(\cB)$ with support contained in $D.$
		If $\alpha_D$ is the measure of $(D^{-1}D)\cap H$ with respect to $d_Ht;$ $\beta_D$ is the maximum value of $\Delta_G(t)^{1/2}\Delta_H(t)^{-1/2}$ for $t\in (D^{-1}D)\cap H$ and $\gamma_D$ the measure of $D$ with respect to $d_Gr,$ then for all $f\in C_D(\cB)$ we have
		\begin{equation*}
		\| f \|_{L^2_H(\cB)}^2 
		\leq \| p(f^* * f) \|_1 
		\leq \int_H \int_G \Delta_G(t)^{1/2}\Delta_H(t)^{-1/2} \|f(r)\|\|f(rt)\|\, d_Grd_Ht \leq  \alpha_D\beta_D\gamma_D \|f\|_{\infty}^2.
		\end{equation*}
		Hence, the inclusion $C_D(\cB)\to L^2_H(\cB)$ is norm continuous and (3) follows from the universal property of the inductive limit topology.
		
		Take a non degenerate *-representation $T\colon \cB_H\to \bB(X)$ with faithful integrated form.
		By Theorem~\ref{thm:B positivity implies inducibility}, for all $f\in C_c(\cB)$ and $b\in \cB$ we have
		\begin{equation*}
         \intform{T}_{p((bf)^**(bf))}\leq \|b\|^2 \intform{T}_{p(f^**f)},
		\end{equation*}
		which in terms of inner product means $\langle bf,bf\rangle \leq \|b\|^2\langle f,f\rangle$ and implies $\|bf\|_H\leq \|b\|\|f\|_H.$
		Then, for every $b\in \cB,$ there is a unique bounded operator $\regrep{H\cB}_b\colon L^2_H(\cB)\to  L^2_H(\cB)$ such that $\regrep{H\cB}_bf=bf$ for all $f\in C_c(\cB).$
		
		The identity $\langle \regrep{H\cB}_b f,g\rangle = \langle f,\regrep{H\cB}_{b^*}g\rangle$ holds for all $b\in \cB$ and $f,g\in C_c(\cB).$
		Thus each $\regrep{H\cB}_b$ is adjointable with adjoint $\regrep{H\cB}_{b^*}.$
		This gives a map $\regrep{H\cB}\colon \cB\to \bB(L^2_H(\cB)),$ which turns out to be multiplicative in $\cB$ and linear on each fibre.
		
		Given $f,g\in C_c(\cB),$ the map $\cB\to C_c(\cB_H),$ $b\mapsto p(f^* * bg)=\langle f,\regrep{\cB}_bg\rangle ,$ is continuous in the inductive limit topology.
		Hence, it is continuous as a function with range $C^*(\cB_H)$ (and the norm topology).
		Consequently, $\regrep{H\cB}$ is a *-representation of $\cB$ in $L^2_H(\cB).$
		
		For $f\in C_c(\cB)\subset C^*(\cB)$ and $g\in C_c(\cB)\subset L^2_H(\cB),$ the function $F\colon G\to C_c(\cB)$ given by $F(t)=f(t)g=\regrep{H\cB}_{f(t)}g$ is continuous in the inductive limit topology and integrable\footnote{Given a positive measure $\mu$ on a set $X$ and a locally convex space $Y,$ a function $f\colon X\to Y$ is integrable if there exists an element $y\in Y$ such that for every continuous linear functional $\varphi$ of $Y,$ $\varphi\circ f$ is integrable and $\int_X \varphi\circ f\, d\mu=\varphi(y).$
		The element $y$ is called the integral of $f$ and it is denoted $\int_X f\,d \mu.$ } with respect to such topology.
		This is so because there exists a compact $D\subset G$ containing the support of all the sections $F(t).$
		Let $\int_G F(t)\, dt$ be the integral of $F.$
		As in \cite[VIII 5.2]{FlDr88} we have $\int_G F(t)\, dt=f*g.$
		
		The topology of $C_c(\cB)$ relative to the norm topology of $L^2_H(\cB)$ is weaker than the inductive limit topology, so $F$ is integrable as a function from $G$ to $L^2_H(\cB)$ and its integral is $\int_G F(t)\, dt.$
		In $L^2_H(\cB)$ we have $F(t)=\regrep{H\cB}_{f(t)}g,$ so
		\begin{equation*}
		f*g = \int_G F(t)\, dt = \int_G \regrep{H\cB}_{f(t)}g\, dt = \intregrep{H\cB}_f g.
		\end{equation*}
		
		The only thing that remains to be shown is the non degeneracy of $\regrep{H\cB}.$  
		But this is a straightforward consequence of the fact that, for any $f\in C_c(\cB)$ and any approximate identity $\{u_i\}_{i\in I}$ of $B_e,$ in the inductive limit topology we have $f=\lim_i \regrep{H\cB}_{e_i}f.$
	\end{proof}
\end{proposition}

Take a *-representation $T\colon \cB_H\to \bB(X).$
The balanced tensor product $L^2_H(\cB)\otimes_{\intform{T}}X,$ where $T$ stands for the integrated form, is generated by tensors $f\otimes_{\intform{T}}\xi$ ($f\in C_c(\cB)$ and $\xi\in X$) such that 
\begin{equation*}
  \langle f\otimes_{\intform{T}} \xi,g\otimes_{\intform{T}} \eta\rangle =\langle \xi,\intform{T}_{p(f^**g)}\eta\rangle.
\end{equation*}

Thus $L^2_H(\cB)\otimes_{\intform{T}}X$ is the induced space $X^p_T$ of Section~\ref{ssec:abstract induction}.
Moreover, the operators $\regrep{H\cB}_u \otimes_{\intform{T}}1$ and $\Ind_H^\cB(T)_u$ agree on each elementary tensor $f\otimes_T \xi$ (for all $u\in \cB$ and $u\in C_c(\cB)$).
Then
\begin{equation*}
  \regrep{H\cB} \otimes_{\intform{T}}1 =\Ind_H^\cB(T)
\end{equation*}
for the integrated and non integrated forms of $\regrep{H\cB}$ and $\Ind_H^\cB(T).$
It is then clear (see Remark \ref{rem:describe Fell induction as Rieffel induction Banach star}) that induction via the *-homomorphism 
\begin{equation*}
  \intregrep{H\cB}\colon C^*(\cB)\to\bB(L^2_H(\cB)) 
\end{equation*}
describes Fell's abstract induction process.

In \cite{ExNg} the authors define $C^*_\red(\cB)$ as $\intregrepEN{\cB}(C^*(\cB))\subset \bB(L^2_e(\cB)).$
Together with Remark~\ref{rem:H=e then H algebra is the reduced one} this implies $C^*_{\{e\}}(\cB)=\intregrep{\{e\}\cB}(C^*(\cB))\subset \bB(L^2_{\{e\}}(\cB));$ which is a particular case of the following.

\begin{corollary}\label{cor:norm of reg rep is induced norm}
  Let $\cB=\{B_t\}_{t\in G}$ be a Fell bundle and $H \sbgp G.$
  Then $\intregrep{H\cB}\colon C^*(\cB)\to \bB(L^2(\cB))$ is faithful when restricted to $L^1(\cB).$
  Besides, for all $f\in L^1(\cB),$  $\|f\|^\cB_H=\|\intregrep{H\cB}_f\|$  and $C^*_H(\cB)=\intregrep{H\cB}(C^*(\cB)).$
\end{corollary}
\begin{proof}
Let $T\colon \cB_H\to \bB(X)$ be a *-representation with faithful integrated form, so $T$ weakly contains any *-representation of $\cB_H.$
The map $\psi\colon \bB(L^2_H(\cB))\to \bB(L^2_H(\cB)\otimes_T X),\ A\mapsto A\otimes_T 1,$ is an isometry and, for all $f\in L^1(\cB),$ $ \| f \|^\cB_H = \|\intform{\Ind}_H^\cB(T)_f\| = \|\intregrep{H\cB}_f\otimes_T 1\|
   = \|\psi(\intregrep{H\cB}_f)\|
   =\|\intregrep{H\cB}_f\|.$
It is then clear that we may identify $C^*_H(\cB)$ with $\intregrep{H\cB}(C^*(\cB)).$
Rephrasing Remark~\ref{rem:BH seminorms are norms if B is a Fell bundle} in terms of $\intregrep{H\cB}$ we get that $\intregrep{H\cB}|_{L^1(\cB)}$ is faithful.
\end{proof}

\begin{remark}[and notation]\label{rmk:abstract induction hilbert modules}
   Any *-representation $T\colon \cB_H\to\bB(Y_A)$ can be induced via $\intregrep{H\cB}$ to a *-representation
   \begin{align*}
   	\Ind_H^\cB(T)& \colon \cB\to \bB(L^2_H(\cB)\otimes_{\intform{T}}Y_A) & \Ind_H^\cB(T)_b&:= \regrep{H\cB}_b\otimes_{\intform{T}} 1.
   \end{align*}
   By Remark~\ref{rem:tensor product of *-representations} the identity $\Ind_H^\cB(T)=\regrep{H\cB}\otimes_{\intform{T}}1$ also holds for integrated forms.  
\end{remark}

\begin{remark}\label{rem:induction and balanced tensor product}
	Let $T\colon \cB_H\to \bB(Y_A)$ and  $\pi\colon A\to  \bB(X)$ be *-representations ($X$ is a Hilbert space).
	Then we can form the $\pi-$balanced tensor product representation $T\otimes_\pi 1\colon \cB_H\to \bB(Y_A\otimes_\pi X)$ mapping $b\in \cB$ to $T_b\otimes_\pi 1.$
	Using the canonical isomorphism 
	\begin{align*}
(L^2_H(\cB)\otimes_{\intform{T}}Y_A)\otimes_\pi X & = L^2_H(\cB)\otimes_{\intform{T}\otimes_\pi 1}(Y_A\otimes X) & 
(f\otimes_T \xi)\otimes_\pi \eta = f\otimes_{T\otimes 1}(\xi\otimes \eta)
	\end{align*}
	we can deduce that (up to a unitary equivalence) the balanced tensor product commutes with induction:
	\begin{equation*}
    \Ind_H^\cB(T\otimes_\pi 1) 
      =\regrep{H\cB}\otimes_{\intform{T}\otimes_\pi 1} 1
      = (\regrep{H\cB}\otimes_{\intform{T}} 1)\otimes_\pi 1
      = \Ind_H^\cB(T)\otimes_\pi 1.
	\end{equation*}	
\end{remark}

\begin{remark}[Weak containment and induction]\label{rem:wc and induction}
  Suppose the *-representation $T\colon \cB_H\to \bB(X)$ is weakly contained in a family $\mathcal{S}$ of *-representations (on Hilbert spaces) of $\cB_H.$
This means $T$ is the limit, in the regional topology of \cite{FlDr88}, of a net of finite direct sums of elements of $\mathcal{S}.$
By \cite[XI 12.2 \& 12.4]{FlDr88}, $\Ind_H^\cB(T)$ is weakly contained in $\{\Ind_H^\cB(S)\}_{S\in \mathcal{S}}.$
\end{remark}

The Hilbert module version of the Remark above is as follows.

\begin{proposition}\label{prop:induction increases norm}
   Let $\cB$ be a Fell bundle over $G,$ $H \sbgp G,$ $T\colon \cB_H\to \bB(X_A)$ a *-representation and $\mathcal{S}$ a set of *-representations of $\cB_H$ on Hilbert modules.
   If $\|\intform{T}_f\|\leq \sup_{S\in \mathcal{S}}\|\intform{S}_f\|$ for all $f\in L^1(\cB_H),$ then $\|\intform{\Ind}_H^{\cB}(T)_f\|\leq \sup_{S\in \mathcal{S}}\|\intform{\Ind}_H^\cB(S)_f\|\leq \|\intregrep{H\cB}_f\|$ for all $f\in L^1(\cB_H),$
\end{proposition}
\begin{proof}
 For each $S\in \mathcal{S}$ we denote $Y^S$ and $A_S$ the Hilbert module and the C*-algebra associated to $S$ ($S\colon \cB_H\to \bB(Y^S)$).
 Take faithful and non degenerate *-representations $\pi\colon A\to \bB(X')$ and $\rho_S\colon A_S\to \bB(Z^S)$ (for all $S\in \mathcal{S}$).
 For all $f\in L^1(\cB_H)$
 \begin{equation*}
 	\| (\intform{T}\otimes_\pi 1)_f \|=\| \intform{T}_f \|\leq \sup_{S\in \mathcal{S}}\|\intform{S}_f\|
 	=\sup_{S\in \mathcal{S}}\|(\intform{S}\otimes_{\rho_S}1)_f\|;
 \end{equation*}
 meaning that $T\otimes_\pi 1$ is weakly contained in $\{S\otimes_{\rho_S}1\}_{S\in \mathcal{S}}.$
 
 By Remarks~\ref{rem:induction and balanced tensor product} and~\ref{rem:wc and induction},
 \begin{equation*}
   \|\intform{\Ind}_H^{\cB}(T)_f\| = \|(\intform{\Ind}_H^{\cB}(T)\otimes_\pi 1)_f\| \leq \sup_{S\in \mathcal{S}}\|(\intform{\Ind}_H^\cB(S)\otimes_{\rho_S}1)_f\| = \sup_{S\in \mathcal{S}}\|\intform{\Ind}_H^\cB(S)_f\|
 \end{equation*}
 for all $f\in L^1(\cB).$
 Besides, for all $S\in \mathcal{S}$ and $f\in L^2(\cB),$
 \begin{equation*}
    \| \intform{\Ind}_H^\cB(S)_f \| = \|\intform{\Ind}_H^\cB(S\otimes_{\rho_S}1)_f\|\leq \|f\|^\cB_H = \|\intregrep{H\cB}_f\|,
 \end{equation*}
 and the proof follows by taking supremum over $S\in \mathcal{S}.$
\end{proof}

\subsection{The universal system of imprimitivity}\label{ssec:universal system of imprimitivity}
Our next goal is to obtain the induced systems of imprimitivity of \cite[XI 14.3]{FlDr88} in an abstract framework.
Let $\cB=\{B_t\}_{t\in G}$ be a Fell bundle and $H \sbgp G.$
The space of cosets $G/H:=\{tH\colon t\in G\}$ admits a canonical action of $G$ on the left that induces an action $\sigma$ of $G$ on $C_0(G/H),$ $\sigma_t(f)(rH)=f(\tmu rH).$
Consider a system of imprimitivity for $\cB$ over $G/H,$ $\langle T,P\rangle,$ as in \cite[VIII 18.7]{FlDr88}.
Then $T\colon \cB\to \bB(X)$ is a non degenerate *-representation and $P$ a regular $X$-projection-valued Borel measure on $G/H$ such that $T_bP(W)=P(tW)T_b$ for all $b\in B_t,$ $t\in G$ and Borel subset $W$ of $G/H.$

The spectral integral of $P$ is the *-representation $\varphi_P\colon C_0(G/H)\to \bB(X)$ such that $\varphi_P(a):=\int_{G/H}a(z)\, dP(z).$
By Stone's Theorem \cite[VI 10.10]{FlDr88}, non degenerate *-representations of $C_0(G/H)$ on $X$ correspond to spectral integrals of regular $X$-projection-valued Borel measures on $G/H.$

As shown in \cite[VIII 18.8]{FlDr88}, given non degenerate *-representations $S\colon \cB\to \bB(Y)$ and $\varphi\colon C_0(G/H)\to \bB(Y),$ if $P$ is the $Y$-projection-valued Borel measure on $G/H$ determined by $\varphi,$ then $\langle T,P\rangle$ is a system of imprimitivity for $\cB$ over $G/H$ if and only if $T_b\varphi(a)=\varphi(\sigma_t(a))T_b$ for all $b\in B_t,$ $t\in G$ and $a\in C_0(G/H).$

In general, it is not possible to describe a *-representation of a commutative C*-algebra $\varphi\colon C_0(M)\to \bB(X_A)$ as a spectral integral because $\bB(X_A)$ may not have enough projections to do this; as it is the case of the identity map $\varphi\colon C_0([0,1])\to \bB(C_0([0,1])).$

\begin{definition}
   A system of imprimitivity for $\cB$ over $G/H$ is a pair $\langle T,\psi\rangle$ consisting of non degenerate *-representations $T\colon \cB\to \bB(X_A)$ and $\psi\colon C_0(G/H)\to \bB(X_A)$ such that $T_b\psi(a)=\psi(\sigma_t(a))T_b$ for all $b\in B_t,$ $t\in G$ and $a\in C_0(G/H).$
\end{definition}

\begin{proposition}\label{prop:construction of psiHB}
	Given a Fell bundle $\cB=\{B_t\}_{t\in G}$ and $H \sbgp G,$ there is a unique *-representation $\psi^{H\cB}\colon C_0(G/H)\to \bB(L^2_H(\cB))$ such that $\psi^{H\cB}(C_0(G/H))C_c(\cB)\subset C_c(\cB)$ and $\overline{\psi^{H\cB}}(f)g(t)=f(tH)g(t)$ for all $f\in C_b(G/H),$ $g\in C_c(\cB)$ and $t\in G.$ 
	Moreover, $\langle \regrep{H\cB},\psi^{H\cB}\rangle$ is a system of imprimitivity for $\cB$ over $G/H.$
	\begin{proof}
		Given $f\in C_b(G/H)$ and $g\in C_c(\cB)$ define $fg\in C_c(\cB)$ by $fg(t):=f(tH)g(t).$
		For all $f\in C_b(G/H),$ $g,h\in C_c(\cB)$ and $t\in H$ we have
		\begin{align*}
		p((fg)^**h)(t) & = \Delta_G(t)^{1/2}\Delta_H(t)^{-1/2} \int_G (fg)^*(r)h(r^{-1}t)\, dr\\
		& = \Delta_G(t)^{1/2}\Delta_H(t)^{-1/2} \int_G g^*(r)\overline{f(r^{-1}tH)}h(rt)\, dr 
		= p(g^** (f^*h))(t).
		\end{align*}
		Thus $p((fg)^**h) = p(g^** (f^*h)).$
		
		Set $f':=(\|f\|_\infty^2 - |f|^2)^{1/2}$ and take a *-representation $S\colon \cB_H\to \bB(Z).$
		For all $\xi\in Z$ and $g\in C_c(\cB)$ it follows that
		\begin{align*}
		\|f\|^2_\infty \langle \xi,\intform{S}_{p(g^* * g)}\xi\rangle - \langle \xi,\intform{S}_{p((fg)^* * (fg))}\xi\rangle = \langle \xi,\intform{S}_{p( [f'g]^**[f'g] )}\rangle\geq 0;
		\end{align*}
		which implies $ \langle fg,fg\rangle_{L^2_H(\cB)}\leq \|f\|^2_\infty \langle g,g\rangle_{L^2_H(\cB)}.$
		Consequently, for every $f\in C_b(G/H)$ there exists a unique bounded operator $ \overline{\psi}^{H\cB}(f)\colon L^2_H(\cB)\to L^2_H(\cB)$ such that for every $g\in C_c(\cB)$ the identity $\overline{\psi}^{H\cB}(f)(g)=fg$ obtains.
		
		We already showed that $p((fg)^* * h)=p(g^**(f^*g))$ for all $f\in C_b(G/H)$ and $g,h\in C_c(\cB).$
		Thus $\langle \overline{\psi}^{H\cB}(f) u,v\rangle_p  =\langle u,\overline{\psi}^{H\cB}(f^*)v\rangle_p$ for all $u,v\in C_c(\cB)$ and it follows that $\overline{\psi}^{H\cB}(f)$ is adjointable with adjoint $\psi^{H\cB}(f^*).$
		We leave to the reader the verification of the fact that $\overline{\psi}^{H\cB}\colon C_b(G/H)\to \bB(L^2_H(\cB))$ is a unital *-representation.
		
		Take $f\in C_c(\cB)$ and let $g\in C_c(G/H)$ be such that $0\leq g\leq 1$ and $g(tH)=1$ if $t\in \supp(f).$
		Then $\overline{\psi}^{H\cB}(g)f=f$ and this implies $C_c(\cB)\subset \overline{\psi}^{H\cB}(C_c(G/H))C_c(\cB).$
		The restriction $\psi^{H\cB}:=\overline{\psi}^{H\cB}|_{C_0(G/H)}$ is then non degenerate and  $\overline{\psi}^{H\cB}$ must be the unique *-homomorphism $\overline{\psi^{H\cB}}\colon C_b(G/H)\to \bB(L^2_H(\cB))$ extending $\psi^{H\cB}.$
		
		Take $t\in G,$ $b\in B_t,$ $f\in C_b(G/H)$ and $g\in C_c(\cB).$
		For all $r\in G$ it follows that
		\begin{equation*}
		(bfg)(r)=b (fg)(\tmu r) = b f(\tmu r H) g(\tmu r) = \sigma_t(f)(rH) (bg)(r) = (\sigma_t(f) (bg))(r) ;
		\end{equation*}
		implying that $\regrep{H \cB}_b \overline{\psi}^{H\cB}(f) (g) = \overline{\psi}^{H\cB}(\sigma_t(f)) \regrep{H\cB}_b(g).$
		Thus the proof follows by a density argument.
		\end{proof}
		\end{proposition}
	
Given a non degenerate *-representation $R\colon \cB_H\to \bB(X_A),$ the pair 
\begin{equation*}
\langle \regrep{H\cB}\otimes_R 1,\psi^{H\cB}\otimes_R 1\rangle \equiv \langle \Ind_H^\cB(R),\psi^{H\cB}\otimes_R 1\rangle
\end{equation*}
is a system of imprimitivity for $\cB$ over $G/H;$ which we call an induced system. 
	
There is still another way of constructing systems of imprimitivity.
		
\begin{example}\label{example:system of imprimitivity}
	Consider a non degenerate *-representation $S\colon \cB\to \bB(X)$ and a unitary representation $U\colon H\to \bB(Y).$
	Let $\langle \regrep{HG},\psi^{HG}\rangle$ be the system of imprimitivity given by the Proposition above for (the trivial bundle over) $G$ and define $\Psi\colon C_0(G/H)\to \bB(X\otimes (L^2_H(G)\otimes_U Y))$ by $\Psi(a)=1\otimes (\psi^{HG}(a)\otimes_U 1).$
	Then $\langle S\otimes \Ind_H^G(U),\Psi\rangle$ is a system of imprimitivity for $\cB$ over $G/H.$ 
\end{example}

In essence, the proposition below characterises the integrated forms of Fell's induced systems of imprimitivity, as constructed in \cite[XI 14.3 pp-1181]{FlDr88}.
		
\begin{proposition}
	Let $\cB=\{B_t\}_{t\in G}$ be a Fell bundle, $H \sbgp G$ and $S\colon \cB_H\to \bB(X)$ a *-representation.
	If $P$ is the projection-valued measure induced by $S,$ as described \cite[XI 14.3]{FlDr88}, then the spectral integral of $P$ is $\psi^S:=\psi^{H\cB}\otimes_S 1\colon C_0(G/H)\to \bB(L^2_H(\cB)\otimes_S X).$
			\begin{proof}
				It is implicit in the claim above that we are identifying the concrete and the abstract induction processes.
				The induced projection-valued measure on $G/H$ is constructed in \cite{FlDr88} using the concrete induction process, so we need to specify the unitary equivalence between the concretely and abstractly induced spaces.
				
				We start (exactly) as in \cite[XI 9]{FlDr88}, so we fix a continuous and everywhere positive $H-$rho function $\rho$ on $G$ (see \cite[III 14.5]{FlDr88}) and denote $\rho^{\#}$ the regular Borel measure on $G/H$ constructed from $\rho$ (as in \cite[III 13.10]{FlDr88}).
				Notice we need to adapt Fell's formulas because our inner products are linear in the second variable.
				
				Let $\cY=\{Y_\alpha\}_{\alpha\in G/H}$ be the Hilbert bundle over $G/H$ induced by $S$ and let  $L^2(\rho^{\#},\cY)$ be the corresponding cross sectional Hilbert space (see \cite[XI 9.7]{FlDr88} and \cite[II 15.12]{FlDr88}).
				Notice that the space $\mathcal{X}(V)$ of \cite[XI 9.8]{FlDr88} is precisely $L^2_H(\cB)\otimes_T X$ and that the unitary operator $E\colon L^2_H(\cB)\otimes_S Y\to L^2(\rho^{\#},\cY)$ of \cite[XI 9.8]{FlDr88} intertwines the abstractly and concretely induced representations.
				The identity $E( \psi^T(f) (g\otimes_{\intform{T}} \xi) ) = fE(g\otimes_{\intform{T}}\xi)$ follows at once for all $f\in C_b(G/H)$ and every elementary tensor $g\otimes_{\intform{T}}\xi.$
				So $E$ intertwines the spectral integral of $P$ and $\psi^T.$
				\end{proof}
				\end{proposition}

It is now natural to call $\langle \regrep{H\cB},\psi^{H\cB}\rangle $ the \textit{universal} system of imprimitivity for $\cB$ over $G/H.$
Our study of systems of imprimitivity will be continued in Section~\ref{ssec:a weak imprimitivity theorem}.
We suggest to pause here and read the statement of Corollary~\ref{cor:reps that are part of s system of imprimitivity} to get a feeling of the kind of results we are after.

In the proof of \cite[Proposition 19.3]{Exlibro} Exel constructs a conditional expectation $E\colon C^*_\red(\cB)\to B_e$ for every Fell bundle $\cB=\{B_t\}_{t\in G}$ over a discrete group.
Putting $H:=\{e\},$ noticing that $B_e\equiv C^*(\cB_H)$  and recalling that $C^*_\red(\cB)\equiv C^*_H(\cB),$  we get a conditional expectation $E\colon C^*_H(\cB)\to C^*(\cB_H).$
One can use the universal system of imprimitivity to construct these kind of conditional expectation (as follows).

\begin{theorem}\label{thm:contiditional expectation open subgroup}
	Let $\cB$ be a Fell bundle over $G$ and $H$ an open subgroup of $G.$
	If we see the elements of $C_c(\cB_H)$ as those $f\in C_c(\cB)$ vanishing outside $H,$ then we get an inclusion of *-algebras $C_c(\cB_H)\subset C_c(\cB)$ which can be extended to get inclusions
	\begin{align*}
    C^*(\cB_H)& \subset C^*(\cB) & C^*(\cB_H) & \subset C^*_H(\cB) & C^*_\red(\cB_H)&\subset C^*_\red(\cB). 
	\end{align*}
	Besides, the map $p\colon C_c(\cB)\to C_c(\cB_H)$ $ f\mapsto f|_H$ is the restriction of (unique) conditional expectations
	\begin{align*}
	P_G\colon& C^*(\cB)\to C^*(\cB_H) & P_H\colon & C^*_H(\cB) \to C^*(\cB_H) & P_\red &\colon C^*_\red(\cB)\to C^*_\red(\cB_H). 
	\end{align*}
	
	Moreover, the process of induction of *-representations of $C^*(\cB_H)$ to *-representations of $C^*(\cB)$ via $P_G$ is induction via the *-homomorphism $\intregrep{H\cB}\colon C^*(\cB)\to \bB(L^2_H(\cB)).$
\end{theorem}
	\begin{proof}
		The inclusion $C_c(\cB_H)\subset C_c(\cB)$ is $\|\ \|_1-$isometric and gives an inclusion of Banach *-algebras $L^1(\cB_H)\subset L^1(\cB),$ which in turn induces a *-homomorphism $\mu\colon C^*(\cB_H)\to C^*(\cB).$
		
		For all $f,g\in C_c(\cB_H)$ we have $\langle f,g\rangle_{L^2_H(\cB)}=p(f^**g)=f^**g.$
		Besides, if given $f,u\in C_c(\cB_H)$ we compute the element $fu\in C_c(\cB_H)$ using the action of $C^*(\cB_H)$ in $L^2_H(\cB)$ and the product of $C^*(\cB_H),$ we get the same element of $C_c(\cB_H).$
		Thus, the closure of $C_c(\cB_H)$ in $L^2_H(\cB)$ is $C^*(\cB_H).$
		
		For every $b\in \cB_H$ the restriction of $\regrep{H\cB}_b$ to $C^*(\cB_H)\subset L^2_H(\cB)$ is exactly $\regrep{\cB_H}_b.$
		Passing to integrated forms this implies that, for every $f\in C_c(\cB_H),$ the restriction of $\intregrep{H\cB}_{\mu(f)}$ to $C^*(\cB_H)$ is $\intregrep{\cB_H}_f.$
		Thus
		\begin{equation*}
     \|f\|_{C^*(\cB_H)}\equiv \| \intregrep{\cB_H}_f \|\leq \|\intregrep{H\cB}_{\mu(f)}\|\equiv \|\mu(f)\|_{C^*_H(\cB)}\equiv \|q^\cB_H(\mu(f))\|\leq \|\mu(f)\|_{C^*(\cB)}\leq \|f\|_{C^*(\cB_H)},
	\end{equation*}
	and we get that both $\mu$ and $\intregrep{H\cB}\circ \mu$ are isometries.
	It is then clear that the inclusions $C_c(\cB_H)\subset C^*(\cB)$ and $C_c(\cB_H)\subset C^*_H(\cB)$ can be extended to get inclusions $C^*(\cB_H)\subset C^*(\cB)$ and $C^*(\cB_H)\subset C^*_H(\cB).$
		
	To prove that $C_c(\cB_H)\subset C^*_\red(\cB)$ extends to an inclusion $C^*_\red(\cB_H)\subset C^*_\red(\cB)$ it suffices to show that $\| \intregrepEN{\cB_H}_f \|= \|\intregrepEN{\cB}_f\|$ for all $f\in C_c(\cB_H);$ which we do by using Theorem~\ref{thm:EXNG absorption principle} multiple times.
	Take a non degenerate *-representation $T\colon \cB\to \bB(X)$ with $T|_{B_e}$ faithful and consider the *-representation $\regrepEN{G}\otimes T,$ $\regrepEN{G}$ being the left regular representation of $G.$
	The coset space $H\setminus G :=\{Hr\colon r\in G\}$ gives a Hilbert space decomposition $L^2(G)=\bigoplus_{U\in H\setminus G}L^2(U)$ into $\regrepEN{G}|_H$-invariant subspaces.
	Moreover, the restriction of $\regrepEN{G}|_H$ to each one of the subspaces $L^2(U)$ is unitary equivalent to the left regular representation of $H,$ $\regrepEN{H}.$
	Then $(\regrepEN{G}\otimes T)|_{\cB_H} = (\regrepEN{G}|_H)\otimes (T|_{\cB_H})$ is unitary equivalent to a direct sum of $\# (H\setminus G)$ copies of $\regrepEN{H}\otimes (T|_{\cB_H}).$
	Hence, for every $f\in C_c(\cB_H),$
	\begin{equation*}
     \| \intregrepEN{\cB}_f \| 
     = \|(\regrepEN{G}\intform{\otimes}T)_f\|
     = \|[(\regrepEN{G}|_H)\otimes (T|_{\cB_H})]_f\|
     = \| [\regrepEN{H}\otimes (T|_{\cB_H})]_f \|
     = \| \intregrepEN{\cB_H}_f \|.
	\end{equation*}
	
	Our next goal is to construct both $P_H$ and $P_\red$ at the same time.
		For this we let $K$ be either $H$ or $\{e\}.$
		When $K=\{e\}$ and we write $P_{K}$ we mean $P_\red$ (recall that $C^*_\red(\cB)=C^*_{\{e\}}(\cB)$).
		In case $K=H,$ we view the module $L^2_K(\cB_H)=C^*(\cB_H)$ as a closed subspace of $L^2_K(\cB)$ as indicated before in this proof; while for $K=\{e\}$ the module $L^2_K(\cB_H)=L^2_e(\cB_H)$ may be regarded as the closure of $C_c(\cB_H)$ in $L^2_K(\cB).$
		In this last case $L^2_K(\cB_H)$ is a submodule of $L^2_K(\cB).$
		
		By \eqref{equ:ref properties of p}, to prove $p$ can be extended to a conditional expectation $P_K\colon C^*_K(\cB)\to C^*_K(\cB_H)$ it suffices to show that $p$ is contractive with respect to the C*-norms; which we do by adapting the ideas of \cite[Lemma 17.8]{Exlibro}.
		
		Fix $f\in C_c(\cB).$ 
		Take $h\in C_c(\cB_H)\subset L^2_K(\cB_H)\subset L^2_K(\cB)$ and an approximate unit $\{g_i\}_{i\in I}\subset C_c(\cB_H)$ of $L^1(\cB_H),$ like that of \cite[VIII 5.11]{FlDr88}.
		If $1_H\in C_0(G/K)$ is the indicator function of the projection of $H$ into $G/K,$ then the norm of $\psi^{K\cB}(1_H)\intregrep{K\cB}_f \intregrep{K\cB}_{g_i} h\in C_c(\cB_H)\subset L^2_K(\cB)$ is not greater than $\|\intregrep{K\cB}_f\|\|h\|_{L^2_K(\cB)}=\|f\|^\cB_K\|h\|_{L^2_K(\cB_H)}.$
		
		Now, given any $t\in H:$
		\begin{equation*}
		\left(\psi^{K\cB}(1_H)\intregrep{K\cB}_f \intregrep{K\cB}_{g_i} h\right)(t)
		= (f*g_i*h)(t)
		= \int_G\int_G f(r)g_i(r^{-1}s)h(s^{-1}t)\, drds.
		\end{equation*}
		For $f(r)g_i(r^{-1}s)h(s^{-1}t)$ to be non zero we must have $s^{-1}t\in H$ and $r^{-1}s\in H,$ which implies $r,s\in H$ because $t\in H.$
		Then 
		\begin{equation*}
		\left(\psi^{K\cB}(1_H)\intregrep{K\cB}_f \intregrep{K\cB}_{g_i} h\right)(t)
		= \int_H\int_H f(r)g_i(r^{-1}s)h(s^{-1}t)\, drds
		= (p(f)*g_i*h)(t)
		\end{equation*}
		and we get $\psi^{K\cB}(1_H)\intregrep{K\cB}_f \intregrep{K\cB}_{g_i} h = \intregrep{K\cB_H}_{p(f)*g_i}h.$
		Since the net $\{p(f)*g_i\}_{i\in I}$ converges to $p(f)$ in $C^*(\cB_H),$  $\|p(f)\|^{\cB_H}_K=\lim_i\|\intregrep{K\cB_H}_{p(f)*g_i}  \|\leq \|f\|^\cB_K;$ which is all we needed to justify the existence of the conditional expectation $P_K.$ 	
		The construction of $P_G$ is very easy: just set $P_G:=P_H\circ q^{\cB}_H.$
		
		Recall that to induce *-representations using $P_G$ one looks at $C^*(\cB)$ as a right $C^*(\cB_H)-$(pre)Hilbert module with inner product $\langle f,g\rangle = P_G(f^**g).$
		For $f,g\in C_c(\cB),$ $P_G(f^**g)= p(f^**g)=\langle f,g\rangle_{L^2_H(\cB)}.$
		Besides, for every $u\in C_c(\cB_H),$ the product $f*u$ computed in $C_c(\cB)$ equals the element $fu$ of \eqref{equ: action to induce} (because $\Delta_G|_H=\Delta_H$).
		Then the Hilbert module induced by $P_G$ is (unitary equivalent to) $L^2_H(\cB).$
		
		We have two *-representations of $C^*(\cB)$ in $L^2_H(\cB)$: $\intregrep{H\cB}$ and the one induced by $P_G,$ $\theta,$ which is characterized by the fact that $\theta(f)g=f*g$ for all $f,g\in C_c(\cB).$
		But we know $f*g=\intregrep{H\cB}_fg,$ so $\theta =\intregrep{H\cB}$ and we get that induction via $P_G$ is induction via $\intregrep{H\cB}\colon C^*(\cB)\to \bB(L^2_H(\cB)).$
  \end{proof}

\subsection{Induction in stages}\label{sec:induction in stages}

In~\cite[XI 12.15]{FlDr88} Fell shows that *-representations of Banach *-algebraic bundles may be induced in stages.
Rephrasing Fell's statement and using his ideas we get the following.

\begin{theorem}[Induction in stages, the Rieffel way]\label{thm:induction in stages Fell bundles}
	Let $\cB=\{B_t\}_{t\in G}$ be a Fell bundle and consider subgroups $H \sbgp K \sbgp G.$
	Then $\Ind_K^\cB( \regrep{H\cB_K} )\equiv \regrep{K\cB}\otimes_{\regrep{H\cB_K}}1 $ is unitary equivalent to $\regrep{H\cB}.$
	Moreover, the unitary equivalence is implemented by the unique adjointable operator
	\begin{equation}\label{equ:unitary for inducing in stages}
	U\colon L^2_K(\cB)\otimes_{\intregrep{H\cB_K}} L^2_H(\cB_K)\to  L^2_H(\cB)
	\end{equation}
	that maps $f\otimes_{\intregrep{H\cB_K}} u $ to $fu$ for all $f\in C_c(\cB)$ and $u\in C_c(\cB_K).$
	\begin{proof}
		The generalized restriction maps will be denoted indicating the groups as follows:
		\begin{align*}
		p^K_H & \colon C_c(\cB_K)\to C_c(\cB_H)&  p^K_H(f)(t)&=\Delta_K(t)^{1/2}\Delta_H(t)^{-1/2}f(t).
		\end{align*}
		A direct computation shows that $p^K_H\circ p^G_K = p^G_H.$
		
		For convenience we write $\pi$ instead of $\intregrep{H\cB_K}$.
		Given any $f,g\in C_c(\cB)\subset L^2_K(\cB)$ and $u,v\in C_c(\cB_K)\subset L^2_H(\cB_K)$ we have:
		\begin{multline}\label{equ:inner product identity of induction in stages}
		\langle f\otimes_\pi u,g\otimes_\pi v\rangle
		= \langle u, \pi(\langle f,g\rangle )v\rangle
		= \langle u,\pi(p^G_K(f^**g))v\rangle
		= \langle u, p^G_K(f^**g)*v\rangle =  \\
	= p^K_H( u^**p^G_K(f^**g)*v  )  
		= p^K_H( p^G_K((fu)^**(gv)))
		= p^G_H((fu)^**(gv))
		= \langle fu,gv\rangle.
		\end{multline}
		
		Thus there exists a unique inner product-preserving linear isometry 
		\begin{equation*}
		U\colon L^2_K(\cB)\otimes_{\pi} L^2_H(\cB_K)\to  L^2_H(\cB)
		\end{equation*}
		mapping $f\otimes_\pi u$ to $fu.$
		If we go to the two paragraph following  \eqref{equ:second inner product double induction} and consider the particular case $\cX=\cB$ and $H=K,$ we get that $\spn \{fu\colon f\in C_c(\cB),\ u\in C_c(\cB_K)\}$ is dense in $C_c(\cB)$ in the inductive limit topology.
		Thus $U$ is a unitary operator.
		
		Finally, for all $b\in \cB,$ $f\in C_c(\cB)$ and $u\in C_c(\cB_K)$ we have
		\begin{equation*}
		U^* \regrep{H\cB}_b U (f\otimes_\pi u)
		= U^* (b(fu))
		= U^* ((bf)u)
		= (bf) \otimes_\pi u 
		= \Ind_K^\cB(\regrep{H\cB_K})_b (f \otimes_\pi u).
		\end{equation*}
		Thus the proof follows by the density of elementary tensor products.
	\end{proof}
\end{theorem}

We now reproduce the statement of \cite[XI 12.15]{FlDr88} and explain Fell's proof in our notation.

\begin{corollary}[Induction in stages, the Fell way]\label{cor:induction in stages, the Fell way}
  Let $\cB=\{B_t\}_{t\in G}$ be Banach *-algebraic bunde, $H$ and $K$ closed subgroups of $G$ such that $H\subseteq K$ and $T\colon\cB_H\to \bB(X)$ a *-representation.
  Then the following are equivalent:
  \begin{enumerate}[(1)]
  	\item\label{item:T is Bpositive} $T$ is $\cB-$positive.
  	\item\label{item:T is BK positive and indT is B positive} $T$ is $\cB_K-$positive and $\Ind_H^{\cB_K}(T)$ is $\cB-$positive.
  \end{enumerate}
  If the conditions above hold, as it is the case if $\cB$ is a Fell bundle, $\Ind_K^\cB(T)$ is unitary equivalent to $\Ind_K^\cB(\Ind_H^{\cB_K}(T))$.
  \begin{proof}
  If $\cB$ is a Fell bundle, then so are $\cB_K$ and $\cB_H,$ and all the positivity conditions hold (Theorem~\ref{thm:characterization of positive representations}).
  Using Theorem~\ref{thm:induction in stages Fell bundles} and writing unitary equivalences as equalities we get
   \begin{equation*}
   \Ind_K^\cB(\Ind_H^{\cB_K}(T))= \regrep{K\cB}\otimes_{\regrep{H\cB_K}\otimes_T 1}1
   = \left(\regrep{K\cB}\otimes_{\regrep{H\cB_K}}1\right)\otimes_T 1
   = \Ind_K^\cB(\regrep{H\cB_K})\otimes_T 1
   =\Ind_H^\cB(T).
   \end{equation*}
   
   We now deal with the general case ($\cB$ is no longer a Fell bundle).
   By Theorem~\ref{thm:characterization of positive representations}, condition \eqref{item:T is Bpositive} is equivalent to $T|_{B_e}$ being $\cB-$positive.
   
   Assume \eqref{item:T is BK positive and indT is B positive} holds and let $S$ be $\Ind_H^{\cB_K}(T).$
   By \eqref{equ:inducing increases norm}, for all $b\in B_e,$ 
   \begin{equation*}
   \|T_b\|\leq \|\Ind_H^{\cB_K}(T)_b\|=\|S_b\|\leq \|\Ind_K^\cB(S)_b\|.
   \end{equation*}
   Then $T|_{B_e}$ is weakly contained in $\Ind_H^\cB(S)|_{B_e}$ and, consequently, any inner product of the form $\langle \xi,T_{b^*b}\xi\rangle$ ($\xi\in X$ and $b\in \cB$) can be approximated by others of the form $\langle \eta,\Ind_K^\cB(S)_{b^*b}\eta\rangle=\|\Ind_K^\cB(S)_{b}\eta\|^2.$
   Thus $T|_{B_e}$ is $\cB-$positive and \eqref{item:T is BK positive and indT is B positive}$\Rightarrow$\eqref{item:T is Bpositive}.
   
   Now assume that $T$ is $\cB-$positive, consider the bundle $C^*-$completion $\cC$ of $\cB$ and the canonical map $\rho\colon \cB\to \cC.$
   By Proposition~\ref{prop:can change a banach algebraic bundle by its completion}, there is a unique *-representation $T^\rho\colon \cC_H\to \bB(X)$ such that $T=T^\rho\circ (\rho|_{\cB_H}).$
   Besides, $T$ is $\cB_K-$positive (i.e. $T|_{B_e}$ is $\cB_K-$positive) and we may construct the induced *-representations $\Ind_H^{\cB_K}(T)$ and $\Ind_H^{\cC_K}(T^\rho),$ acting on Hilbert spaces $Y$ and $Z,$ respectively.
   We claim that $\Ind_H^{\cB_K}(T)$ is (unitary equivalent to) $\Ind_H^{\cC_K}(T^\rho)\circ (\rho|_{\cB_K}),$ which implies $\Ind_H^{\cB_K}(T)$ is $\cB-$positive (by Theorem~\ref{thm:characterization of positive representations}).
   
   The Hilbert spaces $Y$ and $Z$ are the closed linear span of tensors $f\otimes_T \xi$ and $\rho\circ f\otimes_{T^\rho}\xi,$ respectively, with $f\in C_c(\cB)$ and $\xi\in X.$
   Besides, there is a unique unitary $U\colon Y\to Z$ such that $U(f\otimes_T\xi)=f\circ \rho\otimes_{T^\rho}\xi $ (see the proof of Proposition~\ref{prop:can change a banach algebraic bundle by its completion}).
   For all $b\in \cB,$ $f\in C_c(\cB)$ and $\xi\in X$ we have
   \begin{equation*}
   U^* \Ind_H^{\cC_K}(T^\rho)_{\rho(b)}U(f\otimes_T\xi)
   =U^*(\rho\circ (bf) \otimes_{T^\rho}\xi)
   = bf\otimes_T\xi
   =\Ind_H^{\cB_K}(T)_b(f\otimes_T\xi).
   \end{equation*}
   Hence, $\Ind_H^{\cB_K}(T)$ is $\cB-$positive and $\Ind_H^{\cB_K}(T)^\rho=\Ind_H^{\cC_K}(T^\rho).$
   
   Induction preserves unitary equivalence classes of *-representations thus, up to unitary equivalences,
   \begin{multline*}
     \Ind_H^\cB(T)
      = \Ind_H^\cC(T^\rho)\circ \rho
      = \Ind_K^\cC(\Ind_H^{\cC_K}(T^\rho))\circ \rho
      = \Ind_K^\cC(\Ind_H^{\cB_K}(T)^\rho)\circ \rho
      = \Ind_K^\cB(\Ind_H^{\cB_K}(T));
   \end{multline*}
   which completes the proof.
\end{proof}
\end{corollary}

\subsection{Weak containment}
Exel and Ng introduced the notion of amenable Fell bundle in \cite{ExNg}.
As already mentioned in the Introduction, we prefer to use the term ``weak containment'' (WCP) over ``amenable''.
We do use the word ``amenable'' for groups.

\begin{definition}
 Let $\cB=\{B_t\}_{t\in G}$ be a Fell bundle and $H \sbgp K \sbgp G.$
 We say that $\cB$ has the
 \begin{itemize}
 	\item $HK-$WCP if $q^\cB_{HK}\colon C^*_K(\cB)\to C^*_H(\cB)$ is an isomorphism (see Remark~\ref{rmk:composition of q maps}).
 	\item $H-$WCP if $q^\cB_H\colon C^*(\cB)\to C^*_H(\cB)$ is an isomorphism (this is the $HG-$WCP).
 	\item WCP if $\cB$ has the $\{e\}G-$WCP.
 \end{itemize}
\end{definition}

\begin{remark}
By Corollary~\ref{cor:norm of reg rep is induced norm}, $\cB$ has the $HK-$WCP if and only if $\|\intregrep{H\cB}_f\|=\|\intregrep{K\cB}_f\|$ for all $f\in L^1(\cB).$
Then $\cB$ has the $H-$WCP if and only if $\intregrep{H\cB}\colon C^*(\cB)\to \bB(L^2_H(\cB))$ is faithful.
Besides, $\cB$ has the WCP if and only if $\intregrepEN{\cB}\equiv\intregrep{\{e\}\cB}\colon C^*(\cB)\to \bB(L^2_e(\cB))$ is faithful (i.e. $C^*(\cB)=C^*_\red(\cB)$ canonically).
\end{remark}

\begin{remark}
	The composition $q^\cB_H=q^\cB_K\circ q^\cB_{KH}$ is faithful if and only if both $q^\cB_K$ and $q^\cB_{KH}$ are faithful.
	In other words, $\cB$ has the $H-$WCP if and only if it has both the $K-$WCP and the $HK-$WCP.
\end{remark}

There are two reasonable notions of $H-$WCP: the upward ($HG-$WCP) and the downward ($\{e\}H-$WCP).
We won't pay much attention to the second one, but we do have some things to say about it.

\begin{proposition}\label{prop:inner amenable subgroup and downward WCP}
	Let $\cB=\{B_t\}_{t\in G}$ be a Fell bundle and let $H$ be an open inner amenable subgroup of $G.$
	If $C^*_\red(\cB)$ is nuclear, then  so it is $C^*_\red(\cB_H),$ $\cB_H$ has the WCP and $C^*_H(\cB)=C^*_\red(\cB)$ (canonically).
\end{proposition}
\begin{proof}
 Theorem~\ref{thm:contiditional expectation open subgroup} gives a conditional expectation $P_\red \colon C^*_\red(\cB)\to C^*_\red(\cB_H),$ then $C^*_\red(\cB_H)$ is nuclear and, by Corollary~\ref{cor:inner amenability and nuclearity}, $\cB_H$ has the WCP.
 So, if we take a faithful non degenerate *-representation $T\colon B_e\to \bB(X),$ $\intform{\Ind}_e^{\cB_H}(T)$ weakly contains any non degenerate *-representation of $\cB_H.$ 
  By induction in stages, $\Ind_e^\cB(T) = \Ind_H^\cB(\Ind_e^{\cB_H}(T)).$
  Hence, by Remark~\ref{rem:induced and faithful integrated form}, we may think of both $C^*_H(\cB)$ and $C^*_\red(\cB)$ as the closure of $\Ind_e^\cB(T)(C_c(\cB)).$
\end{proof}

\begin{corollary}\label{cor:isomorphisms of q maps}
	Let $\cB=\{B_t\}_{t\in G}$ be a Fell bundle and consider subgroups $H \sbgp K \sbgp G.$
	If $\cB_K$ has the $H-$WCP, then $\cB$ has the $HK-$WCP.
\end{corollary}
\begin{proof}
    Theorem~\ref{thm:induction in stages Fell bundles}, combined with the assumption that $\intregrep{H\cB_K}\colon C^*(\cB_K)\to\bB(L^2_H(\cB_K))$ is faithful, implies that that  for all $f\in L^1(\cB)$ we have $\|\intregrep{K\cB}_f\|= \|(\intregrep{K\cB}_f\otimes_{\intregrep{H\cB_K}}1)_f\|=\|\intregrep{H\cB}_f\|.$
\end{proof}

We seriously doubt that the converse of the Corollary above holds in full generality, but it does if both $H$ and $K$ are normal in $G.$
The proof of this claim is not simple and it will be completed in Section~\ref{sec:weak containment with respect to subgroups} after a sequence of (what we consider to be) some interesting intermediate results.

\section{An absorption principle}\label{sec:absorption principle}

One may view the (unitary) representation theory of any group $G$ as the (*-)representation theory of the trivial Fell bundle over $G$ with constant fibre $\bC,$ which we have denoted $\cT_G=\{\bC\delta_t\}_{t\in G}.$
If $H \sbgp  G,$ then $(\cT_G)_H\equiv \cT_H$ and the induction of representations from $H$ to $G$ may be regarded as the induction from $C^*(H):=C^*(\cT_H)$ to $C^*(G):=C^*(\cT_G)$ via the $C^*(H)-$Hilbert module $L^2_H(G):=L^2_H(\cT_G).$

\begin{definition}
 The $H-$group C*-algebra of $G$ is $C^*_H(G)\equiv C^*_H(\cT_G)\subset \bB(L^2_H(G))$ and the $H-$regular representation of $G$ is 
\begin{equation*}
 \regrep{H}\colon G\to \bB(L^2_H(G)),\ t\mapsto \regrep{H\cT_G}_{1\delta_t}.
\end{equation*}
Given a *-representation $U\colon H\to \bB(Y_A),$ the induced representation 
\begin{equation*}
 \Ind_H^G(U)\colon G\to \bB(L^2_H(G)\otimes_{\intform{U}}Y_A)
\end{equation*}
is given by $\Ind_H^G(U)_t =\regrep{H}_t\otimes_{\intform{U}}1,$ for all $t\in G.$
\end{definition}

\begin{remark}
  If we say $V\colon G\to \bB(Y_A)$ is a non degenerate *-representation we mean that $V$ is a unitary representation.
\end{remark}

The absorption principle we present below emerged during our attempt to generalise Fell's and Exel-Ng's principles (as discussed in the Introduction).
The computations revealed that the ``moving around'' technique of \cite[pp 515]{ExNg} is appropriate when the inducing subgroup $H$ is stable by conjugation (normal).
In the general case, conjugation ``moves $H$ around'' and this is the reason why the conjugated groups $tH\tmu$ appear in our principle. 

\begin{theorem}\label{thm:Fells absorption principle I}
 Let $\cB=\{B_t\}_{t\in G}$ be a Fell bundle,  $H \sbgp G$ and consider non degenerate *-representations $T\colon \cB\to \bB(Y_A)$ and $V\colon H\to \bB(Z_C).$
 If for each $t\in G$  we denote ${}_tV$ the conjugated representation $tH\tmu \to \bB(Z_C), \ r\mapsto V_{\tmu r t},$ then 
 \begin{equation}\label{equ:thesis inequality for Fell absorption I}
 \|T\intform{\otimes} \Ind_H^G(V)_f  \|= \sup_{t\in G}\|\intform{\Ind}_{tH\tmu }^{\cB}(T|_{\cB_{tH\tmu}}\otimes {}_tV)_f\| \qquad \forall\ f\in C^*(\cB).
 \end{equation}
 
 In particular, if both $Y_A$ and $Z_C$ are Hilbert spaces, then $T\otimes \Ind_H^G(V)\sim \{\Ind_{tH\tmu }^{\cB}(T|_{\cB_{tH\tmu}}\otimes {}_tV) \}_{t\in G}.$
\end{theorem}

 \begin{proof}
 	For convenience we introduce the following notation 
 	\begin{align*}
 	{}^tH&\equiv tH\tmu & {}_{t|}T & \equiv T|_{\cB_{tH\tmu}} & U&\equiv \Ind_H^G(V).
 	\end{align*}
 	
 	The Hilbert module induced by (the integrated form of) $V$ is $W_C:=L^2_{H}(G)\otimes_{V} Z_C$ and we regard the $A\otimes C-$Hilbert module $Y_A\otimes^G W_C:=\ell^2(G)\otimes (Y_A\otimes W_C)$ as an $\ell^2-$direct sum of $\# G$ copies of $Y_A\otimes W_C$.
 	Similarly, the direct sum of $\#G$ copies of $T\otimes U\equiv T\otimes \Ind_H^\cB(V)$ is 
 	\begin{align*}
 	T\otimes^G U &\colon \cB\to \bB(Y_A\otimes^G W_C) & b&\mapsto 1_{\ell^2(G)}\otimes (T\otimes U)_b,
 	\end{align*}
 	which is the composition of $T\otimes U$ with the unital and faithful *-representation 
 	\begin{equation*}
 	\Theta\colon \bB(Y_A\otimes W_C)\to \bB(Y_A\otimes^G W_C),\ \Theta(R)=1_{\ell^2(G)}\otimes R. 
 	\end{equation*}
 	Note that the identity $T\otimes^G U = \Theta\circ (T\otimes U)$ holds for integrated forms.
 	Thus it suffices to prove
 	\begin{equation}\label{equ:equivalent condition of thesis for absorption principle I}
 	\|(T\otimes^G U)^{\intform{\ }}_f  \|= \sup\{ \|\intform{\Ind}_{{}_t H}^{\cB}({}_{t|}T\otimes {}_tV)_f\|\colon t\in G \}\qquad \forall\ f\in L^1(\cB).
 	\end{equation}
 	
 	We claim there exists a unique linear map 
 	\begin{equation}\label{equ:map L}
 	L\colon C_c(G,Y_A\otimes^G Z_C )\to Y_A\otimes^G   W_C
 	\end{equation}
 	which is continuous in the inductive limit topology and
 	\begin{equation}\label{equ:defining condition for L}
 	L( f\odot [\delta_t\otimes \xi\otimes \eta ]) =\delta_t\otimes ( \xi \otimes [f\otimes_{V} \eta])
 	\end{equation}
 	for all $f\in C_c(G),$ $\xi\in Y_A,$ $\eta\in Z_C$ and $t\in G,$ where:
 	\begin{itemize}
 		\item We regard $C_c(G)$ as $C_c(\cT_G)$ and think of $L^2_{{}_t H}(G)$ as a completion of $C_c(G).$
 		\item If $f\in C_c(G)$ and $w\in Y_A\otimes^G Z_C ,$ $f\odot w\in C_c(G,Y_A\otimes^G Z_C)$ is given by $(f\odot w) (r):=f(r)w.$
 	\end{itemize}
 	
 	Uniqueness of $L$ follows from the fact that the functions of the form $f\odot (\delta_t\otimes \xi\otimes \eta )$ span a subset of $C_c(G,Y_A\otimes^G Z_C)$ which is dense in the inductive limit topology (see~\cite[II 14.6]{FlDr88}).
 	Take functions $u,v\in C_c(G,Y_A\otimes^G Z_C)$ that can be expressed as elementary tensors 
 	\begin{align*}
 	u &= f\odot (\delta_r\otimes \xi\otimes \eta ) & v&=g\odot (\delta_s\otimes \zeta\otimes \kappa)
 	\end{align*}
 	as explained before.
 	If $\delta_{r,s}$ is the Kronecker delta function, then the inner product of (the candidates for) $L(u)$ and $L(v)$ is
 	\begin{multline}\label{equ:basic inner product identity to construct L}
 	\langle \delta_r\otimes (\xi \otimes [f\otimes_{V} \eta]),\delta_s\otimes (\zeta \otimes [g \otimes_{V} \kappa])\rangle =\\
 	=  \delta_{r,s}\langle \xi,\zeta\rangle \otimes \langle f\otimes_{V} \eta,g \otimes_{V} \kappa\rangle\\
 	= \delta_{r,s} \langle \xi,\zeta\rangle \otimes\int_{H} \int_G \Delta_G(t)^{1/2}\Delta_H(t)^{-1/2} \overline{f(z)} g(zt) \langle \eta, V_t  \kappa\rangle \, d_Gzd_Ht\\
 	= \int_{H}  \int_G \Delta_G(t)^{1/2}\Delta_{H}(t)^{-1/2} \overline{f(z)} g(zt) \langle \delta_r \otimes \xi\otimes\eta, (1\otimes 1\otimes V_t)(\delta_s\otimes \zeta\otimes \kappa )\rangle \, d_Gzd_Ht\\
 	= \int_{H}  \int_G \Delta_G(t)^{1/2}\Delta_{H}(t)^{-1/2} \langle u(z), (1\otimes 1\otimes V_t)v(zt)\rangle \, d_Gzd_Ht.
 	\end{multline}
 	
 	Fix a compact $D\subset G$ and denote $C_D(G,Y_A\otimes^G Z_C)$ the set formed by those $f\in C_c(G,Y_A\otimes^G Z_C)$ with support contained in $D;$ this set is in fact a Banach space with the norm $\|\ \|_\infty.$
 	Let $C_D^\odot$ be the subspace of $C_D(G,Y_A\otimes^G Z_C)$ spanned by the functions of the form $f\odot (\delta_t\otimes \xi\otimes \eta)$ with $f\in C_D(G),$ $t\in G,$ $\xi\in Y_A$ and $\eta\in Z_C.$
 	We clearly have $C_c(G)C_D^\odot \subset C_D^\odot.$
 	If for each $t\in G$ we define $C_D^\odot(t)$ as the closure of $\{u(t)\colon u\in C_D^\odot\},$ then $C_D^\odot(t)=Y_A\otimes^G Z_C$ for every $t$ in the interior of $D$ and $\{0\}$ otherwise.
 	By \cite[Lemma 5.1]{Ab03}, the closure of $C_D^\odot$ in $C_c(G,Y_A\otimes^G Z_C)$ with respect to the inductive limit topology is $\{f\in C_c(G,Y_A\otimes^G Z_C)\colon f(t)\in C_D^\odot(t)\, \forall\, t\in G  \}=C_D(G,Y_A\otimes^G Z_C).$
 	
 	Take any $u,v\in C_D(G,Y_A\otimes^G Z_C).$
 	By the preceding paragraph there are sequences $\{u_n\}_{n\in \bN}$ and $\{v_n\}_{n\in \bN}$ in $C_D^\odot$ converging uniformly to $u$ and $v,$ respectively.
 	Then for all $n\in \bN$ there exists a positive integer $m_n$ and (for each $j=1,\ldots,m_n$) elements $f_{n,j},g_{n,j}\in C_D(G),$ $r_{n,j},s_{n,j}\in G,$ $\xi_{n,j},\zeta_{n,j}\in Y_A$ and $\zeta_{n,j},\kappa_{n,j}\in Z_C$ such that
 	\begin{align*}
 	u_n &=\sum_{j=1}^{m_n} f_{n,j}\odot (\delta_{r_{n,j}}\otimes \xi_{n,j}\otimes \eta_{n,j} ) & v&=\sum_{j=1}^{m_n}g_{n,j}\odot (\delta_{s_{n,j}}\otimes \zeta_{n,j}\otimes \kappa_{n,j}).
 	\end{align*}
 	
 	Let $\alpha_D$ be the measure of $D$ with respect to  $d_Gs$  and $\beta_D$ that of $H\cap (D^{-1}D)$ with respect to $d_Ht.$
 	If $\gamma_D:=\sup\{|\Delta_G(t)^{1/2}\Delta_{H}(t)^{-1/2}|\colon t\in H\cap (D^{-1}D) \},$ then \eqref{equ:basic inner product identity to construct L} implies that for all $p,q\in \bN$ we have
 	\begin{multline*}
 	\|\sum_{j=1}^{m_p} \delta_{r_{p,j}} \otimes (\xi_{p,j}\otimes [f_{p,j}\otimes_{V} \eta_{p,j}] )-\sum_{k=1}^{m_q}\delta_{s_{q,k}}\otimes ( \zeta_{q,k}\otimes [g_{q,k}\otimes_{V} \kappa_{q,k}] )\|^2 =\\
 	= \| \int_{H}  \int_G \Delta_G(t)^{1/2}\Delta_{H}(t)^{-1/2} \langle (u_p-v_q)(z), (1\otimes 1\otimes V_t)(u_p-v_q)(zt)\rangle \, d_Gzd_Ht \|\\
 	\leq \alpha_D\beta_D\gamma_D \|u_p-v_q\|_\infty^2.
 	\end{multline*}
 	Several conclusion arise from the inequality above:
 	\begin{enumerate}
 		\item If $u=v$ and $v_q=u_q,$ it follows that $\{\sum_{j=1}^{m_p} \delta_{r_{p,j}} \otimes (\xi_{p,j}\otimes [f_{p,j}\otimes_{V} \eta_{p,j}]) \}_{p\in \bN}$ is a Cauchy sequence in $Y_A\otimes^G W_C,$ whose limit we denote $L_D(\{u_n\}_{n\in \bN}).$
 		\item If $u=v$ and we take limit in $p$ and $q,$  it follows that $L_D(\{u_n\}_{n\in \bN})= L_D(\{v_n\}_{n\in \bN}).$
 		Thus we may define a function $L_D\colon C_D(G,Y_A\otimes^G Z_C)\to Y_A\otimes^G W_C,\ u\mapsto L_D(u):=L_D(\{u_n\}_{n\in \bN}).$
 		\item Taking limit in $p$ and $q$ we obtain $\| L_D(u)-L_D(v) \|\leq \sqrt{\alpha_D\beta_D\gamma_D}\|u-v\|_{\infty},$ so $L_D$ is continuous.
 		\item $L_D(f\odot (\delta_t\otimes \xi\otimes\eta))=\delta_t\otimes \xi\otimes (f\otimes_V\eta)$ and $L_D$ is linear when restricted to $C_D^\odot.$
 		Thus $L_D$ is linear.
 		\item By~\eqref{equ:basic inner product identity to construct L} and the continuity of $L_D,$
 		\begin{equation}\label{equ:L and inner product}
 		\langle L_D(u),L_D(v)\rangle = \int_G \int_H \Delta_G(t)^{1/2}\Delta_{H}(t)^{-1/2}\langle u(s), (1\otimes 1 \otimes V_t)v(st)\rangle \, d_Ht d_Gs .
 		\end{equation}
 	\end{enumerate}
 	
 	It follows immediately that $L_E$ is an extension of $L_D$ whenever $E\subset G$ is a compact set containing $D.$
 	Then there exists a unique function $L\colon C_c(G,Y_A\otimes^G Z_C)\to Y_A\otimes^G W_C$ extending all the $L_D$'s.
 	This extension is linear and continuous in the inductive limit topology by \cite[II 14.3]{FlDr88}.
 	Note also that $L$ satisfies~\eqref{equ:defining condition for L} and, consequently, it has dense range.

 	Given $t\in G,$ $f\in C_c(\cB),$ $\xi\in Y_A$ and $\eta\in Z_C$ we define $[t,f,\xi,\eta]\in C_c(G,Y_A\otimes^G Z_C)$ by
 	\begin{equation}\label{equ:right brakets for triple tensors}
 	[t,f,\xi,\eta](r)= \Delta_G(t)^{-1/2} \delta_t\otimes T_{f(r\tmu)}\xi\otimes \eta
 	\end{equation}

 	We claim there exists a unique linear and continuous map
 	\begin{equation}\label{equ:domain and range of I}
 	I\colon \bigoplus_{t\in G} L^2_{{}_t H}(\cB)\otimes_{{}_{t|}T\otimes {}_tV} (Y_A\otimes Z_C) \to Y_A\otimes^G W_C
 	\end{equation}
 	such that for all $t\in G,$ $f\in C_c(\cB)\subset L^2_{{}_t H}(\cB),$ $\xi\in Y_A$ and $\eta\in Z_C,$
 	\begin{equation}\label{equ:defining property for I}
 	I( f\otimes_{{}_{t|}T\otimes {}_tV} (\xi\otimes \eta) ) = L([t,f,\xi,\eta]).
 	\end{equation}
 	
 	The direct summand $L^2_{{}_t H}(\cB)\otimes_{{}_{t|}T\otimes {}_tV} (Y_A\otimes Z_C)$ of $\bigoplus_{s\in G} L^2_{{}_{s}H}(\cB)\otimes_{{}_{s|}T\otimes {}_{s}V} (Y_A\otimes Z_C)$ is generated by elements $f\otimes_{{}_{t|}T\otimes {}_tV} (\xi\otimes \eta),$ with $f\in C_c(\cB),$ $\xi\in Y_A$ and $\eta\in Z_C.$
 	Notice the ``$t$'' in the tensor product indicates the direct summand the tensor belongs to.
 	
 	Take vectors $f\otimes_{{}_{r|}T\otimes {}_r V}(\xi\otimes \eta)$ and $g\otimes_{{}_{t|}T\otimes {}_tV}(\zeta\otimes\kappa).$
 	If $\Psi_r(w):=\Delta_G(w)^{1/2}\Delta_{{}_rH}(w)^{-1/2},$ then by~\eqref{equ:L and inner product} we have
 	\begin{multline*}
 	\langle L([r,f,\xi,\eta]),L([t,g,\zeta,\kappa])\rangle=\\
 	= \int_H\int_G \Psi_e(w)\Delta_G(rt)^{-1/2} \langle \delta_r\otimes T_{f(z\rmu )}\xi \otimes \eta ,\delta_t\otimes T_{g(zw\tmu)}\zeta\otimes V_w\kappa\rangle  \, d_Gzd_Hw\\
 	= \int_H\int_G \Psi_e(w) \Delta_G(\rmu)\delta_{r,t} \langle T_{f(z\rmu )}\xi \otimes \eta ,T_{g(zw\rmu)}\zeta\otimes V_w\kappa\rangle  \, d_Gzd_Hw\\
 	= \int_H\int_G \Psi_e(w) \delta_{r,t} \langle T_{f(z)}\xi \otimes \eta ,T_{g(zr w\rmu)}\zeta\otimes V_w\kappa\rangle  \, d_Gzd_Hw\\
 	= \int_H\int_G \Psi_r(rw\rmu) \delta_{r,t} \langle \xi \otimes \eta ,T_{f(z)^*g(zr w\rmu )}\zeta\otimes {}_r V_{rw\rmu}\kappa\rangle  \, d_Gzd_Hw \\
 	= \int_{{}_rH}\int_G \Psi_r(w) \delta_{r,t} \langle \xi \otimes \eta ,T_{f(z)^*g(zw)}\zeta\otimes {}_r V_{w}\kappa\rangle  \, d_Gzd_{{}_rH}w \\
 	= \delta_{r,t} \langle \xi \otimes \eta ,({}_{r|}T\otimes {}_r V)^{\intform{\ }}_{p^G_{{}_rH}(f^**g)}(\zeta\otimes \kappa)\rangle  \\
 	= \langle f\otimes_{{}_{r|}T\otimes {}_r V}(\xi\otimes \eta),g\otimes_{{}_{t|}T\otimes {}_tV}(\zeta\otimes\kappa)\rangle.
 	\end{multline*}
 	Thus there exists a linear isometry $I$ satisfying both~\eqref{equ:domain and range of I} and~\eqref{equ:defining property for I}.
 	Besides, $I$ preserves inner products.

 	Following the ideas of~\cite{ExNg} we define
 	\begin{align*}
 	\rho & \colon G\to \bB(Y_A\otimes^G W_C) & \rho_t &:= \lt_t\otimes 1_{Y_A}\otimes 1_{W_C},
 	\end{align*}
 	where $\lt\colon G\to \bB(\ell^2(G))$ is the left regular representation if we consider $G$ as a discrete group ($\lt_t(\delta_s)=\delta_{ts}$).
 	Note $\rho$ and $\Theta$ have commuting ranges, so the range of $\rho$ commutes with that of $T\otimes^GU.$
 	
 	We remark that the continuity of $\rho$ (which may fail) plays no r\^ole in the proof, we use $\rho$ just to ``move things around''.
 	
 	Let $K$ be the image of the map $I$ of~\eqref{equ:domain and range of I}.
 	We claim that $ G\cdot K:= \spn\{\rho_t K\colon t\in G\}$
 	is dense in $Y_A\otimes^G W_C.$
 	To prove this we define, for each $t\in G,$ the function 
 	\begin{equation*}
 	\mu_t \colon C_c(G,Y_A\otimes^G Z_C)\to C_c(G,Y_A\otimes^G Z_C),\ (\mu_t f)(z) =(\lt_t\otimes 1_{Y_A}\otimes 1_{Z_C})f(z). 
 	\end{equation*}
 	In particular, $\mu_t (f\odot (\delta_r\otimes \xi\otimes \eta))=f\odot (\delta_{tr}\otimes \xi\otimes \eta).$
 	Hence,
 	\begin{multline*}
 	L\circ \mu_t (f\odot (\delta_r\otimes \xi\otimes \eta))=
 	L(f \odot (\delta_{tr}\otimes \xi\otimes \eta))
 	=\delta_{tr}\otimes \xi \otimes [f\otimes_{V} \eta] = \\
 	=\rho_t  (\delta_{r}\otimes \xi \otimes [f\otimes_{V} \eta])
 	=\rho_t L(f\odot( \delta_r\otimes \xi\otimes \eta )).
 	\end{multline*}
 	Since both $L\circ \mu_t $ and $\rho_t \circ L$ are linear and continuous in the inductive limit topology and agree on a dense set, it follows that $L\circ \mu_t =\rho_t \circ L.$
 	Thus $\overline{G\cdot K}$ contains the image through  $L$ of
 	\begin{equation*}
 	K_0:= \spn\{ \mu_t [r,f,\xi,\eta]\colon r,t\in G,\, f\in C_c(\cB),\, \xi\in Y_A,\, \eta\in Z_C  \}\subset C_c(G,Y_A\otimes^G Z_C).
 	\end{equation*}
 	
 	Note $C(G)K_0\subset K_0.$
 	Besides, 
 	\begin{equation}\label{equ:mu []}
 	\mu_r [t,f,\xi,\eta](z)
 	=\Delta_G(t)^{-1/2}\delta_{rt}\otimes T_{f(z\tmu)} \xi\otimes \eta.
 	\end{equation}
 	Fixing $z\in G$ and varying $r,t\in G,$ $\xi\in Y_A,$ $\eta\in Z_C$ and $f\in C_c(\cB),$ the elements we obtain on the right hand side of \eqref{equ:mu []} are all those of the form $\delta_{s}\otimes T_b \xi\otimes \eta,$ for arbitrary $s\in G,$ $b\in \cB,$ $\xi\in Y_A$ and $\eta\in Z_C.$
 	This last type of vectors span $Y_A\otimes^G Z_C$ because $T$ is non degenerate, and we conclude (using~\cite[II 14.3]{FlDr88}) that $K_0$ is dense in $C_c(G,Y_A\otimes^G Z_C)$ in the inductive limit topology.
 	Hence $\overline{G\cdot K}$ contains the dense set $L(K_0)$ and it follows that $\overline{G\cdot K}=Y_A\otimes^G W_C.$
 	
 	Our next goal is to show that defining the *-representation
 	\begin{equation*}
 	R:=\bigoplus_{r\in G} \Ind_{{}_rH}^\cB({}_{r|}T\otimes {}_r V )
 	\end{equation*}
 	the identity 
 	\begin{equation*}
 	(T\otimes^G U)_b \circ I = I\circ R_b
 	\end{equation*}
 	obtains for all $b\in \cB.$
 	To prove this claim we fix $r,p,q\in G,$ $b\in B_r,$ $f\in C_c(\cB),$ $g\in C_c(G),$ $\xi,\zeta\in Y_A$ and $\eta, \kappa\in Z_C.$
 	For convenience we denote $u$ and $v$ the tensors $f\otimes_{{}_{p|}T\otimes {}_{p}V} (\xi\otimes \eta) $ and $g\odot (\delta_q\otimes \zeta\otimes \kappa),$ respectively.
 	Recalling~\eqref{equ:L and inner product} we get
 	\begin{multline*}
 	\langle (T\otimes^G U)_b \circ I (u ),L(v ) \rangle=\\
 	= \langle  L([p,f,\xi,\eta]),(T\otimes^G U)_{b^*}(\delta_q\otimes \zeta\otimes (g\otimes_{V}\kappa)) \rangle\\
 	= \langle  L([p,f,\xi,\eta]),\delta_q\otimes T_{b^*}\zeta\otimes ( \regrep{HG}_\rmu(g)\otimes_{V}\kappa ) \rangle \\
 	= \langle  L([p,f,\xi,\eta]),L(\regrep{HG}_\rmu(g)\odot( \delta_q \otimes T_{b^*}\zeta\otimes \kappa))\rangle\\
 	=\int_H\int_G  \Psi_e(w)  \langle [p,f,\xi,\eta](z),\delta_q \otimes T_{b^*}\zeta\otimes V_w\kappa \rangle g(r zw)\, d_Gzd_Hw\\
 	=\int_H\int_G  \Psi_e(w) \Delta_G(p)^{-1/2} \langle \delta_p\otimes T_{f(zp^{-1})}\xi\otimes \eta,\delta_q \otimes T_{b^*}\zeta\otimes V_w\kappa \rangle g(r zw)\, d_Gzd_Hw\\
 	=\int_H\int_G  \Psi_e(w) \Delta_G(p)^{-1/2}\delta_{p,q} \langle T_{bf(\rmu zp^{-1})}\xi\otimes \eta,\zeta\otimes V_w\kappa \rangle g(zw)\, d_Gzd_Hw\\
 	=\int_H\int_G  \Psi_e(w) \Delta_G(p)^{-1/2}\delta_{p,q} \langle T_{(bf)(zp^{-1})}\xi\otimes \eta,\zeta\otimes V_w\kappa \rangle g(zw)\, d_Gzd_Hw\\
 	=\langle L([p,bf,\xi,\eta]),L(g\odot (\delta_q\otimes \zeta\otimes \kappa))\rangle
 	=\langle I(bf\otimes_{S^q\otimes V^q}(\xi\otimes \eta)),L(v)\rangle\\
 	= \langle I\circ R_b(u),L(v)\rangle.
 	\end{multline*}
 	Since $u$ and $v$ are arbitrary basic tensors and $L$ has dense range, by linearity and continuity we get that $(T\otimes^G U)_b \circ I = I\circ R_b;$ which implies that
 	\begin{equation}\label{equ:T,U,I,R}
 	(T\otimes^G U)^{\intform{\ }}_f \circ I = I\circ \intform{R}_f,\qquad \forall \ f\in L^1(\cB).
 	\end{equation}
 	
 	Consider the *-representations
 	\begin{align*}
 	\Omega_T&\colon C^*(\cB)\to \bB(Y_A\otimes^G W_C)\\
 	\Omega_R&\colon C^*(\cB)\to \bB\left( \bigoplus_{r\in G} L^2_{{}_rH}(G)\otimes_{{}_{r|}T\otimes {}_r V}(Y_A\otimes Z_C)    \right)
 	\end{align*}
 	such that $\Omega_T\circ \intregrep{\cB } = T\otimes^G U$ and $\Omega_R\circ  \intregrep{\cB }  = \intform{R}.$
 	Since the image of $\intregrep{\cB }\colon L^1(\cB)\to C^*(\cB)$ is dense in $C^*(\cB),$ \eqref{equ:T,U,I,R} implies $\Omega_T(f)\circ I=I\circ \Omega_R(f)$ for all $f\in C^*(\cB).$
 	Besides, the image of $\Omega_T$ is the closure of $(T\otimes^G U)^{\intform{\ }}(L^1(\cB)),$ which commutes with the image of $\rho.$
 	Thus the images of $\Omega_T$ and $\rho$ commute.
 	
 	Assume $f\in C^*(\cB)$ is in the kernel of $\Omega_R.$
 	Then, for all $u\in \bigoplus_{r\in G} L^2_{{}_rH}(G)\otimes_{{}_{r|}T\otimes {}_r V}(Y_A\otimes Z_C)$ and $t\in G,$  we have
 	\begin{equation*}
 	\Omega_T(f)\circ \rho_t \circ I(u) 
 	= \rho_t \circ \Omega_T(f)\circ I(u)
 	=\rho_t \circ I\circ \Omega_R(f)(u) = 0.
 	\end{equation*}
 	Recalling that $I(u)\in K$ and that  $G\cdot K$ spans a dense subset of $Y_A\otimes^G W_C,$ we deduce that $\Omega_T(f)=0.$
 	
 	By thinking of the image of $\Omega_R$ as the quotient of $C^*(\cB)$ by the kernel of $\Omega_R,$ we can define a morphism of *-algebras $\Phi\colon \Omega_R(C^*(\cB))\to \Omega_T(C^*(\cB))$ such that $\Phi\circ \Omega_R = \Omega_T.$
 	Since $\Phi$ is contractive, $\| \Omega_R(f) \|\geq \|\Omega_T(f)\|$ for all $f\in C^*(\cB).$
 	The inequality $\|\Omega_R(f)\|\leq \|\Omega_T(f)\|$ is trivial because $I$ is an isometry and $\Omega_T(f)\circ I=I\circ \Omega_R(f).$
 	Then $\Phi$ is isometric and, consequently, a C*-isomorphism.
 	It thus follows that for all $f\in C^*(\cB),$
 	\begin{align*}
 \| (T\otimes U)_f \|	& = \| (T\otimes^G U)_f \|
 	=  \|\Omega_T(\intregrep{\cB }_f )\|
 	=  \|\Omega_R(\intregrep{\cB }_f )\|
 	= \| \intform{R}_f\|
 	= \| \bigoplus_{r\in G}  \intform{\Ind}_{{}_rH}^\cB({}_{r|}T\otimes {}_r V )_f \|\\
 	& = \sup\{\|  \intform{\Ind}_{{}_rH}^\cB({}_{r|}T\otimes {}_r V )_f\|\colon r\in G\};
 	\end{align*}
 	showing~\eqref{equ:equivalent condition of thesis for absorption principle I} holds and completing the proof.
 \end{proof}

\begin{corollary}\label{cor:of FEexell absorption principle}
 Let $\cB=\{B_t\}_{t\in G}$ be a Fell bundle and $H \sbgp G.$
 Given non degenerate *-representations $T\colon \cB\to \bB(X_A)$ and $V\colon H\to \bB(Z_C),$  for all $f\in C^*(\cB)$ we have
 \begin{equation}\label{equ:bound form FExel absorption I}
  \| (T\intform{\otimes}\Ind_H^G(V))_f \|\leq \sup_{t\in G}\| f \|_{tHt^{-1}}^\cB.
 \end{equation}
 The bound above turns into an equality if the following two conditions are satisfied:
 \begin{enumerate}[(a)]
 	\item $|\int_H f(t)\, dt|\leq \|\intform{V}_f\|$ for all $f\in L^1(H);$ as it is the case if $V$ is the trivial representation $\kappa_H\colon H\to\bC.$
 	\item For every $t\in G,$ the integrated form $T|_{\cB_{tHt^{-1}}}\colon C^*(\cB_{tHt^{-1}})\to \bB(X_A)$ is faithful.
  \end{enumerate}
\end{corollary}
\begin{proof}
The bound \eqref{equ:bound form FExel absorption I} follows immediately from Theorem~\ref{thm:Fells absorption principle I} and the definition of the norms $\|\ \|_{tH\tmu}^\cB.$
Assume conditions (a) and (b) are satisfied.
Take a faithful *-representation $\pi\colon C\to \bB(W)$ and notice that $V\otimes_\pi 1$ weakly contains $\kappa_H$ because for all $f\in L^1(H),$
\begin{equation*}
 \|\intform{\kappa_H}_f\| = |\int_H f(t)\, dt|\leq \|\intform{V}_f\|=\|\intform{V}_f\otimes_\pi 1 \|=\|(V\otimes_\pi 1)^{\intform{\ }}_f\|.
\end{equation*}

Now take a faithful *-representation $\rho\colon A\to \bB(M);$ form the Hilbert space 
\begin{equation*}
N:=(X_A\otimes Z_C)\otimes_{\pi\otimes \rho}(W\otimes M)=(X_A\otimes_\pi W)\otimes(Z_C\otimes_\rho M)
\end{equation*}
and consider the unique *-homomorphism 
\begin{equation*}
  \Omega\colon \bB( X_A\otimes Z_C )\to \bB(N)
\end{equation*}
such that $\Omega( j\otimes k )=(j\otimes_\pi 1)\otimes(k\otimes_\rho 1)=(j\otimes k)\otimes_{\pi\otimes \rho}1,$ for all $j\in \bB(X_A)$ and $k\in \bB(Z_C).$
By Remark~\ref{rem:induction and balanced tensor product}
\begin{equation*}
\Omega\circ (T\otimes \Ind_H^ G(V))
=(T\otimes_\pi 1)\otimes (\Ind_H^G(V)\otimes_\rho 1)
=(T\otimes_\pi 1)\otimes \Ind_H^G(V\otimes_\rho 1)
\end{equation*}
and, consequently, for all $f\in L^1(\cB)$
\begin{equation*}
\| ((T\otimes_\pi 1)\intform{\otimes }\Ind_H^G(V\otimes_\rho 1))_f \|=\|\Omega\left((T\intform{\otimes}\Ind_H^G(V))_f\right) \|\leq\| (T\intform{\otimes}\Ind_H^G(V))_f\|.
\end{equation*}

Using Remark~\ref{rem:induction and balanced tensor product} once again we get that the integrated form of $(T\otimes_\pi 1)|_{\cB_{tHt^{-1}}} = (T|_{\cB_{tHt^{-1}}})\otimes_\pi 1$ is $(T|_{\cB_{tHt^{-1}}})^{\intform{\ }}\otimes_\pi 1,$  which is a faithful *-representation of $C^*(\cB_{tHt^{-1}}),$ for all $t\in G.$ 
Then, by replacing $T$ and $V$ with $T\otimes_\rho 1$ and $V\otimes_\pi 1$ respectively, we can assume: $X=X_A$ and $Z=Z_C$ are Hilbert spaces and $V$ weakly contains $\kappa_H.$

By Proposition~\ref{prop:induction increases norm}, condition (b) implies $\|\Ind_{tHt^{-1}}^\cB(T|_{\cB_{tHt^{-1}}})_f\|=\|f\|_{tHt^{-1}}^\cB$ for all $f\in \cB_{tHt^{-1}}$ and $t\in G.$
Since $\kappa_{tHt^{-1}}$ is weakly contained in the conjugated representation ${}_tV\colon tHt^{-1}\to \bB(Z);$ for all $f\in L^1(\cB_{tHt^{-1}})$ and $t\in G$ we obtain
\begin{equation*}
 \|f\|_{tHt^{-1}}^\cB = \|\Ind_{tHt^{-1}}^\cB(T|_{\cB_{tHt^{-1}}})_f\| \leq \|\Ind_{tHt^{-1}}^\cB({}_tV\otimes T|_{\cB_{tHt^{-1}}})_f\|.
\end{equation*}
Thus, for all $f\in L^1(\cB),$
\begin{equation*}
	\sup_{t\in G}\|f\|_{tHt^{-1}}^\cB 
	\leq \sup_{t\in G} \|\Ind_{tHt^{-1}}^\cB({}_tV\otimes T|_{\cB_{tHt^{-1}}})_f\|
	= \| (\Ind_H^G(V)\intform{\otimes}T)_f \|.
\end{equation*}
and the thesis follows from Theorem~\ref{thm:Fells absorption principle I}.
\end{proof}

We have some remarks regarding condition (b) of the corollary above.
Firstly, it follows from the proof that if (b) is fulfilled by a *-representation on a Hilbert module, then it is fulfilled by a *-representation on a Hilbert space.
Secondly, suppose that for every $t\in H$ we can get a *-representation $S^t$ of $\cB$ on a Hilbert space such that $S^t|_{\cB_{tHt^{-1}}}$ weakly contains any other *-representation of $\cB_{tHt^{-1}}.$
Then the direct sum of $\{S^t\}_{t\in G}$ satisfies condition (b).
Finally, if there is one *-representation $T$ as in (b), then any $S\colon \cB\to \bB(Y)$ with with faithful and non degenerate integrated form also satisfies (b).
Indeed, since $S$ weakly contains $T,$ $S|_{tHt^{-1}}$ weakly contains $T|_{tHt^{-1}}$ for all $t\in G$ (see Theorem~\ref{thm:foundamental facts on weak containment}).

The preceding considerations motivate the following question:
given a subgroup $K \sbgp G$ and a non degenerate *-representation $S\colon \cB_K\to \bB(X),$ is $S$ weakly contained in the restriction $T|_{\cB_K}$ of  some *-representation $T$ of $\cB$?

A natural candidate for $T$ is $\Ind_K^\cB(S).$
In case $\cB$ is the trivial bundle over $G$ and $S$ is the trivial representation $\kappa_K\colon K\to \bC,$ we are asking if $\kappa_K\preceq \Ind_K^G(\kappa_K)|_K.$
This problem was studied by Derighetti in \cite{Derighetti}, where he found conditions that guarantee $\kappa_K\preceq \Ind_K^G(\kappa_K)|_K.$
For example, this is the case if $K$ is normal or open in $G,$ or if the normaliser of $K$ is open in $G.$
We will start Section~\ref{sec:weak containment with respect to subgroups} by studying this problem for *-representations of Fell bundles (see, for example,  Lemma~\ref{lem:normal or open implies strong Derighetti condition}).

If $H$ is normal in $G,$ then the bound \eqref{equ:bound form FExel absorption I} becomes
\begin{equation*}
 \| ( T\otimes\Ind_H^G(V))^{\intform{\ }}_f \|\leq \| f \|_H^\cB.
\end{equation*}
Quite interestingly, Fell's Absorption Principle implies the same bound holds if instead of assuming $H$ is normal in $G$ we suppose $\cB$ is saturated.
The Proposition below explains this phenomena.

\begin{proposition}\label{prop:condition for CHB=CtHtmuB}
 Let $\cB=\{B_t\}_{t\in G}$ be a Fell bundle, $H \sbgp  G,$ fix $s\in G$ and set $K:=sH\smu.$
 Consider the Haar measure on $K,$ $d_Kt,$ such that $d_K(t) = d_H (\smu ts)$ and assume that at least one of the following conditions holds:
 \begin{enumerate}
  \item $H$ is normal in $G.$
  \item $\cB$ is saturated.
  \item $\cB$ has a unitary multiplier of order $s$ (see \cite[VIII 3.9]{FlDr88}).
 \end{enumerate}
 Then $\|f\|_H^\cB=\|f\|_K^\cB$ for all $f\in L^1(\cB)$ or, equivalently, there exists a C*-isomorphism $\Omega\colon C^*_K(\cB)\to C^*_H(\cB)$ extending the inclusion $C_c(\cB)\subset C^*_H(\cB).$
 \begin{proof}
  The thesis follows trivially if $H$ is normal.
  Assume $\cB$ is saturated, in which case we adopt Fell's construction and notation for  conjugated representations~\cite[XI 16]{FlDr88}.
  Let $T\colon \cB_H\to \bB(Y)$ be a non degenerate *-representation such that $\|\intform{\Ind}_H^\cB(T)_f\|=\|f\|_H^\cB$ for all $f \in L^1(\cB).$
  Define ${}^sT\colon \cB_K\to \bB(Z)$ as the conjugated *-representation  and recall from \cite[XI 16.19]{FlDr88} that $\Ind_K^\cB({}^sT)$ is unitary equivalent to $\Ind_H^\cB(T).$
  Thus, for all $f\in L^1(\cB),$ 
  \begin{equation*}
   \|f\|_H^\cB 
     = \|\intform{\Ind}_H^\cB(T)_f\|
     = \|\intform{\Ind}_K^\cB({}^sT)_f\|
     \leq \|f\|_K^\cB.
  \end{equation*}
  By symmetry we get that $\|f\|_H^\cB=\|f\|_K^\cB.$
  
  Now assume $\cB$ has a unitary multiplier $u$ of order $s.$
  In this situation we may proceed as in \cite[XI 16.16]{FlDr88}, we repeat the construction to show the reader the saturation hypothesis is not really needed. 
  Let $T\colon \cB_H\to \bB(Y)$ be a non degenerate *-representation with faithful integrated form.
  Define ${}^uT\colon \cB_K\to \bB(Y)$ by ${}^uT_b:=T_{u^* bu}$ and note that ${}^{u^*}({}^uT)=T.$
  Set $S:={}^uT$ and, given $f\in C_c(\cB),$ define $[fu]\in C_c(\cB)$ by $[fu](t)=\Delta_G(s)^{-1/2}f(t\smu)u.$
  
  For all $f,g\in C_c(\cB)$ and $\xi,\eta\in Y:$
  \begin{multline*}
   \langle [fu]\otimes_T \xi,[gu]\otimes_T\eta\rangle 
      =\langle \xi,\intform{T}_{p^G_H( [fu]^**[gu] )}\eta\rangle =\\
   = \int_H\int_G \Delta_G(t)^{1/2}\Delta_H(t)^{-1/2} \langle \xi, T_{[fu](r)^*[gu](rt)} \eta\rangle \, d_Grd_Ht \\
   = \int_H\int_G \Delta_G(s)^{-1} \Delta_G(t)^{1/2}\Delta_H(t)^{-1/2} \langle \xi, T_{u^* f(r\smu )^*g(rt\smu)u} \eta\rangle \, d_Grd_Ht\\
        = \int_K\int_G \Delta_G(t)^{1/2}\Delta_K(t)^{-1/2} \langle \xi, S_{f(r)^*g(rt)} \eta\rangle \, d_Grd_Kt
        = \langle f\otimes_S\xi,g\otimes_S\eta\rangle.
  \end{multline*}
  Thus there exists a unique unitary operator $U\colon L^2_K(\cB)\otimes_S Y\to L^2_H(\cB)\otimes_TY$ such that $U(f\otimes_S\xi)=[fu]\otimes_T\xi.$

  Given $t\in G,$ $b\in B_t,$ $f\in C_c(\cB)$ and $\xi\in Y$ we have
  \begin{equation*}
   (b[fu])(r)=b[fu](\tmu r) = \Delta_G(s)^{-1/2} b f(\tmu r\smu )u = [(bf) u ](r),
  \end{equation*}
  for all $r\in G.$
  Thus,
  \begin{align*}
   U^* \Ind_H^\cB(T)_b U(f\otimes_S\xi)
    = U^* (b[fu]\otimes_T \xi)
    = U^* ( [(bf) u] \otimes_T\xi)
     = \Ind_K^\cB(S)_{b}(f\otimes_S\xi);
  \end{align*}
  implying that $U$ intertwines $\Ind_H^\cB(T)$ and $\Ind_K^\cB(S)$ (and their integrated forms).
  
  Our choice of $T$ guarantees that for all $f\in L^1(\cB)$
  \begin{equation*}
   \|f\|_H^\cB = \| \intform{\Ind}_H^\cB(T)_f \| = \| \intform{\Ind}_K^\cB(S)_f \|\leq \|f\|_K^\cB.
  \end{equation*}
  Since $u^*$ is a unitary multiplier of order $\smu$ and $\smu Ks=H,$ by symmetry we obtain $\|f\|_H^\cB=\|f\|_K^\cB.$
  \end{proof}
\end{proposition}

The Corollary below, which is a general version of \cite[Corollary 2.15]{ExNg}, follows immediately from the last Proposition above, Remark~\ref{rem:induced and faithful integrated form} and Theorem~\ref{thm:Fells absorption principle I}.

\begin{corollary}\label{cor:Fell absorption principle II bis}
	Consider a Fell bundle $\cB=\{B_t\}_{t\in G}$ and a subgroup $H \sbgp G$ such that at least one of the following conditions holds:
\begin{enumerate}
	\item $H$ is normal.
	\item $\cB$ is saturated.
	\item $\cB$ has enough unitary multipliers (in the sense of \cite[VIII 3.8]{FlDr88}).
\end{enumerate}

Then, given non degenerate *-representations $T\colon \cB\to \bB(X_A)$ and $V\colon H\to \bB(Z_C),$ it follows that 
\begin{equation}\label{equ:bound form FExel absorption II}
\| (T\intform{\otimes}\Ind_H^G(V) )_f \|\leq \| f \|_H^\cB.
\end{equation}	
 for all $f\in L^1(\cB).$
Equivalently, there is a *-representation $\pi\colon C^*_H(\cB)\to \bB((L^2_H(G)\otimes_V Z_C)\otimes X_A)$ such that $\pi\circ q^\cB_H=\Ind_H^G(V)\intform{\otimes} T.$
If we regard $C^*_H(\cB)$ as the image of $\intregrep{H\cB},$ then $\overline{\pi}\circ \regrep{H\cB}=T\otimes\Ind_H^G(V) .$

In case $(T|_{\cB_H})\intform{\otimes} V$ is faithful \eqref{equ:bound form FExel absorption II} becomes an equality and $\pi$ is faithful.
In particular, if $\kappa_H\colon H\to \bC$ is the trivial representation and the integrated form of $T|_{\cB_H}$ is faithful, then $\|(T\intform{\otimes}\Ind_H^G(\kappa_H))_f\|=\|f\|^\cB_H$ for all $f\in L^1(\cB).$
\end{corollary}

In case $H$ is the trivial subgroup $\{e\}$ the three claims in the corollary below are different ways of saying that $G$ is amenable.

\begin{corollary}
 Let $G$ be a LCH group and $H \sbgp G.$
 For the claims
 \begin{enumerate}
  \item The trivial representation $\kappa_G\colon G\to \bC$ is weakly contained in $\Ind_H^G(\kappa_H).$
  \item There exists a unitary representation $U\colon H\to \bB(X)$ such that $\kappa_G$ is weakly contained in $\Ind_H^G(U).$
  \item $G$ has the $H-$WCP.
 \end{enumerate}
 the implications $(1)\Rightarrow (2)\Leftrightarrow (3)$ hold.
 \begin{proof}
  In this proof we make no difference between $G$ and $\cT_G.$ 
  The implication (1)$\Rightarrow$(2) is trivial and (3)$\Rightarrow$(2) is a direct consequence of the definition of $H-$WCP.
  Assume (2) holds and take a representation $U\colon G\to \bB(X)$ with faithful integrated form.
  Then $\kappa_G\otimes U$ is unitary equivalent to $U$ and, by \cite[VIII 21.24]{FlDr88} and the hypothesis, it is weakly contained in $\Ind_H^G(V)\otimes U$ for some unitary representation $V\colon H\to \bB(Y).$
  Now we can use Corollary~\ref{cor:Fell absorption principle II bis} to deduce that for all $f\in L^1(G)$
  \begin{equation*}
   \|f\|_{C^*(G)} 
       = \|\intform{U}_f\| 
       = \|(\kappa_G\intform{\otimes }U)_f \|
       \leq \| (\Ind_H^G(V)\intform{\otimes } U)_f \|
       \leq \|f\|_{C^*_H(G)}
       \leq \|f\|_{C^*(G)}.
  \end{equation*}
  Hence $q^G_H$ is a faithful quotient map, i.e. a C*-isomorphism.
 \end{proof}
\end{corollary}

In \cite{ExNg} the authors use an approximation property to show that every Fell bundle over an amenable group has the WCP.
This is a particular case of the following.

\begin{corollary}\label{cor:icomplete answer to question HWC of group passes to bundle}
  Let $\cB=\{B_t\}_{t\in G}$ be a Fell bundle.
  If $H\sbgp G$ is such that $G$ has the $H-$WCP and at least one of the conditions enumerated in Corollary~\ref{cor:Fell absorption principle II bis} is satisfied, then $\cB$ has the $H-$WCP.
\begin{proof}
  	Let $T\colon \cB\to \bB(X)$ and $U\colon \bB(Z)$ be a non degenerate *-representation such that the integrated form of $T$ is faithful and $\kappa_G\preceq \Ind_H^G(U).$
  	By \cite[VIII 21.24]{FlDr88}, $T=\kappa_G\otimes T$ is weakly contained in $\Ind_H^G(U)\otimes T.$
  	Using Corollary~\ref{cor:Fell absorption principle II bis} we get that for all $f\in C^*(\cB),$
  	\begin{equation*}
  		\|f\|
  		   = \|\intform{T}_f\|
  		    \leq \| (\Ind_H^G(U)\otimes T)^{\intform{\ }}_f\|\leq \|q^\cB_H(f)\|
  		    \leq \|f\|;
  	\end{equation*}
  	which implies that $q^\cB_H\colon C^*(\cB)\to C^*_H(\cB)$ is a C*-isomorphism.
  \end{proof}
\end{corollary}

\section{Weak containment with respect to a pair of subgroups}\label{sec:weak containment with respect to subgroups}

We now resume the discussion left after Corollary~\ref{cor:isomorphisms of q maps} and try to find conditions under which the $HK-$WCP of $\cB$ implies the $H-$WCP of $\cB_K$ ($H \sbgp K \sbgp G$).
To do this we study the process $T\mapsto \Ind_H^\cB(T)|_{\cB_K},$ which we may apply to any *-representation $T$ of $\cB_H$ and can be decomposed (up to unitary equivalence) into two steps: 
\begin{align*}
T & \mapsto  \Ind_H^{\cB_K}(T) &  S&\mapsto \Ind_K^\cB(S)|_{\cB_K}
\end{align*}

Assume that $\cB$ has the $HK-$WCP and that we are lucky enough as to have $R\sim \Ind_K^\cB(R)|_{\cB_K}$ for every non degenerate *-representation $R$ of $\cB_K.$
Take one such $R.$
Since $\cB$ has the $HK-$WCP, there is a non degenerate *-representation $S$ of $\cB_H$ such that $\Ind_K^\cB(R)\preceq \Ind_H^\cB(S)=\Ind_K^\cB(\Ind_H^{\cB_K}(S)).$
By restricting to $\cB_K$ we get $R\sim \Ind_K^\cB(R)|_{\cB_K}\preceq \Ind_K^\cB(\Ind_H^{\cB_K}(S))|_{\cB_K}\sim \Ind_H^{\cB_K}(S);$ which implies $R\preceq \Ind_H^{\cB_K}(S).$
Thus $\cB_K$ has the $H-$WCP
.

In \cite{Derighetti} Derighetti shows that for $G=SL(2,\bC),$ $H=SL(2,\bR)$ and the trivial representation $\kappa_H\colon H\to \bC,$ it is not true that $\kappa_H\preceq  \Ind_H^G(\kappa_H)|_{H}.$
So, in general, we can not expect to have the equivalence $R\sim \Ind_K^\cB(R)|_{\cB_K}.$
But Derighetti also notices that $\kappa_H\preceq  \Ind_H^G(\kappa_H)|_{H}$ whenever $H$ is normal, open or has open normaliser in $G.$
This motivates the following.

\begin{lemma}[c.f. {\cite[XI 11.3]{FlDr88}}]\label{lem:normal or open implies strong Derighetti condition}
 Let $\cB=\{B_t\}_{t\in G}$ be a Fell bundle and $H \sbgp G.$
 If either $H$ is normal or open in $G,$ then  for every representation $T\colon \cB_H\to \bB(Y_A)$ it follows that 
 \begin{equation}\label{equ:T is weakly contained in restriction of induced}
  \|\intform{T}_f\|\leq \|[\Ind_H^\cB(T)|_{\cB_H}]^{\intform{\ }}_f\|
 \end{equation}
 for all $f\in C^*(\cB).$
 In case $T$ is non degenerate and $Y_A$ is a Hilbert space, $T\preceq \Ind_H^\cB(T)|_{\cB_H}.$
 \begin{proof}
 Both $\|\intform{T}_f\|$ and $\|[\Ind_H^\cB(T)|_{\cB_H}]^{\intform{\ }}_f\|$ are unaltered if we replace $T$ with its non degenerate part, so we may assume $T$ is non degenerate.
 Take a faithful non degenerate *-representation $\rho\colon A\to \bB(X)$ and recall that $\Ind_H^\cB(T\otimes_\rho 1) = \Ind_H^\cB(T)\otimes_\rho 1.$
 Also, $\|T_f\|=\|(T\otimes_\rho 1)_f\|$ and $\|\Ind_H^\cB(T)_f\|=\|\Ind_H^\cB(T\otimes_\rho 1)_f\|.$
 Thus we may assume, without loss of generality, that $T$ is non degenerate and $Y_A=Y$ is a Hilbert space.

 In case $H$ is open in $G,$ by 
  \cite[XI 14.21]{FlDr88} $T$ is a subrepresentation of $\Ind_H^\cB(T)|_{\cB_H}.$ Thus \eqref{equ:T is weakly contained in restriction of induced} clearly holds for all $f\in C^*(\cB).$
 
 Now suppose $H$ is normal in $G$ and let $\cC=\{C_\alpha\}_{\alpha\in G/H}$ be the partial cross sectional bundle over $G/H$ derived from $\cB,$ as defined in \cite[VIII 6]{FlDr88}.
 Recall that for each $tH\in G/H$ the fibre $C_{tH}$ is a completion of $C_c(\cB_{tH})$  with $C_{H}=L^1(\cB_H).$
 By \cite[VIII]{FlDr88}, there exists a *-representation $S\colon \cC\to \bB(L^2_H(\cB)\otimes_{\intform{T}}Y)$ such that $S_f\xi = \int_H \Ind_H^\cB(T)_{f(tz)}\xi\, dz$ for all $f\in C_c(\cB_{tH}),$ $\xi\in L^2_H(\cB)\otimes_{\intform{T}}Y$ and $tH\in G/H.$
 In particular, $S|_{C_H}$ is the restriction to $L^1(\cB_H)$ of (the integrated form of) $\Ind_H^\cB(T)|_{\cB_H}.$
 
  Theorem~\ref{thm:characterization of positive representations} implies $T$ is $\cB-$positive and the restriction $R\colon C_H\to \bB(Y)$ of $\intformexpl{T}$ to $L^1(\cB_H)\equiv C_H$ is $\cC-$positive by \cite[XI 12.7]{FlDr88}.
 Then we can form the induced representation $\Ind_{\{H\}}^\cC(R)\colon \cC\to \bB(L^2_{\{H\}}(\cC)\otimes_R Y).$
 But \cite[XI 12.7]{FlDr88} also implies that $\Ind_{\{H\}}^\cC(R)$ is unitary equivalent to $S.$
 By \cite[XI 11.3]{FlDr88}, $\intform{T}|_{L^1(\cB)}\equiv R\preceq \Ind_{\{H\}}^\cC(R)|_{C_H}\approx S|_{C_H}=\Ind_H^\cB(T)|_{L^1(\cB)}\Rightarrow T\preceq \Ind_H^\cB(T)|_{\cB_H}.$
 \end{proof}
\end{lemma}

We must point out that a previous version of the present article claimed to prove the lemma above for arbitrary $H\sbgp G$ (which is wrong by Derighetti's counter example).
Quite interestingly, the main conclusions remain unchanged because in all the situations where the false lemma was used there were additional assumptions on the subgroups that implied the thesis of the lemma was true.

\begin{definition}\label{defi:strong Derighetti}
 A Fell bundle $\cB=\{B_t\}_{t\in G}$ satisfies
 \begin{itemize}
 	\item Derighetti's \textit{weak} condition with respect to $H \sbgp G$ if for every non degenerate *-representation $T\colon \cB_H\to \bB(Y)$ there exists a *-representation $S\colon \cB_H\to \bB(Y)$ such that $T\preceq \Ind_H^\cB(S)|_{\cB_H}.$
 	\item Derighetti's \textit{strong} condition with respect to $H \sbgp G$ if every non degenerate *-representation $T\colon \cB_H\to \bB(Y)$ is weakly contained in $\Ind_H^\cB(T)|_{\cB_H}.$
 \end{itemize}  
\end{definition}

Lemma~\ref{lem:normal or open implies strong Derighetti condition} can be rephrased in terms of Derighetti's strong condition and it can be extended to cover other situations, as we shall see after the following.

\begin{proposition}\label{prop:sequence of subgroups and WC}
	Let $\cB=\{B_t\}_{t\in G}$ be a Fell bundle, consider subgroups $H=H_0 \sbgp H_1 \sbgp \cdots  \sbgp H_n=G$ and, given a non degenerate *-representation $T\colon \cB_H\to \bB(Y),$ define $T^{(0)}:=T$ and $T^{(k+1)}:=\Ind_{H_k}^{\cB_{H_{k+1}}}(T^{(k)})$ for $k=0,\ldots,n-1.$
	If $T^{(k)}\preceq  T^{(k+1)}|_{\cB_{H_k}}$ for all $k=0,\ldots,n-1,$ then $T\preceq \Ind_H^\cB(T)|_{\cB_H}.$
	\begin{proof}
		We proceed by induction in $n.$
		For $n=0,1$ the claim is trivial.
		Suppose the statement holds whenever we have chain of, at mots, $n\geq 1$ subgroups of $G.$
		Lets say we have $n+1$ closed subgroups of $G,$ $H=H_0 \sbgp \cdots  \sbgp H_{n+1} = G.$
		By hyphotesys, $T\equiv T^{(0)}$ is weakly contained in $T^{(1)}|_{\cB_{H_1}}$ and the induction hyphotesys implies $T^{(1)}$ is weakly contained in $\Ind_{H_1}^{\cB}(T^{(1)})|_{\cB_{H_1}}.$
		
		Since the restriction operation is continuous with respect to the regional topology \cite[VIII 21.20]{FlDr88} and preserves unitary equivalence classes, $T$ is weakly contained in 
		\begin{equation*}
		\Ind_{H_1}^{\cB}(T^{(1)})|_{\cB_{H_1}}|_{\cB_H} = \Ind_{H_1}^{\cB}(\Ind_H^{\cB_{H_1}}(T))|_{\cB_{H}} = \Ind_H^{\cB}(T))|_{\cB_{H}},
		\end{equation*}
		which gives the desired result.
	\end{proof}
\end{proposition}

\begin{remark}
	A Fell bundle $\cB=\{B_t\}_{t\in G}$ satisfies Derighetti's strong condition with respect to $H\sbgp G$ if $H$ is normal or open in $G$ (Lemma~\ref{lem:normal or open implies strong Derighetti condition}) and also if its normalizer is open in $G$ (Lemma~\ref{lem:normal or open implies strong Derighetti condition} and Proposition~\ref{prop:sequence of subgroups and WC}).
\end{remark}

We won't make any effort in determining if whether or not the amenability of $H\sbgp G$ implies every Fell bundle $\cB$ over $G$ satisfies Derighetti's (strong or weak) conditions.
The reason being that we want to use the conditions to answer Question~\ref{q:HK wcp of B implies H wcp of BK} and, if we know $H$ is amenable, then the problem reduces to the case  $H=\{e\}.$
Indeed, suppose we have a Fell bundle $\cB$ over $G$ and subgroups $H\sbgp K\sbgp G$ with $H$ amenable and $\cB$ having the $HK-$WCP.
By Corollary~\ref{cor:inner amenability and nuclearity}, $\cB_H$ has the WCP and, by induction in stages, $\cB$ has the $HK-$WCP $\Leftrightarrow$ $\cB$ has the $K-$WCP.
Also, $\cB_K$ has the $H$-WCP $\Leftrightarrow$ $\cB_K$ has the WCP.
Thus, if $H$ is amenable, our problem reduces to the case $H=\{e\}.$
This is very convenient because $C^*(\cB_H)$ and $C^*_H(\cB)$ become $B_e$ and $C^*_\red(\cB),$ respectively.

\begin{proposition}
 Let $\cB$ be a Fell bundle over $G,$ consider subgroups  $H \sbgp K \sbgp G$ and define $\regrep{K\cB|H}:=\regrep{K\cB}|_{\cB_H}\colon \cB_H\to \bB(L^2_K(\cB)).$
 If $\cB$ satisfies Derighetti's weak condition with respect to $H,$ then the integrated form of $\regrep{K\cB|H}$ is faithful.
 This gives an inclusion $C^*(\cB_H)\subset \bB(C^*_K(\cB))\subset \bB(L^2_H(\cB)).$
 \begin{proof}
  We consider the ``worst'' case first: $K=H.$ 
  Take a non degenerate *-representation $T\colon \cB_H\to \bB(Y)$ with faithful integrated form and let $S\colon \cB_H\to \bB(Z)$ be a *-representation such that $T\preceq \Ind_H^\cB(S)|_{\cB_H}.$
  Then $\Ind_H^\cB(S)|_{\cB_H}$ is  $\regrep{H\cB|H}\otimes_{\intform{S}} 1\colon \cB_H\to \bB(L^2_K(\cB)\otimes_{\intform{T}}Y)$ and for all $f\in C_c(\cB_H)$
  \begin{equation*}
   \|f\|_{C^*(\cB_H)} =\|\intform{T}_f\| \leq \|[\Ind_H^\cB(S)|_{\cB_H}]^{\intform{\ }}_f\| = \|\intregrep{H\cB|H}_f\otimes_{\intform{T}} 1\| = \|\intregrep{H\cB|H}_f\|\leq \|f\|_{C^*(\cB_H)}.
  \end{equation*}
  Thus the integrated form of  $\regrep{H\cB|H}$ is a faithful.
  
  For the general case we consider the canonical quotient map $q^\cB_{KH}\colon C^*_K(\cB)\to C^*_H(\cB)$ and the *-representations $\regrep{K\cB}\colon \cB\to \bB(C^*_K(\cB))\subset \bB(L^2_K(\cB))$ and $\regrep{H\cB}\colon \cB\to \bB(C^*_H(\cB)), $ whose integrated forms are $q^\cB_K$ and $q^\cB_H,$ respectively.
  If $\overline{q^\cB_{KH}}\colon \bB(C^*_K(\cB))\to \bB(C^*_H(\cB))$ is the natural extension of $q^\cB_{KH},$ then $\overline{q^\cB_{KH}}\circ \regrep{K\cB|H}= \regrep{H\cB|H}$ (for integrated and non integrated forms).
  The the integrated form of $\intregrep{K\cB|H}$ is faithful (because that of $\intregrep{H\cB|H}$ is).
 \end{proof}
\end{proposition}

Probably the following lemma is well known by  specialist, we include its proof because we were not able to find a reference.
The closer result we found is the Proposition in \cite[XI 16.27]{FlDr88}, where $G/H$ is assumed to be discrete.

\begin{lemma}\label{lem:induction and restriction of unitary rep to normal subgroups}
  Given normal subgroups of a group $G,$ $H\sbgp K\sbgp G,$ and a unitary representation $U\colon H\to \bB(X)$ it follows that $\Ind_H^G(U)|_K\sim \Ind_H^K(U).$
\end{lemma}
\begin{proof}
  We think of $G$ as the trivial bundle over $G$ with constant fibre $\bC.$
  By Lemma~\ref{lem:normal or open implies strong Derighetti condition} and induction in stages $\Ind_H^K(U)\preceq \Ind_K^G(\Ind_H^K(U))|_K = \Ind_H^G(U)|_K.$
  Thus it suffices to show that $\| (\Ind_H^G(U)|_K)_f \|\leq \| \Ind_H^K(U)_f \|$ for all $f\in C_c(K),$ which we do by using part or the proof of Theorem~\ref{thm:Fells absorption principle I}.
  
  The r\^oles of $T$ and $V$ will be played by the trivial representation $\kappa_G\colon G\to \bC$ and $U,$ respectively.
  Using the group $K$ instead of $G$ in~\eqref{equ:map L} we get a linear map 
  \begin{equation*}
L\colon C_c(K,\ell^2(K)\otimes X)\to \ell^2(K)\otimes (L^2_H(K)\otimes_U X)
  \end{equation*}
  which is continuous in the inductive limit topology and satisfies~\eqref{equ:defining condition for L}.
  
  Given $f\otimes \xi \in C_c(K)\odot X\subset C_c(K,X)$ we have $L(f\odot (\delta_e\otimes \xi))=\delta_e\otimes (f\otimes_U \xi),$ so there exists a unique linear map $L_e\colon C_c(K,X)\to L^2_H(K)\otimes_UX$ which is continuous in the inductive limit topology and $L_e(f\odot \xi)=f\otimes_U\xi.$
  
  Some straightforward computations show that the equality 
  \begin{equation}\label{equ:property of Le}
   \langle L_e(u),\Ind_H^K(U)_t L_e(v)\rangle 
   =  \int_H \int_K \langle u(p),U_s v(\tmu ps)\rangle\, dpds
   \end{equation}
   holds whenever $u$ and $v$ are elementary tensors of the form $f\odot \xi\in C_c(K)\odot X.$
   Besides, fixing $u$ or $v$ both sides of~\eqref{equ:property of Le} define linear or conjugate linear functions (of $v$ or $u,$ respectively) that are continuous in the inductive limit topology, so~\eqref{equ:property of Le} holds for all $u,v\in C_c(K,X).$
  
   Since $K$ is normal in $G$ we have $\Delta_G(s)=\Delta_K(s)$ for all $s\in K,$ in particular this holds for all $s\in H.$
   Besides, $H$ is normal in $G$ and $K,$ so $\Delta_G(s)=\Delta_K(s)=\Delta_H(s)$ for all $s\in H.$
   Let $\Gamma_{GK}\colon G\to (0,+\infty)$ be the function such that for all $a\in C_c(K)$ and $r\in G,$ $\int_K a(tr\tmu)\, d_K r = \Gamma_{GK}(t)\int_Ka(r)\, d_Kr$ (see \cite[III 8.3]{FlDr88}).
   We define $\Gamma_{GH}$ analogously, viewing $H$ as a normal subgroup of $G.$
   We consider the left invariant Haar measures of $G,K$ and $G/K$ have been chosen so that
   \begin{equation*}
   \int_G a(r)\, d_Gr = \int_{G/K} \int_K a(rt)\, d_Kt\,  d_{G/K} (rK),\ \forall\ a\in C_c(G).
   \end{equation*}
  
  Fix $f\in C_c(K),$ $g_i\in C_c(G)$ and $\xi_i\in X$ for $i=1,\ldots, n.$
  Set $\zeta:=\sum_{i=1}^n  g_i\otimes_U \xi_i.$
  It suffices to show that $\| (\Ind_H^G(U)|_K)_f \zeta \|^2\leq \|\Ind_H^K(U)_f\|^2\|\zeta\|^2$ and to do this we start by computing 
  \begin{multline*}
    \| (\Ind_H^G(U)|_K)_f \zeta\|^2
     = \int_K\sum_{i,j= 1}^n \langle g_i\otimes_U \xi_i, \Ind_H^G(U)_t (g_j\otimes_U \xi_j)\rangle f^**f(t)\, dt=\\
     = \int_K\sum_{i,j= 1}^n \langle \xi_i,\intform{U}_{(g_i^* *tg_j)|_H}\xi_j\rangle f^**f(t)\, dt=
     \int_K \int_H \sum_{i,j= 1}^n \langle \xi_i,U_s \xi_j\rangle (g_i^* *tg_j)(s) f^**f(t)\, dsdt=\\
     = \int_K \int_H \int_G \sum_{i,j= 1}^n \langle \xi_i,U_s \xi_j\rangle \overline{g_i(r)}g_j(\tmu r s) f^**f(t)\, drdsdt
     =\\
     = \int_{G/K}\int_K \int_K \int_H  \sum_{i,j= 1}^n \langle \xi_i,U_s \xi_j\rangle \overline{g_i(rp)}g_j(\tmu rp s) f^**f(t)\,  dsdtdp d(rK)=\\
     = \int_{G/K}\Gamma_{GK}(r)\Gamma_{GH}(r)\int_K \int_K \int_H  \sum_{i,j= 1}^n \langle g_i(pr) \xi_i,U_{\rmu sr} g_j(\tmu psr)\xi_j\rangle  f^**f(t)\,  dsdtdp d(rK).
  \end{multline*}
  
  Lets analyze the inner triple integral, $I(rK),$ above.
  Fix $r\in G$ and define $h_r\in C_c(K,X)$ by $h_r(p):=\sum_{i=1}^n g_i(pr)\xi_i.$
  Notice $s\mapsto U_{\rmu sr}$ is a unitary representation ${}_rU$ of $H$ and, by \eqref{equ:property of Le},
  \begin{align*}
  I(rK)= \langle L_e(h_r),\Ind_H^K({}_rU)_{f^**f}L_e(h_r)\rangle.
  \end{align*}
  We know from \cite[XI 12.21]{FlDr88} that $\Ind_H^K({}_rU)$ is unitary equivalent to $\Ind_H^K(U),$ thus 
  \begin{equation*}
  	I(rK)\leq \| \Ind_H^K(U)_f  \|^2\langle L_e(h_r),L_e(h_r)\rangle
  	=\| \Ind_H^K(U)_f  \|^2 \int_K \int_H  \sum_{i,j= 1}^n \langle g_i(pr) \xi_i,U_s g_j( psr)\xi_j\rangle  \,  dsdp
  \end{equation*}
  and we get
  \begin{multline*}
\| (\Ind_H^G(U)|_K)_f \zeta\|^2
= \int_{G/K} \Gamma_{GK}(r)\Gamma_{GH}(r) I(rK)\, drK\leq  \\
\leq \| \Ind_H^K(U)_f  \|^2\int_{G/K}\int_K \int_H \Gamma_{GK}(r)\Gamma_{GH}(r)  \sum_{i,j= 1}^n \langle g_i(pr) \xi_i,U_s g_j( psr)\xi_j\rangle  \,  dsdp d(rK) \\
\leq \| \Ind_H^K(U)_f  \|^2\int_{G/K}\int_K \int_H   \sum_{i,j= 1}^n \langle g_i(rp) \xi_i, g_j(rps)\xi_j\rangle  \,  dsdp d(rK)\\
\leq \| \Ind_H^K(U)_f  \|^2\int_{G}\int_H   \sum_{i,j= 1}^n \langle g_i(r) \xi_i, U_sg_j(rs)\xi_j\rangle  \,  dsdr
=  \| \Ind_H^K(U)_f\|^2  \int_H   \sum_{i,j= 1}^n \langle \xi_i, U_{g_i^**g_j|_H}\xi_j\rangle  \,  dsdr\\
\leq  \| \Ind_H^K(U)_f\|^2 \|\zeta\|^2.
  \end{multline*}
  Hence $\Ind_H^G(U)|_K\preceq \Ind_H^K(U).$
\end{proof}

\begin{proposition}\label{prop:inducting tensor product and restriction for normal subgroups}
Let $\cB=\{B_t\}_{t\in G}$ be a Fell bundle and take $H \sbgp K\sbgp G$ with both $H$ and $K$ normal in $G.$
Given non degenerate *-representations $T\colon \cB\to \bB(X)$ and $U\colon H\to \bB(Y)$ 
 it follows that 
\begin{equation*}
(T\otimes \Ind_H^G(U))|_{\cB_K}\sim \Ind_H^{\cB_K}(T|_{\cB_H}\otimes U).
\end{equation*} 
\begin{proof}
	By Theorem~\ref{thm:Fells absorption principle I}, $T\otimes \Ind_H^G(U)\sim \{ \Ind_H^\cB( T|_{\cB_H}\otimes {}_t U) \}_{t\in G}$ and we claim  $ \{ \Ind_H^\cB( T|_{\cB_H}\otimes {}_t U) \}_{t\in G}\sim \Ind_H^\cB(T|_{\cB_H}\otimes U).$
	In fact $\Ind_H^\cB( T|_{\cB_H}\otimes {}_t U)$ is weakly equivalent to $\Ind_H^\cB( T|_{\cB_H}\otimes U)$ for every $t\in G.$
	To prove this we fix $t\in G$ and use Lemma~\ref{lem:induction and restriction of unitary rep to normal subgroups} to get the equivalence ${}_t U\sim \Ind_H^G({}_t  U)|_H.$
	But we know from \cite[XI 12.21]{FlDr88} that  $\Ind_H^G({}_t  U)\sim \Ind_H^G(U),$ so ${}_t U\sim \Ind_H^G({}_t  U)|_H\sim \Ind_H^G(U)|_H\sim U.$
	Combining \cite[VIII 21.14]{FlDr88} with \cite[XI 12.4]{FlDr88} and recalling that the definition of weak equivalence of *-representations is given in terms of convergence with respect to the regional topology, we get that $(T|_{\cB_H})\otimes {}_t  U\sim (T|_{\cB_H})\otimes U$ and, consequently, $\Ind_H^\cB(T|_{\cB_H}\otimes U)\sim \Ind_H^\cB( T|_{\cB_H}\otimes {}_t U).$
	Any finite direct sum of members of $\{ \Ind_H^\cB( T|_{\cB_H}\otimes {}_t U) \}_{t\in G}$  is weakly contained in a finite direct sum of copies of $\Ind_H^\cB( T|_{\cB_H}\otimes U) ,$ thus $
      T\otimes \Ind_H^G(U)\sim \{ \Ind_H^\cB( T|_{\cB_H}\otimes {}_t U) \}_{t\in G}\sim \Ind_H^\cB( T|_{\cB_H}\otimes U).$
      
	If we apply the conclusion above to $ (T|_{\cB_K})\otimes \Ind_H^K(U),$ which we can do because $H$ is normal in $K,$ it follows that  $(T|_{\cB_K})\otimes \Ind_H^K(U)\sim \Ind_H^{\cB_K}(T|_{\cB_K}|_{\cB_H}\otimes U)=\Ind_H^{\cB_K}(T|_{\cB_H}\otimes U).$
	Using Lemma~\ref{lem:induction and restriction of unitary rep to normal subgroups} in conjunction with \cite[VIII 21.14]{FlDr88} we get
	\begin{equation*}
(T\otimes \Ind_H^G(U))|_{\cB_K} = (T|_{\cB_K})\otimes (\Ind_H^G(U)|_{K})\sim (T|_{\cB_K})\otimes \Ind_H^K(U).
	\end{equation*} 
	Hence, by transitivity, $(T\otimes \Ind_H^G(U))|_{\cB_K} \sim \Ind_H^{\cB_K}(T|_{\cB_H}\otimes U).$
\end{proof}
\end{proposition}

\begin{theorem}\label{thm:inclusion CHBK in BCHB}
 Let $\cB=\{B_t\}_{t\in G}$ be a Fell bundle that satisfies Derighetti's weak condition with respect to $K \sbgp G.$
 Given $H \sbgp K$ define $\regrep{H\cB|K}:=\regrep{H\cB}|_{\cB_K}\colon \cB_K\to \bB(L^2_H(\cB)).$
 Then, for all $f\in L^1(\cB_K),$ $\|f\|^{\cB_K}_H \leq \| \intregrep{H\cB|K}_f \|.$ 
 The equality holds (for all $f\in L^1(\cB_K)$) if both $H$ and $K$ are normal in $G,$ in which case we may identify $C^*_H(\cB_K)$ with $\intregrep{H\cB|K}(C^*(\cB_K))\subset \bB(C^*_H(\cB)).$
 
 More precisely, each $f\in C_c(\cB_K)\subset C^*_H(\cB_K)$ gets identified with the operator $M_f\in \bB(C^*_H(\cB))$ mapping $g\in C_c(\cB)$ to $f*_Kg\in C_c(\cB),$ with $f*_Kg(r)=\int_K f(t)g(\tmu r)\, dr.$
 Besides, for each $f\in L^1(\cB_K),$ the operator $\intregrep{H\cB|K}_f\in \bB(L^2_H(\cB))$ maps $g\in C_c(\cB)$ to $f*_Kg.$
 \begin{proof}
  Let $R\colon \cB_H\to \bB(X)$ and $S\colon\cB_H\to \bB(Y)$ be non degenerate *-representations such that the integrated form of $R$ is faithful and  $R\preceq \Ind_H^\cB(S)|_{\cB_H}.$
  The induction process is continuous with respect to the regional topology and preserves direct sums, so
  \begin{equation*}
\Ind_H^{\cB_K}(R)\preceq \Ind_H^{\cB_K}(\Ind_H^\cB(S)|_{\cB_H})
=\Ind_H^{\cB_K}(\Ind_H^\cB(S)|_{\cB_K}|_{\cB_H})
\approx \Ind_H^\cB(S)|_{\cB_K}.
  \end{equation*}

  The tensor product $\regrep{H\cB|K}\otimes_{S} 1$ is (unitary equivalent to)  $\Ind_H^\cB(S)|_{\cB_K}$ then, for all $f\in L^1(\cB_K),$
  \begin{equation*}
    \| f \|^{\cB_K}_H = \|\intregrep{H\cB_K}_f\|= \|\intregrep{H\cB_K}_f\otimes_R 1\| = \| \intform{\Ind}_H^{\cB_K}(R)_f  \|\leq \| \intregrep{H\cB|K}_f\otimes_{S} 1 \|
    \leq \|\intregrep{H\cB|K}_f \|.
  \end{equation*}
    
  Assume both $H$ and $K$ are normal in $G,$ in which case $\cB$ satisfies  Derighetti's strong condition with respect to $H.$
  Let $T\colon \cB\to \bB(Z)$ and $U\colon H\to \bB(W)$ be *-representations with faithful integrated form.
  Then $R\preceq \Ind_H^\cB(S)|_{\cB_H}\preceq T|_{\cB_H}$ and $R\preceq T|_{\cB_H}\otimes U$ because $U$ weakly contains the trivial representation of $H.$
  Hence, for every $f\in L^1(\cB_K),$ $\|f\|^{\cB_K}_H = \|\intform{\Ind}_H^{\cB_K}(T|_{\cB_H}\otimes U)_f\|$ and Proposition~\ref{prop:inducting tensor product and restriction for normal subgroups} gives $\|f\|^{\cB_K}_H=\|((T\intform{\otimes} \Ind_H^G(U))|_{\cB_K})_f \|=\|(\Ind_H^\cB(T|_{\cB_H}\otimes U)|_{\cB_K})_f \|.$
  Taking $G=K$ we get that for all $f\in L^1(\cB),$ $\|f\|^{\cB}_H=\|\Ind_H^\cB(T|_{\cB_H}\otimes U)_f \|.$
  
  Since the integrated form of $T|_{\cB_H}\otimes U$ is faithful and $\Ind_H^\cB(T|_{\cB_H}\otimes U)|_{\cB_K}=\regrep{H\cB|K}\otimes_{T|_{\cB_H}\otimes U} 1,$ for all $f\in L^1(\cB_K)$ we have
  \begin{equation*}
    \| f \|^{\cB_K}_H =  \|(\Ind_H^\cB(T|_{\cB_H}\otimes U)|_{\cB_K})_f \| =\| \intregrep{H\cB|K}_f\otimes_{T|_{\cB_H}\otimes U} 1 \| = \|\intregrep{H\cB|_K}_f\|;
  \end{equation*}
  which implies  $\intregrep{H\cB|_K}\colon C^*(\cB_K)\to \bB(L^2_H(\cB))$ factors via $q^{\cB_K}_H$ through a faithful *-representation of $C^*_H(\cB_K).$
  In the rest of the proof we identify  $C^*_H(\cB_K)$ with $\intregrep{H\cB|_K}(C^*(\cB_K))\subset \bB(L^2_H(\cB)).$
  
   Given $f\in C_c(\cB_K)$ and $g\in C_c(\cB)$ define the section $f*_K g\colon G\to \cB$ by $f*_K g(s):=\int_K f(t)g(\tmu s)\, dt.$
   We may proceed as in the proof of Proposition~\ref{prop:rep on inducing module} to show that $f*_K g\in C_c(\cB)$ is the integral of $F\colon K\mapsto C_c(\cB),\ F(t)= f(t)g,$ with respect to the inductive limit topology.
   Thus $f*_K g\in C_c(\cB_K).$
   The inclusion map $\nu\colon C_c(\cB)\to L^2_H(\cB)$  is continuous when considering the inductive limit topology and the norm topology, then $f*_Kg =\nu(f*_Kg)=\int_K \nu(f(t)g)\, dt = \int_K \regrep{H\cB|K}_{f(t)}g\, dt=\intregrep{H\cB|K}_f g.$

   We view $C^*_H(\cB)$ as a non degenerate C*-subalgebra of $\bB(L^2_H(\cB))$ thus, as usual, we make the following identification
   \begin{equation*}
   \bB(C^*_H(\cB))\equiv\{M\in \bB(L^2_H(\cB))\colon MC^*_H(\cB)\cup C^*_H(\cB)M\subset C^*_H(\cB) \}.
   \end{equation*}
   
   Given $f\in C_c(\cB_K)$ and $g,h\in C_c(\cB)$ we have $\intregrep{H\cB|K}_f \intregrep{H\cB}_g h = f*_K(g*h).$
   A simple computation shows that $f*_K(g*h)=(f*_Kg)*h.$
   Then $\intregrep{H\cB|K}_f \intregrep{H\cB}_g = \intregrep{H\cB|K}_{f*_Kg}.$
   Also,  $\intregrep{H\cB}_g\intregrep{H\cB|K}_f= (\intregrep{H\cB|K}_{f^*}\intregrep{H\cB}_{g^*})^* = \intregrep{H\cB}_{(f^**Kg^*)^*}.$
   Thus $\intregrep{H\cB|K}_f\in \bB(C^*_H(\cB))$ gets identified with the multiplier $M_f\colon C^*_H(\cB)\to C^*_H(\cB)$ sending $g\in C_c(\cB)$ to $f*_Kg.$
  \end{proof}
\end{theorem}

The partial converse of Corollary~\ref{cor:isomorphisms of q maps} that we can prove is the following.

\begin{corollary}\label{cor:qBKH iso iff qBKH iso}
 Consider a Fell bundle $\cB=\{B_t\}_{t\in G}$ and two normal subgroups of $G,$ $H\sbgp K\sbgp G.$
 If $\cB$ has the $HK-$WCP, then $\cB_K$ has the $H-$WCP.
 \begin{proof}
  We have a commutative diagram
  \begin{equation*}
   \xymatrix{  C^*(\cB_K)\ar[rr]^{\intregrep{K\cB|K}} \ar[d]_{q^{\cB_K}_H} & & \bB(C^*_K(\cB))\ar[d]^{\overline{q^\cB_{KH}}} \\
               C^*_H(\cB_K)\ar[rr]_{\intregrep{H\cB|K}} &  & \bB(C^*_H(\cB)) }
  \end{equation*}
  with the horizontal arrows and the right vertical one being faithful.
  Hence $q^{\cB_K}_H$ is faithful.
 \end{proof}
\end{corollary}

We can now improve Theorem~\ref{thm:homomorphism for reduced cross sectional algebra} as follows.

\begin{corollary}\label{cor:two subgroups and psi maps for tensor product}
 Assume the Fell bundle $\cB=\{B_t\}_{t\in G}$ and the subgroups  $H \sbgp K \sbgp G$ satisfy at least one of the conditions below:
 \begin{enumerate}
  \item $H$ is normal in $G.$
  \item $\cB$ is saturated and satisfies  Derighetti's strong condition with respect to $H.$
 \end{enumerate}
 
 Consider the universal representations
 \begin{align*}
 \regrep{H} & \colon G\to \bB(C^*_H(G))\subset \bB(L^2_H(G)) & 
  \regrep{K\cB}& \colon \cB\to \bB(C^*_K(\cB))\subset \bB(L^2_K(\cB))
 \end{align*}
and form the (minimal) tensor product representation
 \begin{equation*}
  \Psi^\cB_{KH}:= \regrep{H}\otimes\regrep{K\cB} \colon \cB\to \bB( C^*_H(G)\otimes C^*_K(\cB)).
 \end{equation*}
Then there is a unique *-homomorphism $\psi^{\cB}_{KH}$ making 
\begin{equation*}
\xymatrix{ C^*(\cB)\ar[rr]^{\intform{\Psi}^{\cB}_{KH}} \ar[rd]_{\regrep{H\cB}} &  &\bB(C^*_H(G)\otimes C^*_K(\cB))\\
	& C^*_H(\cB)\ar[ur]_{\psi^{\cB}_{KH}} & }
\end{equation*}
a commutative diagram.
Moreover, $\psi^{\cB}_{KH}$ is faithful.
 \begin{proof} 
 Let $V\colon H\to \bB(X)$ be a unitary representation with faithful integrated form, set $X':=L^2_H(G)\otimes_V X$ and $U:=\regrep{H}\otimes_V 1\colon G\to \bB(X').$
 Notice the integrated form of $U$ factors (via $q^G_H$) through a faithful representation $\pi_G\colon C^*_H(\cB)\to \bB(X').$
 Now take a non degenerate *-representation $R\colon \cB_K\to \bB(Y)$ with faithful integrated form and set $T:=\Ind_K^\cB(R)\colon \cB\to \bB(Y'),$ with $Y'=L^2_K(\cB)\otimes_T Y.$
 Remark~\ref{rem:induced and faithful integrated form} implies the existence of a unique faithful and non degenerate representation $\pi_\cB\colon C^*_K(\cB)\to \bB(Y')$ such that $\pi_\cB\circ q^\cB_K=\intform{T}.$
 Thus the tensor product of integrated forms 
 \begin{equation*}
  \pi:=\pi_G\otimes \pi_\cB\colon C^*_H(G)\otimes C^*_K(\cB)\to \bB( X'\otimes Y' )
 \end{equation*}
 is faithful and non degenerate, as it is the extension $\overline{\pi}$ to the multiplier algebra.

 Note $\overline{\pi}\circ \intform{\Psi^\cB_{KH}}= U\otimes T.$ Hence, Corollary~\ref{cor:of FEexell absorption principle} and Proposition~\ref{prop:condition for CHB=CtHtmuB} imply that for all $f\in L^1(\cB)$
 \begin{equation}\label{equ:psiKH factors via q}
  \| \intform{\Psi^\cB_{KH}}(f) \| 
  = \| \overline{\pi}\circ \intform{\Psi^\cB_{KH}}(f) \| 
  = \| (U\otimes T)^{\intform{\ }}_f \|
  \leq \|f\|_H^\cB.
 \end{equation}
 Thus there exists a unique morphism of C*-algebras $\psi^{\cB}_{KH}\colon C^*_H(\cB)\to \bB(C^*_H(G)\otimes C^*_K(\cB))$ such that $\psi^{\cB}_{KH}\circ \intregrep{H\cB}=\intform{\Psi}^\cB_{KH}$ (or, in non integrated forms, $\overline{\psi^{\cB}_{KH}}\circ \regrep{H\cB}=\Psi^{\cB}_{KH}$).
 
 In order to prove $\psi^{\cB}_{KH}$ is faithful we start the proof all over again choosing $R$ and $V$ differently, but we keep all the other constructions.
 For $V$ we now take the trivial representation $\kappa\colon H\to \bC$ and we put $R:=\Ind_H^{\cB_K}(R_0)$ for a non degenerate representation $R_0\colon \cB_H\to \bB(X_0)$ with faithful integrated form.
 Then $T$ is unitary equivalent to $\Ind_H^\cB(R_0).$
 
 The tensor product representation $V\otimes (T|_{\cB_H})$ is unitary equivalent to $\kappa \otimes (\Ind_H^{\cB}(R_0)|_{\cB_H}),$ which weakly contains $R_0$ (because $\cB$ satisfies Derighetti's strong condition with respect to $H$).
 Then the integrated form of $V\otimes (T|_{\cB_H})$ is a faithful representation of $C^*(\cB_H).$
 Now Corollary~\ref{cor:Fell absorption principle II bis} gives the inequalities
 \begin{equation*}
 \| \intregrep{H\cB}_f \|=\|f\|_H^\cB  = \| (U\otimes T)^{\intform{\ }}_f \| = \| \pi\circ \intform{\Psi^\cB_{KH}}(f) \| \leq \| \intform{\Psi^\cB_{KH}}(f)  \|
 = \|\intform{\psi}^{\cB}_{KH}(\intregrep{H\cB}_f)\|\leq \| \intregrep{H\cB}_f \|.
 \end{equation*}
 for all $f\in L^1(\cB).$
 Thus $\intform{\psi}^{\cB}_{KH}$ is faithful.
 \end{proof}
\end{corollary}

As mentioned in the introduction, it is known that every Fell bundle over an amenable group has the WCP.
The result below is a generalisation of this fact.

\begin{corollary}\label{cor:amenability of G implies that of B}
	Let $\cB=\{B_t\}_{t\in G}$ be a Fell bundle and let $H \sbgp  K \sbgp G$ be such that $G$ has the $HK-$WCP.
	If at least one of the conditions below is satisfied, then $\cB$ has the $HK-$WCP.
\begin{enumerate}
 \item Both $H$ and $K$ are normal in $G.$
 \item $\cB$ is saturated and satisfies  Derighetti's strong condition with respect to both $H$ and $K.$
\end{enumerate}
 \begin{proof}
 	Consider the diagram
 \begin{equation*}
 \xymatrix{ C^*_K(\cB)\ar[rr]^(.3){\psi^\cB_{GK}} \ar[dd]_{q^\cB_{KH}}& & \bB(C^*_K(G)\otimes C^*(\cB))\ar[dd]^{\overline{ q^G_{KH}\otimes 1}}\\
 	 & C^*(\cB)\ar[lu]^{q^\cB_K}\ar[dl]_{q^\cB_H}\ar[ur]_{\Psi^\cB_{GK}} \ar[dr]^{\Psi^\cB_{GH}} & \\
 	C^*_H(\cB)\ar[rr]_(.3){\psi^\cB_{GH}}& & \bB( C^*_H(G)\otimes C^*(\cB))} 
 \end{equation*}
 Lets analyze the four triangles with an horizontal or vertical arrow and  opposite vertex in $C^*(\cB).$
 The triangles on the top and bottom commute by Corollary~\ref{cor:two subgroups and psi maps for tensor product} (recall $q^\cB_H$ is the integrated form of $\intregrep{H\cB}$). 
 The one on the right commutes because it does if we replace the $\Psi$ maps with their non integrated forms.
 Finally, the definition of $q^\cB_{KH}$ implies the triangle on the left commutes.
 
 The external arrows form a commutative diagram because $q^\cB_K$ is surjective, and the two horizontal arrows and the right vertical one are faithful, so $q^\cB_{KH}$ is faithful. 
 \end{proof}
\end{corollary}

\subsection{The Kaliszewski-Landstad-Quigg cross sectional C*-algebra}

In this short (sub)section we introduce $E-$cross sectional C*-algebras following the definition of $E-$crossed products of \cite{KaLAQu2013}; where $E$ is a non zero $G-$invariant weak*-closed ideal of the Fourier-Stielltjes algebra of $G,$ $B(G)$ (see \cite{KaLAQu2013} for the definitions and notation).
We then relate these cross sectional C*-algebras to our $H-$cross sectional C*-algebras.

Let $\cB=\{B_t\}_{t\in G}$ be a Fell bundle and $E\subset B(G)$ as before.
By identifying $B(G)$ with the dual space $C^*(G)',$ Kaliszewski, Landstad and Quigg define the (exotic) group C*-algebra $C^*_E(G)$ as the quotient of $C^*(G)$ by the pre-annihilator ${}^\perp E$ of $E.$
The natural quotient map $q_E\colon C^*(G)\to C^*_E(G)$ can be extended to the multiplier algebras as $\overline{q_E}\colon \bB(C^*(G))\to \bB(C^*_E(G))$ and the composition $\regrep{EG}:=\overline{q_E}\circ \Lambda^G\colon G\to \bB(C^*_E(G))$ is a unitary representation with integrated form $q_E.$

The map $\regrep{\cB}\otimes \regrep{EG}\colon \cB\to \bB(C^*(\cB)\otimes C^*_E(G))$ sending $b\in B_t$ to $\regrep{\cB}_b\otimes \regrep{EG}_t$ is a *-representation that we can integrate to define the ideal $J_{\cB,E}:=\ker(\regrep{\cB}\intform{\otimes} \regrep{EG})$ and the $E-$cross sectional C*-algebra of $\cB$ as
\begin{equation*}
C^*_E(\cB):=C^*(\cB)/J_{\cB,E}\equiv \regrep{\cB}\intform{\otimes} \regrep{EG}(C^*(\cB))\subset \bB(C^*(\cB)\otimes C^*_E(G)).
\end{equation*}

Fix a closed subgroup $H$ of $G.$
Then $C^*_H(G)=C^*_{E_H}(G) $ with $E_H:=\ker(\intregrep{HG})^\perp$ (the annihilator).
As usual, we identify the dual space $C^*_H(G)'$ with $E_H.$
Note that $\varphi\in B(G)$ belongs to $E_H$ if and only if there exists a unitary representation $U\colon G\to \bB(X)$ and vectors $\xi,\eta\in X$ such that (i) $\varphi(t)=\langle \xi,U_t\eta\rangle,$ for all $t\in G;$ and (ii) $U$ is weakly contained in a *-representation induced from $H.$

Consider the subspace $E^0_H$ formed by all those functions $\varphi\colon G\to \bC$ of the form
\begin{equation*}
\varphi(t)= \langle \xi,\Ind_H^G(U)\eta\rangle
\end{equation*}
for some unitary representation $U\colon H\to \bB(X)$ and vectors $\xi,\eta\in L^2_H(G)\otimes_U X.$
Then ${}^\perp E^0_H=\ker(\intregrep{HG})$  and $E_H=\overline{E^0_H}^{w^*}$ (weak*-closure).
In fact, by Theorem~\ref{thm:foundamental facts on weak containment}, $E_H$ is the $\|\ \|_\infty-$closure of $E_H^0.$

A natural question arises: is there any relation between $C^*_H(\cB)$ and $C^*_{E_H}(\cB)$?
Well, for the trivial bundle $\cB=\cT_G\equiv G$  we obviously have $C^*_H(\cB)=C^*_H(G)=C^*_{E_H}(G)=C^*_{E_H}(\cB),$ but we do not know if the identity $C^*_H(\cB)=C^*_{E_H}(\cB)$ holds for every Fell bundle over $G.$
It does if $\cB$ is saturated or $H$ is normal in $G,$ as we shall see.

Clearly, $\regrep{E_HG}\equiv \regrep{HG}.$
Thus $\regrep{\cB}\otimes \regrep{E_HG}=\regrep{\cB}\otimes \regrep{HG}.$
Consider non degenerate *-representations $T\colon \cB\to \bB(X)$ and $U\colon H\to \bB(Y)$ with faithful integrated forms.
Then we may think
\begin{align*}
C^*(\cB)&\subset \bB(X)    & C^*_H(G&)=\intform{\Ind}_H^G(C^*(G))\subset \bB(L^2_H(G)\otimes_U Y),
\end{align*}
which gives $C^*_{E_H}(\cB)\subset \bB(X\otimes (L^2_H(G)\otimes_U Y)).$
More precisely, $C^*_{E_H}(\cB)$ gets identified with the image of $C^*(\cB)$ under the integrated form of
\begin{equation*}
(b\in B_t)\mapsto \regrep{\cB}_b\otimes \regrep{E_HG}_t \equiv T_b\otimes \Ind_H^G(U)_t.
\end{equation*}

\begin{remark}
Combining \cite[VIII 21.24 \& XI 12.4]{FlDr88} we get that $T\otimes\Ind_H^G(U)$ weakly contains any other *-representation of the form $S\otimes V;$ $S$ being a non degenerate *-representation of $\cB$ and $V$ a unitary representation of $G$ weakly contained in some unitary representation induced from $H.$
\end{remark}

By construction and Theorem~\ref{thm:Fells absorption principle I},
\begin{equation*}
J_{\cB,E_H} =\ker ( T\intform{\otimes}\Ind_H^G(U))
=\bigcap_{t\in G}\ker (\intform{\Ind}_{tHt^{-1}}^\cB({}_{t|}T\otimes {}_t U))\supset \bigcap_{t\in G}\ker (\intregrep{tHt^{-1}\cB}).
\end{equation*}
If either $H$ is normal in $G$ or $\cB$ is saturated, then Corollary~\ref{cor:Fell absorption principle II bis} gives
\begin{equation*}
J_{\cB,E_H}\supset \bigcap_{t\in G}\ker (\intregrep{tHt^{-1}\cB})= \ker(\intregrep{H\cB})
\end{equation*}

Moreover, if either $H$ is normal in $G$ or $\cB$ is saturated and satisfies Derighetti's strong condition with respect to $H,$ then Corollary~\ref{cor:two subgroups and psi maps for tensor product} gives $J_{\cB,E_H}=\ker(\intregrep{H\cB})$ and $C^*_{E_H}(\cB)=C^*_H(\cB).$

\section{Induction from subgroups and via equivalence bundles}\label{sec:induction and Morita}

In this section we use the notion of strong equivalence of Fell bundles presented in \cite{abadie2019morita} together will all the accessory constructions of \cite{AbFrrEquivalence} (specially linking Fell bundles).
Given Fell bundles $\cA=\{A_t\}_{t\in G}$ and $\cB=\{B_t\}_{t\in G}$ and an $\cA-\cB-$(strong) equivalence bundle $\cX=\{X_t\}_{t\in G},$ there exists a $C^*(\cA)-C^*(\cB)-$equivalence module $C^*(\cX)$ which can be constructed by completing $C_c(\cX)$ inside the cross sectional C*-algebra of the linking Fell bundle $\bL(\cX).$
In fact $C^*(\cA)$ and $C^*(\cB)$ are full corners of $C^*(\bL(\cX))$ and $C^*(\cX)=C^*(\cA)C^*(\bL(\cX))C^*(\cB).$

Given a non degenerate *-representation $T\colon \cB\to \bB(X_A),$ we can induce the integrated form $\intform{T}\colon C^*(\cB)\to \bB(X_A)$ via $C^*(\cX)$ to a non degenerate *-representation
\begin{equation*}
 \Ind_{C^*(\cX)}(\intform{T})\colon C^*(\cA)\to \bB(C^*(\cX)\otimes_{\intform{T}}X_A),
\end{equation*}
that can be disintegrated (uniquely) to a *-representation we denote
\begin{equation*}
\Ind_\cX(T)\colon \cA\to \bB(C^*(\cX)\otimes_{\intform{T}}X_A).
\end{equation*}

\begin{proposition}\label{prop:induction and Morita-Rieffel equivalence}
 Let $\cA=\{A_t\}_{t\in G}$ and $\cB=\{B_t\}_{t\in G}$ be Fell bundles, $H \sbgp G$ and $\cX=\{X_t\}_{t\in G}$ an $\cA-\cB-$equivalence bundle establishing a strong Morita-Rieffel equivalence.
 Then the reduction $\cX_H=\{X_t\}_{t\in H}$ establishes a strong Morita-Rieffel equivalence between $\cA_H$ and $\cB_H.$
 
 For every *-representation $T\colon \cB_H\to \bB(Y_C)$ we have $\Ind_H^\cA( \Ind_{\cX_H}(T)) \approx \Ind_{\cX}(\Ind_H^\cB(T)).$
 \begin{proof}
  The claim about $\cX_H=\{X_t\}_{t\in H}$ being a strong $\cA_H-\cB_H-$strong equivalence bundle follows immediately form the definition of the strong equivalence.
  There is a very important consequence that we want to extract from the strong equivalence assumption.
  We know that, for all $t\in G,$ $\cspn {}_\cA\langle X_t,X_t\rangle = \cspn A_tA_t^*.$
  Besides, considering $X_t$ as a left $A_e-$Hilbert module we conclude that
  \begin{equation}\label{equ:Xt=AtXe}
   X_t = \cspn {}_\cA\langle X_t,X_t\rangle X_t \subset \cspn A_t A_{t^{-1}} X_t\subset \cspn A_t X_e \subset X_t.
  \end{equation}
  It then follows (by symmetry) that $\cspn A_t X_e = X_t =\cspn X_e B_t$ for all $t\in G.$

  The construction itself of the linking Fell bundle $\bL(\cX)$ implies that we may regard $\bL(\cX)_H$ as $\bL(\cX_H).$
  Besides $\cA,$ $\cB$ and $\cX$ are sub Banach bundles of $\bL(\cX)$ and the equivalence bundle structure of $\cX$ comes from the natural Fell bundle structure of $\bL(\cX).$ 
  It is then convenient to work with the restriction map $p\colon C_c(\bL(\cX))\to C_c(\bL(\cX)_H)$ given by $p(f)(t)=\Delta_G(t)^{1/2}\Delta_H(t)^{-1/2}f(t)$ and to consider it as the extension of the restriction maps of $\cA$ and $\cB.$
  Since $\cB_H$ sits inside $\bL(\cX)_H$ as a full hereditary subbundle, we may regard $L^2_H(\cB)$ as the closure of $C_c(\cB)$ in $L^2_H(\bL(\cX)).$
  This implies, for example, that expressions like $p(f^** g)*u=p(f^**(gu))$ make sense and hold for $f,g\in C_c(\cB)$ and $u\in C_c(\cX_H),$ where $gu$ is the element produced by the action of $u\in C_c(\cX_H)\subset C^*(\bL(\cX)_H)$ over $g\in C_c(\cB)\subset L^2_H(\bL(\cX)).$
  We will use other identities of the sort that we can deduce from \eqref{equ:ref properties of p}.
  
  Define $L^2_H(\cX)$ as the closure of $C_c(\cX)$ in $L^2_H(\bL(\cX)).$  
  For all $f,g\in C_c(\cX)$ we have $f^**g\in C_c(\cB),$ $p(f^**g)$ is positive in $C^*(\bL(\cX)_H)$ and is contained in $C^*(\cB_H),$ which is a C*-subalgebra of $C^*(\bL(\cX)_H).$
  Thus $\langle L^2_H(\cX),L^2_H(\cX)\rangle \subset C^*(\cB_H).$
  Besides, for all $f\in C_c(\cX)$ and $u\in C_c(\cB_H)$ we have $fu\in C_c(\cX).$
  Hence $L^2_H(\cX)C^*(\cB_H)\subset L^2_H(\cX)$ and we may think of $L^2_H(\cX)$ as a right $C^*(\cB_H)-$Hilbert module with the structure inherited from $L^2_H(\bL(\cX)).$
  
  Given $a\in \cA$ and $f\in C_c(\cX),$ we have $\regrep{H\bL(\cX)}_a f\in C_c(\cX).$
  Thus $\regrep{H\bL(\cX)}(\cA)L^2_H(\cX)\subset L^2_H(\cX)$ and we have a non degenerate *-representation $\regrep{H\cX}\colon \cA\to \bB(L^2_H(\cX))$ such that $\regrep{H\cX}_af = \regrep{H\bL(\cX)}_a f.$
  
  Fix, for the rest of the proof, a representation $T\colon \cB_H\to \bB(Y_C).$
  For convenience we set: $S:=\Ind_{\cX_H}(T);$ $T':=\Ind_H^\cB(T)$ and $S':=\Ind_H^\cA(S).$
  We must show that $S'$ is unitary equivalent to $\Ind_{\cX}(T'),$ which we will do by proving that both of them are unitary equivalent to the *-representation
  \begin{equation*}
   \regrep{H\cX}\otimes_{\intform{T}} 1\colon \cA\to \bB( L^2_H(\cX)\otimes_{\intform{T}} Y_C ),\ a\mapsto \regrep{H\cX}_a\otimes_{\intform{T}} 1.
  \end{equation*}  
  
  The underlying Hilbert modules of $S'$ and $\Ind_\cX(T')$ are, respectively,
  \begin{align*}
   L^2_H(\cA)\otimes_{\intform{S}} &(C^*(\cX_H)\otimes_{\intform{T}}Y_C)  & C^*(\cX)\otimes_{\intform{T'}} & (L^2_H(\cB)\otimes_{\intform{T}} Y_C)
  \end{align*}
  
  Take $f,g\in C_c(\cA)\subset L^2_H(\cA),$ $u,v\in C_c(\cX_H)\subset C^*(\cX_H)$ and $\xi,\eta\in Y_C.$
  We denote $ f\otimes_{\intform{S}} u\otimes_{\intform{T}} \xi $ the image of the elementary algebraic tensor $f\odot u\odot \xi$ in $ L^2_H(\cA)\otimes_{\intform{S}} (C^*(\cX_H)\otimes_{\intform{T}}Y_C).$
  Then
  \begin{multline}\label{equ:first inner product double induction}
   \langle f\otimes_{\intform{S}} u\otimes_{\intform{T}} \xi,  g\otimes_{\intform{S}} v\otimes_{\intform{T}} \eta\rangle
      = \langle  u\otimes_{\intform{T}} \xi, \intform{S}_{p(f^**g) } ( v\otimes_{\intform{T}} \eta)\rangle\\
      = \langle  u\otimes_{\intform{T}} \xi,  p(f^**g) * v\otimes_{\intform{T}} \eta)\rangle
      = \langle \xi, \intform{T}_{ u^* * p(f^**g) * v }\eta\rangle
      = \langle \xi, \intform{T}_{  p((fu)^**(gv)) }\eta\rangle.
  \end{multline}

  Now take $h,k\in C_c(\cX)\subset C^*(\cX),$ $z,w\in C_c(\cB)\subset L^2_H(\cB)$ and $\xi,\eta\in Y_C.$
  Then
  \begin{multline}\label{equ:second inner product double induction}
   \langle h\otimes_{\intform{T'}} z\otimes_{\intform{T}} \xi,  k\otimes_{\intform{T'}} w\otimes_{\intform{T}} \eta\rangle 
      = \langle z\otimes_{\intform{T}} \xi ,  \intform{T'}_{ h^**k } ( w\otimes_{\intform{T}} \eta)  \rangle\\
       = \langle z\otimes_{\intform{T}} \xi ,   (h^**k  *w)\otimes_{\intform{T}} \eta  \rangle
      = \langle \xi,\intform{T}_{ p(z^* * h^* * k * w) }  \eta\rangle
       = \langle \xi,\intform{T}_{ p((h* z)^* * (k * w)) }  \eta\rangle.
  \end{multline}

  Expressions \eqref{equ:first inner product double induction} and \eqref{equ:second inner product double induction} seem to represent the same kind of inner product.
  The products $h*z$ and $k*w$ in \eqref{equ:second inner product double induction} belong to $C_c(\cX)$ and span a dense subset of $C_c(\cX)$ with respect to the inductive limit topology.
  To prove this one may notice that if one takes $X=\cB,$ then the desired result follows from the construction of approximate units of $L^1(\cB)$ contained in $C_c(\cB)$ of \cite{FlDr88}.
  Then one can easily use that proof to fit the general case.
  On the other hand, the elements $fu$ in \eqref{equ:first inner product double induction} belong to $C_c(\cX)$ and we claim they expand a subset of $C_c(\cX)$ that is dense in the inductive limit topology.
  
  By \cite[II 14.6]{FlDr88} and \eqref{equ:Xt=AtXe}, the sections of the form $t\mapsto f(t)x$ for $f\in C_c(\cA)$ and $x\in X_e$ span a dense subspace of $C_c(\cX),$ with respect to the inductive limit topology.
  Fix $f\in C_c(\cA)$ and $x\in X_e.$
  Take $g\in C_c(X_H)$ such that $g(e)=x.$
  Given a compact neighbourhood $V\subset H$ of $e$ let $\varphi_V\in C_c(H)$ such that $\supp(\varphi_V)\subset V$ and $\int_H \varphi_V(r^{-1}) \Delta_G(r^{-1})^{1/2}\Delta_H(r)^{-1/2}\, dr = 1.$
  The support of $f(\varphi_Vg)$ is contained in $\supp(f)V$ and for all $t\in \supp(f)V$
  \begin{align*}
   \|f(t)x - fu(t)\|
     & \leq \int_H \| f(t)x-  f(tr)g(r^{-1})\|\varphi_V(r^{-1})\Delta_G(r)^{1/2}\Delta_H(r)^{-1/2}\, dr\\
     & \leq \mu(V) \sup\{\|f(s)x-f(sr)g(r^{-1})\|\colon s\in \supp(f),\ r\in V\},
  \end{align*}
  where $\mu(V)$ is the measure of $V\subset H.$
  The expression on the right hand side above goes to $0$ as $V$ decreases to $\{e\}.$
  Thus $t\mapsto f(t)x$ lays in the closure (in the inductive limit topology) of the functions $fu$ for $u\in C_c(\cX_H).$
  
  The discussion above, together with \eqref{equ:first inner product double induction}, \eqref{equ:second inner product double induction} and claim (3) in Proposition~\ref{prop:rep on inducing module}, imply the existence of unitary operators
  \begin{align*}
   U\colon  L^2_H(\cA)\otimes_{\intform{S}} (C^*(\cX_H)\otimes_{\intform{T}}Y_C)& \to L^2_H(\cX)\otimes_{\intform{T}} Y_C  & f\otimes_{\intform{S}} u\otimes_{\intform{T}} \xi & \mapsto fu\otimes_{\intform{T}}\xi\\
   V\colon  C^*(\cX)\otimes_{\intform{T'}} (L^2_H(\cB)\otimes_{\intform{T}} Y_C)& \to L^2_H(\cX) \otimes_{\intform{T}} Y_C & h\otimes_{\intform{T'}} z\otimes_{\intform{T}} \xi & \mapsto h*z\otimes_{\intform{T}}\xi.
  \end{align*}
  
  For all $a\in \cA$ and $f\otimes_{\intform{S}} u\otimes_{\intform{T}} \xi$ as above we have
  \begin{align*}
   U^* \regrep{H\cX}_a U (f\otimes_{\intform{S}} u\otimes_{\intform{T}} \xi)
     & = U^*(a(fu)\otimes_{\intform{T}} \xi)   
       = U^*((af)u\otimes_{\intform{T}} \xi)
       = af\otimes_{\intform{S}} u\otimes_{\intform{T}} \xi\\
     & = S'_a(f\otimes_{\intform{S}} u\otimes_{\intform{T}} \xi).
   \end{align*}
  A similar computation shows that $V^*\regrep{H\cX}_a V = \Ind_{\cX}(T')_a,$ completing the proof. 
 \end{proof}
\end{proposition}

Form the proof above we can extract the following conclusion.

\begin{corollary}
 The right $C^*(\cB_H)-$Hilbert module $L^2_H(\cX)$ carries a canonical representation of $\cA,$ $\regrep{H\cX}\colon \cA\to \bB(L^2_H(\cX)),$ that induces representations from $\cB_H$ to representations of $\cA.$
 This induction process is unitary equivalent to any of the compositions of inductions
 \begin{align}\label{equ:composed inductions}
  T & \mapsto \Ind_H^\cB(T)\mapsto \Ind_\cX(\Ind_H^\cB(T))  & T&\mapsto \Ind_{\cX_H}(T)\mapsto \Ind_H^\cA(\Ind_{\cX_H}(T)).
 \end{align}
 Moreover, the equivalence holds because $L^2_H(\cX)$ is unitary equivalent to the modules
 \begin{align*}
  C^*(\cX)\otimes_{C^*(\cB)} &L^2_H(\cB) & L^2_H(\cA)&\otimes_{C^*(\cA_H)} C^*(\cX_H),
 \end{align*}
 which are the modules that give the composed induction processes of~\eqref{equ:composed inductions}.
\end{corollary}

Suppose we are given two C*-dynamical systems, $(A,G,\alpha)$ and $(B,G,\beta),$ which are Morita-Rieffel equivalent via an action $\gamma$ of $G$ on a $A-B-$equivalence bimodule $X.$
It is a known fact that the full crossed products $A\rtimes_\alpha G$ and $B\rtimes_\beta G$ are Morita-Rieffel equivalent.
Moreover, it is also known that the reduced crossed products $A\rtimes_{\red\alpha} G$ and $B\rtimes_{\red\beta} G$ are Morita-Rieffel equivalent.

The semidirect product bundle of $\alpha,$ $\cB_\alpha,$ is a saturated Fell bundle such that $C^*(\cB_\alpha)=A\rtimes_{\alpha}G$ and $C^*_\red(\cB_\alpha)=A\rtimes_{\red\alpha}G.$
We can even construct other crossed products $A\rtimes_{H\alpha}G:=C^*_H(\cB_\alpha)$ associated to subgroups $H \sbgp G.$
As shown in \cite{AbFrrEquivalence,abadie2019morita}, the semidirect product (equivalence) bundle of $\gamma$ establishes a strong equivalence between $\cB_\alpha$ and $\cB_\beta.$
This puts us in the situation of the result below and so we get a Morita-Rieffel equivalence between $A\rtimes_{H\alpha }G$ and $B\rtimes_{H\beta }G.$

\begin{corollary}\label{cor:strong equivalence and H algebras for saturated}
 If $\cX$ is an $\cA-\cB-$strong equivalence bundle, then the C*-algebras $C^*_H(\cA)$ and $C^*_H(\cB)$ may be identified with the closures of $C_c(\cA)$ and $C_c(\cB)$ in $C^*_H(\bL(\cX)),$ respectively, and both closures are full hereditary C*-subalgebras of $C^*_H(\bL(\cX)).$
 Thus $C^*_H(\cA),$ $C^*_H(\cB)$ and $C^*_H(\bL(\cX))$ are Morita-Rieffel equivalent to each other.
 \begin{proof}
 Take a non degenerate *-representation $T\colon \bL(\cX)_H\to \bB(Y)$ with faithful integrated form.
 The integrated form of the restriction $S\colon \cB_H\to \bB(Y),$ $b\mapsto T_b,$ is faithful (an possibly degenerate) because it is the restriction of $\intform{T}\colon C^*(\bL(\cX)_H)\to \bB(Y)$ to $C^*(\cB).$
 If we regard $C^*(\bL(\cX)_H)\equiv C^*(\bL(\cX_H)) = \bL(C^*(\cX_H))$ as in \cite[Corollary 4.10]{AbFrrEquivalence} and induce $\intform{S}\colon C^*(\cB_H)\to \bB(Y)$ to a *-representation of $\bL(C^*(\cX_H))$ via the $\bL(C^*(\cX_H))-C^*(\cB_H)-$equivalence module 
 \begin{equation*}
\cspn \bL(C^*(\cX_H)) C^*(\cB_H) = C^*(\cX_H)\oplus C^*(\cB_H),
 \end{equation*}
 then we obtain a *-representation that is unitary equivalent to $\intform{T}.$
   In symbols this gives the unitary equivalence $\Ind_{C^*(\cX_H)\oplus C^*(\cB_H)}(\intform{S})\approx \intform{T},$ where $C^*(\cX_H)\oplus C^*(\cB_H)\subset \bL(C^*(\cX_H))\equiv C^*(\bL(\cX)_H)$ is the direct sum implementing the equivalence between $C^*(\bL(\cX)_H)$ and $C^*(\cB_H).$
   But $C^*(\cX_H)\oplus C^*(\cB_H)=C^*(\cX_H\oplus \cB_H),$  where the (bundle) direct sum  $\cX_H\oplus\cB_H\equiv (\cX\oplus \cB)_H$ is formed inside $\bL(\cX)_H.$
   Thus the unitary equivalence of integrated forms $\Ind_{C^*(\cX_H)\oplus C^*(\cB_H)}(\intform{S})\approx \intform{T}$ passes to disintegrated forms (by \cite[II 11.15]{FlDr88}) and can be used to obtain
   \begin{equation*}
    \Ind_{\cX_H\oplus \cB_H}( T|_{\cB_H} ) \approx T.
   \end{equation*}
   
   By Proposition~\ref{prop:induction and Morita-Rieffel equivalence} and Remark~\ref{rem:induced and faithful integrated form}, for all $g\in C_c(\cB)\subset C_c(\bL(\cX))$ we have
   \begin{align*}
    \| g \|_H^{\bL(\cX)}
     &  = \| \intform{\Ind}_H^{\bL(\cX)} (T)_g \|
        = \| \intform{\Ind}_H^{\bL(\cX)} ( \Ind_{\cX_H\oplus \cB_H}( T|_{\cB_H} ))_g \|\\
     &  = \| \intform{\Ind}_{\cX\oplus \cB} ( \Ind_H^\cB( T|_{\cB_H} ))_g \|.
   \end{align*}
   
   Notice that every $b\in \cB\subset \bL(\cX)$  acts as the null operator on the copy of $\cX$ inside $\cX\oplus \cB\subset \bL(\cX).$
   Thus, $\intform{\Ind}_{\cX\oplus \cB} ( \Ind_H^\cB( T|_{\cB_H} ))_b$ acts as the null operator on the copy of $L^2_H(\cX)\otimes_{\Ind_H^\cB(T|_{\cB_H})}Y$ inside $L^2_H(\cX\oplus \cB)\otimes_{\Ind_H^\cB(T|_{\cB_H})}Y.$
   Hence, the direct sum decomposition
   \begin{equation*}
     L^2_H(\cX\oplus \cB)\otimes_{\Ind_H^\cB(T|_{\cB_H})}Y =\left(L^2_H(\cX)\otimes_{\Ind_H^\cB(T|_{\cB_H})}Y\right)\oplus \left(L^2_H(\cB)\otimes_{\Ind_H^\cB(T|_{\cB_H})}Y\right)
   \end{equation*}
   gives a decomposition
   \begin{equation*}
    \intform{\Ind}_{\cX\oplus \cB} ( \Ind_H^\cB( T|_{\cB_H} ) = 0\oplus \intform{\Ind}_{\cB} ( \Ind_H^\cB( T|_{\cB_H} )).
   \end{equation*}
   But the induction from $\cB\equiv \cB_G$ to $\cB$ preserves the unitary equivalence class of the representation, which yields
      \begin{equation*}
   \intform{\Ind}_{\cX\oplus \cB} ( \Ind_H^\cB( T|_{\cB_H} )) 
   = 0\oplus  \Ind_H^\cB( T|_{\cB_H} ).
   \end{equation*}
   Hence, for all $g\in C_c(\cB)$ we have
   \begin{align*}
    \| g \|_H^{\bL(\cX)}
     &  = \| \intform{\Ind}_{\cB} ( \Ind_H^\cB( T|_{\cB_H} ))_g \|
       =  \|  \Ind_H^\cB( T|_{\cB_H} )_g \|
       = \|g\|_H^\cB.
   \end{align*}
   
   The identity above implies we may regard $C^*_H(\cB)$ as the closure of $C_c(\cB)$ in $C^*_H(\bL(\cX))$ and, by symmetry, $C^*_H(\cA)$ gets identified with the closure of $C_c(\cA).$
   From now on the proof is a straightforward adaptation of \cite[Theorem 4.5]{AbFrrEquivalence}.
 \end{proof}
\end{corollary}

The strong equivalence hypothesis in the Corollary above can be replaced by a weak equivalence (in the sense of \cite{abadie2019morita}) at the expense of assuming that $H$ is normal in $G,$ as we shall see after some preparatory results.
To prove this claim we adapt \cite[Proposition 3.2]{Ab03} to cross sectional C*-algebras other than the reduced one; which corresponds to $H=\{e\}$ (and motivates conditions 1 and 2 below).

\begin{lemma}\label{lem:subbundle and sub cross sectional algebras}
 Let $\cA=\{A_t\}_{t\in G}$ be a sub Fell bundle of $\cB=\{B_t\}_{t\in G}$ and $H \sbgp G$ such that
 \begin{enumerate}
  \item Every Fell bundle over $G$ satisfies Derighetti's strong condition with respect to $H,$ like it is the case if $H$ is either normal, open or has open normaliser in $G.$
  \item The canonical morphism of C*-algebras $C^*(\cA_H)\to C^*(\cB_H)$ extending the inclusion $L^1(\cA_H)\subset C^*(\cB_H)$ is faithful.
 \end{enumerate}
 Then $C^*_H(\cA)$ is C*-isomorphic to the closure fo $C_c(\cA)$ in $C^*_H(\cB),$ which gives an inclusion of C*-algebras $C^*_H(\cA)\subset C^*_H(\cB).$
 \begin{proof}
  Let $S\colon \cB_H\to \bB(Y)$ be a non degenerate *-representation with faithful integrated form and $T\colon \cB\to \bB(Z)$ the representation induced by $S.$ 
  Then $\intform{T|_{ \cB_H}}\colon C^*(\cB_H) \to \bB(Z)$ is faithful because it weakly contains $\intform{S}.$
  
  Consider the trivial representation $\kappa\colon H\to \bC$ and the representation of $G$ induced by it,  $\lambda \colon G\to \bB(L^2_H(G)_\kappa).$
  Since the integrated form of $\kappa\otimes (T|_{\cB_H})\approx T|_{\cB_H}$ is a faithful representation of $C^*(\cB_H),$ Corollary~\ref{cor:Fell absorption principle II bis} implies that we may think of $C^*_H(\cB)$ as $(\lambda\otimes T)^{\intform{\ }}(C^*(\cB))\subset \bB(L^2_H(G)_\kappa\otimes Y).$
  Under this identification the closure of $C_c(\cA)$ in $C^*_H(\cB)$ becomes the closure of
  \begin{equation*}
   I_c:=  \{(\lambda\otimes T)^{\intform{\ }}_f\colon f\in C_c(\cA)\} = \{(\lambda\otimes (T|_{\cA}))^{\intform{\ }}_f\colon f\in C_c(\cA)\}\subset \bB(L^2_H(G)_\kappa\otimes Y).
  \end{equation*}

  The identity above implies we should consider the restriction $R:=T|_{\cA}\colon \cA\to \bB(Y)$ and it's essential part $R'\colon \cA\to \bB(Y').$
  We may regard $L^2_H(G)_\kappa \otimes Y'$ as a closed subspace of $L^2_H(G)_\kappa\otimes Y$ which is $I_c-$invariant.
  Notice that $(\lambda\otimes T)(\cA)(L^2_H(G)_\kappa \otimes Y)\subset L^2_H(G)_\kappa \otimes Y'.$
  Hence $I_c(L^2_H(G)_\kappa \otimes Y)\subset L^2_H(G)_\kappa \otimes Y'$ and this implies the elements of $I_c$ vanish in $(L^2_H(G)_\kappa \otimes Y')^\perp.$
  Consequently, $I_c$ is isometrically *-isomorphic to 
  \begin{equation*}
   I'_c= \{(\lambda\otimes R')^{\intform{\ }}_f\colon f\in C_c(\cA)\}\subset \bB(L^2_H(G)_\kappa\otimes Y').
  \end{equation*}
  
  Notice that $R'$ is non degenerate and that the integrated form of $\kappa\otimes (R'|_{\cA_H})= R'|_{\cA_H}$ is faithful because it is the restriction of the non degenerate part of $\intform{T|_{\cB_H}}$ to $C^*(\cA_H)\subset C^*(\cB_H).$
  Thus Corollary~\ref{cor:Fell absorption principle II bis} implies the closure of $I'_c$ (and hence of $I_c$) is C*-isomorphic to $C^*_H(\cA).$
 \end{proof}
\end{lemma}

\begin{corollary}\label{cor:equivalence and H algebras for non saturated}
 Let $\cA$ and $\cB$ be Fell bundles over $G$ which are weakly equivalent via a Hilbert bundle $\cX=\{X_t\}_{t\in G}.$
 If $H \sbgp G$ is such that every Fell bundle over $G$ satisfies Derighetti's strong condition with respect to $H$ (e.g. if $H$ is either normal, open or has open normaliser in $G$) then $C^*_H(\cA)$ and $C^*_H(\cB)$ are canonically C*-isomorphic to the closures of $C_c(\cA)$ and $C_c(\cB)$ in $C^*_H(\bL(\cX)),$ respectively.
 This describes $C^*_H(\cA)$ and $C^*_H(\cB)$ as full hereditary C*-subalgebras of $C^*_H(\bL(\cX))$ and, consequently, they are Morita-Rieffel equivalent.
 \begin{proof}
  Notice that $\cB_H$ is hereditary in $\bL(\cX)_H$ in the sense of \cite[Theorem 4.3]{AbFrrEquivalence}, thus the inclusion $L^1(\cB_H)\subset C^*(\bL(\cX)_H)$ extends to $C^*(\cB_H)\subset C^*(\bL(\cX)_H).$
  After this we can use Lemma~\ref{lem:subbundle and sub cross sectional algebras} and the comments we made after Definition~\ref{defi:strong Derighetti} to adapt the proof of \cite[Theorem 4.5]{AbFrrEquivalence}.
 \end{proof}
\end{corollary}

\begin{theorem}\label{thm:induction via CKX}
  Let $\cX=\{X_t\}_{t\in G}$ be a bundle establishing a weak Morita-Rieffel equivalence between the Fell bundles $\cA=\{A_t\}_{t\in G}$ and $\cB=\{B_t\}_{t\in G},$ consider subgroups $H \sbgp K \sbgp G$ and assume that at least one of the following conditions is satisfied:
  \begin{enumerate}
  	\item $\cX$ is a strong equivalence bundle.
  	\item  Every Fell bundle over $G$ satisfies Derighetty's strong condition with respect to $H$ and $K.$
  \end{enumerate}
  By Corollaries~\ref{cor:strong equivalence and H algebras for saturated} and~\ref{cor:equivalence and H algebras for non saturated}, we may think $C^*_K(\cA)$ and $C^*_K(\cB)$ as full hereditary subalgebras of $C^*_K(\bL(\cX)).$
  If $C^*_K(\cX)$ is the closure of $C_c(\cX)$ in $C^*_K(\bL(\cX)),$ then $C^*_K(\cX)$ is a $C^*_K(\cA)-C^*_K(\cB)-$equivalence bimodule with the canonical structure inherited from $C^*_K(\bL(\cX)),$ and $C^*_K(\bL(\cX))$ is $C^*-$ isomorphic to $\bL(C^*_K(\cX)).$    
  Moreover, $C^*_K(\cX)$ induces $\ker(q^{\cA}_{KH})$ to $\ker(q^{\cB}_{KH}).$ 
\end{theorem}
\begin{proof}
 Under the present hypotheses, the statements and the proofs of Theorem 4.5 and Corollary 4.10 of \cite{AbFrrEquivalence} hold true after we replace every instance of $C^*$ with $C^*_K.$
 One then obtains that $C^*_K(\cX)$ is a $C^*_K(\cA)-C^*_K(\cB)-$equivalence bimodule with the canonical structure inherited from $C^*_K(\bL(\cX)),$ and also that  $C^*_K(\bL(\cX))$ is C*-isomorphic to $\bL(C^*_K(\cX)).$
 
 Consider the following subspaces of $C^*_K(\bL(\cX))=\bL(C^*_K(\cX)):$
 \begin{align*}
M&:=C^*_K(\cA)\oplus C^*_K(\cX) & N&:=C^*_K(\cB)\oplus C^*_K(\cX).
 \end{align*}
 Then $M$ is a $C^*_K(\cA)-C^*_K(\bL(\cX))-$equivalence bimodule with the structure inherited from $C^*_K(\bL(\cX)).$
 Similarly, $N$ is a $C^*_K(\bL(\cX))-C^*_K(\cB)-$equivalence bimodule with the structure inherited from $C^*_K(\bL(\cX)).$  
 The induction of ideals from $C^*_K(\bL(\cX))$ to $C^*_K(\cA)$ via $M$ is easily described as $I\mapsto I\cap C^*_K(\cA).$
 An analogous claim holds for induction via $N.$
 
 By computing the internal tensor product $M\otimes_{C^*_K(\bL(\cX))}N$ using the C*-algebra structure of $C^*_K(\bL(\cX))$  we get that $M\otimes_{C^*_K(\bL(\cX))}N=C^*_K(\cX).$
 So, induction of ideals via $C^*_K(\cX)$ can be decomposed as induction via $M$ and $N.$ 
 Hence, to show that $C^*_K(\cX)$ induces $\ker(q^{\cA}_{KH})$ to $\ker(q^{\cB}_{KH}),$ we only need to show that 
 \begin{equation}\label{equ:induction via M}
   \ker(q^\cA_{KH})=C^*_K(\cA)\cap \ker(q^{\bL(\cX)}_{KH})
 \end{equation}
 because (by symmetry) this would imply that $\ker(q^\cB_{KH})=C^*_K(\cB)\cap \ker(q^{\bL(\cX)}_{KH}).$
 
 We have $q^{\bL(\cX)}_{KH}(C_c(\cA))=C_c(\cA);$ thus $q^{\bL(\cX)}_{KH}(C^*_K(\cA))=C^*_H(\cA).$
 Both $q^\cA_{KH}$ and the restriction $C^*_K(\cA)\to C^*_H(\cA),\ f\mapsto q^{\bL(\cX)}_{KH}(f),$ are the unique *-homomorphism extending the natural inclusion $C_c(\cA)\subset C^*_H(\cB),$ so they must agree.
 This gives $q^\cA_{KH}=q^{\bL(\cX)}_{KH}|_{C^*_K(\cA)},$ which clearly implies \eqref{equ:induction via M}. 
\end{proof}

\begin{corollary}
  Let $\cA$ and $\cB$ be weakly equivalent Fell bundles over $G$ and $H \sbgp K \sbgp G.$
  Suppose at least one of the following conditions is satisfied:
  \begin{enumerate}[(1)]
  	\item\label{item:both bundles saturated} Both $\cA$ and $\cB$ are saturated.
  	\item Every Fell bundle over $G$ satisfies Derighetty's strong condition with respect to $H$ and $K.$
  \end{enumerate}
  Then $\cA$ has the $HK-$WCP if and only if $\cB$ has the $HK-$WCP.
\end{corollary}
\begin{proof}
    If the first condition is fulfilled, then the bundles are strongly equivalent by \cite[Corollary 4.10]{abadie2019morita}.
    Hence, in any of the two cases we can find an $\cA-\cB-$equivalence bimodule $\cX$ satisfying the hypotheses of Theorem~\ref{thm:induction via CKX}.
    Thus $\cA$ has the $HK-$WCP (i.e. $\ker(q^\cA_{KH})=\{0\}$) if and only if $\cB$ has the $HK-$WCP.
\end{proof}

\subsection{A weak (containment) imprimitivity theorem}\label{ssec:a weak imprimitivity theorem}
Let $\langle \regrep{H\cB},\psi^{H\cB}\rangle $ be the universal system of imprimitivity for the Fell bundle $\cB=\{B_t\}_{t\in G}$ over $G/H$ ($H \sbgp G$).
Given any *-representation $R$ of $\cB_H,$  $T:=\Ind_H^\cB(R)\equiv \regrep{H\cB}\otimes_R 1$ is part of the system of imprimitivity $\langle \regrep{H\cB}\otimes_R1,\psi^{H\cB}\otimes_R 1\rangle;$ which is called the system induced by $R.$

Fell's Imprimitivity Theorem \cite[XI 14.17]{FlDr88} gives a necessary and sufficient condition for a system of imprimitivity for $\cB$ over $G/H$ to be, up to a unitary equivalence, an induced system.
This condition holds for every saturated bundle \cite[XI 14.18]{FlDr88} and we ask ourselves what can we say for non saturated bundles.

Consider a system of imprimitivity $\langle S\otimes \Ind_H^G(U),\Psi\rangle$ as in Example~\ref{example:system of imprimitivity}.
For it to be an induced system, $S\otimes \Ind_H^G(U)$ should be induced from a *-representation of $\cB_H.$
But Theorem~\ref{thm:Fells absorption principle I} seems to imply that this is not always the case because $S\otimes \Ind_H^G(U)$ is weakly equivalent to a family of *-representation induced from the reductions $\cB_{tHt^{-1}}$ ($t\in G$).
All these reductions agree if $H$ is normal in $G$ and, moreover, $S\otimes \Ind_H^G(U)$ is weakly equivalent to $\Ind_H^\cB(S|_{\cB_H}\otimes U).$
One may pose the following question:

\begin{question}\label{q:systems of impr and weak equivalence}
Suppose $\cB=\{B_t\}_{t\in G}$ is a Fell bundle, $H \sbgp G$ and $T\colon \cB\to \bB(X)$ is a non degenerate *-representation which is part of a system of imprimitivity $\langle T,\varphi\rangle$ for $\cB$ over $G/H.$
Is $T$ is weakly equivalent to $\Ind_H^\cB(R)$ for some *-representation $R$ of $\cB_H$?
\end{question}

As discussed above, if $\cB$ is saturated, then the $T$ of the question is known to be unitary equivalent to an induced representation.
For general bundles (saturated or not) the best answer we can give is as follows.

\begin{theorem}\label{thm:main about weak containment of systems of imprimitivity}
 Consider the situation of Question~\ref{q:systems of impr and weak equivalence} and assume every Fell bundle over $G$ satisfies Derighetti's strong condition with respect to $H \sbgp G$ (e.g. $H$ is normal, open or has open normaliser in $G$).
 Then $T$ is weakly contained in $\Ind_H^\cB(R)$ for some *-representation $R$ of $\cB_H.$
\end{theorem}
\begin{proof}
	It suffices to show that $\ker (q^\cB_H)\subset \ker(\intform{T}).$
	Indeed, if this is the case, we take for $R$ any non degenerate *-representation of $\cB_H$ with faithful integrated form.
	To prove this is an appropriate choice notice that $\Ind_H^\cB(R)=\regrep{H\cB}\otimes_R 1$ or, in integrated form,  $q^\cB_H\otimes_{\intform{R}}1=\intform{\Ind}_H^\cB(R).$
	Thus $\ker (q^\cB_H)=\ker(\intform{\Ind}_H^\cB(R))$ and, consequently, $T\preceq \Ind_H^\cB(R)\Leftrightarrow \ker (q^\cB_H)\subset \ker(\intform{T}).$
	
We now adopt the notation and constructions of Section 3 of \cite{abadie2019morita}, so the C*-algebra of kernels $\mathbb{k}(\cB)$ is a completion of the *-algebra of compactly supported kernels of $\cB,$ $\mathbb{k}_c(\cB)$.
The canonical action of $G$ on $\bk(\cB)$ will be denoted $\beta$ and $\cB_\beta$ is the semidirect product bundle of $\beta.$
By \cite[Theorem 3.4]{abadie2019morita}, the canonical $L^2-$bundle $\cL^2\cB=\{L^2_e(\cB)\delta_t\}_{t\in G}$ is a $\cB_\beta-\cB-$(weak) equivalence bundle.
A typical element of the fibre $L^2_e(\cB)\delta_t$ is denoted $f\delta_t$ with $f\in L^2_e(\cB).$

By \cite[Proposition 5.5]{Ab03}, there exists a (unique) covariant pair $(\psi,U)$ for $\beta$ such that: $\psi\colon \bk(\cB)\to \bB(L^2(G,X)),$ $U\colon G\to \bB(L^2(G,X))$ and for all $k\in \bk_c(\cB), $ $f\in C_c(G,X)$ and $s,t\in G:$
\begin{align*}
  \psi(k)f& \in C_c(G,X) & (\psi(k)f)(s) &= \int_G T_{k(s,r)}f(r)\, dr\\
  U_tf & \in C_c(G,X) & (U_tf)(s)&=f(st)\Delta(t)^{1/2}.
\end{align*}

We claim that there is a unique *-representation $\hat{\varphi}\colon C_0(G/H)\to \bB(L^2(G,X))$ such that for all $a\in C_0(G/H)),$ $f\in C_c(G,X)$ and $t\in G:$
\begin{align*}
 \hat{\varphi}(a)f& \in C_c(G,X) & (\varphi(a)f)(t)=\varphi(\sigma_t(a))f(t),
\end{align*}
where $\sigma$ is the natural action of $G$ on $C_0(G/H).$
To prove this we start by noticing that given $a\in C_0(G/H)$ and $f\in C_c(G,X),$ the function $a\cdot_\varphi f\colon G\to X,\ t\mapsto \varphi(\sigma_t(a))f(t),$ is continuous and has compact support.
Also,
\begin{equation*}
\| a\cdot_\varphi f \|_2^2 = \int_G \|\varphi(\sigma_t(a))f(t)\|^2\, dt\leq \int_G \|\varphi(\sigma_t(a))\|^2\|f(t)\|^2\, dt\leq \|a\|\|f\|_2^2.
\end{equation*}
Since the expression $a\cdot_\varphi f$ is linear in $f,$ there exists a unique operator $\hat{\varphi}(a)\in \bB(L^2(G,H))$ such that $\hat{\varphi}(a)f=a\cdot_\varphi f$ for all $f\in C_c(G,X).$
We leave to the reader to prove  $\hat{\varphi}\colon C_0(G/H)\to \bB(L^2(G,X))$ is a *-representation.

Using the fact that $\varphi$ is non degenerate, one can show that for every approximate unit $\{a_i\}_{i\in I}$ of $C_0(G/H)$ and every $g\in C_c(G,X),$ the net $\{\hat{\varphi}(a_i)g\}_{i\in G}$ converges uniformly to $g.$
Since the supports of all the elements of the net are contained in the support of $g,$ the net converges to $g$ in $L^2(G,X).$
Thus $\hat{\varphi}$ is non degenerate and, by Stone's Theorem \cite[VI 10.10]{FlDr88}, there exists a unique $L^2(G,X)$-projection-valued Borel measure $P$ on $G/H$ such that $\hat{\varphi}(a)=\int_{G/H} a(x)\, dP(x).$

We claim that the ranges of $\hat{\varphi}$ and $\psi$ commute and that $(\hat{\varphi},U)$ is a covariant pair for $\sigma.$
Indeed, given $a\in C_0(G/H)$ and $k\in \bk_c(\cB),$ for all $f\in C_c(G,X)$ and $t\in G$ we have
\begin{multline*}
(\psi(k)\hat{\varphi}(a)f)(t)
 = \int_G T_{k(t,s)}(\hat{\varphi}(a)f)(s)\, ds
 =  \int_G T_{k(t,s)}\varphi(\sigma_s(a))f(s)\, ds\\
 =  \int_G \varphi(\sigma_{ts^{-1}}\sigma_s(a))  T_{k(t,s)}f(s)\, ds
 = \varphi(\sigma_{t}(a)) \int_G   T_{k(t,s)}f(s)\, ds
 = (\hat{\varphi}(a)\psi(k)f)(t);
\end{multline*}
because $k(t,s)\in B_{ts^{-1}}.$
Then $\psi(k)$ and $\hat{\varphi}(a)$ commute for all $k\in \bk_c(\cB)$ and $a\in C_0(G/H).$
Since $\bk_c(\cB)$ is dense in $\bk(\cB),$ the same conclusion holds for all $k\in \bk(\cB).$

Given $a\in C_0(G/H) $  and $t\in G,$ for all $f\in C_c(G,X)$ and $s\in G$ we have
\begin{align*}
(U_t\hat{\varphi}(a)U_t^*f)(s)
  = \Delta(t)^{1/2}\varphi(\sigma_{st}(a))(U_t^*f)(st)
  =\varphi(\sigma_{st}(a))f(s)
  =(\hat{\varphi}(\sigma_t(a))f)(s);
\end{align*}
which implies $(\hat{\varphi},U)$ is a covariant pair for $\sigma.$

The fibre over $t\in G$ of the semidirect product bundle $\cB_\beta$ is formed by elements $k\delta_t$ with $k\in \bk(\cB).$
With this notation, the function $S\colon \cB_\beta\to \bB(L^2(G,X)),$ $k\delta_t\mapsto \psi(k)U_t,$ is a non degenerate *-representation of $\cB_\beta.$
To prove the pair $\langle S,P\rangle $ is a system of imprimitivity for $\cB_\beta$ over $G/H$ we only need to show that $\psi(k)U_t \hat{\varphi}(a)=\hat{\varphi}(\sigma_t(a))\psi(k)U_t$ for all $k\in \bk(\cB),$ $t\in G$ and $a\in C_0(G/H)$ (see \cite[VIII 18.8]{FlDr88}).
Since $\psi$ is non degenerate, $\langle S,P\rangle $ is a system of imprimitivity if and only if the ranges of $\psi$ and $\hat{\varphi}$ commute and $(\hat{\varphi},U)$ is a covariant pair of $\sigma$ (which we already know).

The semidirect product bundle of any C*-dynamical system is saturated then, by Fell's Imprimitivity Theorem for saturated bundles \cite[XI 14.18]{FlDr88}, $S$ is induced from some *-representation of the reduction of $\cB_\beta$ to $H.$
This yields $\ker( q^{\cB_\beta}_H )\subset \ker(\intform{S}).$
When specializing Theorem~\ref{thm:induction via CKX} to the case $G=K$ and the equivalence bundle $\cX=\cL^2\cB,$ we get that $C^*(\cL^2\cB)$ induces $\ker(q^\cB_H)$ to $\ker( q^{\cB_\beta}_H ).$

Suppose we can show that the *-representation of $C^*(\cB_\beta)$ induced by $\intform{T}$ via $C^*(\cL^2\cB),$ $\Ind_{C^*(\cL^2\cB)}(\intform{T}),$ is a subrepresentation of $\intform{S}.$
Then
\begin{equation*}
	\Ind_{C^*(\cL^2\cB)}(\ker(q^\cB_H)) = \ker( q^{\cB_\beta}_H )\subset \ker(\intform{S})\subset \ker(\Ind_{C^*(\cL^2\cB)}(\intform{T}))
	= \Ind_{C^*(\cL^2\cB)}(\ker(\intform{T})),
\end{equation*}
and it would follow that 
\begin{equation}\label{equ:inclusion of induced kernels}
	\ker(q^\cB_H)\subset \ker(\intform{T})
\end{equation}
 because $C^*(\cL^2\cB)$ is a $C^*(\cB_\beta)-C^*(\cB)$-equivalence bimodule \cite{AbFrrEquivalence,abadie2019morita}; which would compete the proof.

The induced Hilbert space $C^*(\cL^2\cB)\otimes_{\intform{T}}X$ is the closed linear span of tensor products $f\otimes_{\intform{T}} \xi$ with $f\in C_c(\cL^2\cB)$ and $\xi\in X.$
To simplify our arguments we will use a subspace of  $C_c(\cL^2\cB)$ that is still dense in $C^*(\cL^2\cB).$

Let $\Gamma$ be the set of all continuous functions $f\colon G\times G\to \cB$ of compact support such that $f(s,t)\in B_t$ for all $(s,t)\in G\times G.$
We may think of $\Gamma$ as $C_c(\cD)$ for the retraction $\cD$ of $\cB$ by the map $G\times G\to G,(s,t)\mapsto t$ (as defined in \cite[II 13.3]{FlDr88}).
Recall from \cite{FlDr88} that $\cD$ is a Banach bundle, so for every $(s,t)\in G$ and $b\in B_t$ there exists $f\in \Gamma$ such that $f(s,t)=b.$

For each $f\in \Gamma$ we define $f|\colon G\to C_c(\cB),\ s\mapsto f|_s,$ where $f|_s\in C_c(\cB)$ is given by $f|_s(t)=f(s,t).$
Then $f|$ is continuous in the inductive limit topology and with respect to the norm topology of $L^2_e(\cB).$
Thus $f|\delta\colon G\to \cL^2\cB,\ t\mapsto f|_t\delta_t,$ is a continuous section of compact support.

Given $t\in G$ and $g\in C_c(\cB)\subset L^2_e(\cB),$ we can use \cite[II 14.8]{FlDr88} to prove the existence of a function $f\in \Gamma$ such that $f(t,s)=g(s)$ for all $s\in G.$
Thus $(f|\delta)(t)=g\delta_t.$
Since $\Gamma|\delta:=\{f|\delta\colon f\in \Gamma\}$ is a linear subspace of $C_c(\cL^2\cB)$ closed under pointwise multiplication by continuous complex functions of $G,$ \cite[II 14.6]{FlDr88} implies $\Gamma|\delta$ is dense in $C_c(\cL^2\cB)$ in the inductive limit topology.
This last topology is larger than the $\|\ \|_1-$topology and, consequently, also larger than the norm topology given by the inclusion  $C_c(\cL^2\cB)\subset C^*(\cL^2\cB)$ (recall $C^*(\cL^2\cB)$ is the closure of $C_c(\cL^2\cB)$ in $C^*(\bL(\cL^2\cB))$).
Then $\Gamma|\delta$ is dense in $C^*(\cL^2\cB)$ and $C^*(\cL^2\cB)\otimes_{\intform{T}}X$ is the closed linear span of
\begin{equation*}
\{(f|\delta)\otimes_{\intform{T}} \xi\colon f\in \Gamma,\xi\in X\}.
\end{equation*}
 
Given $f\in \Gamma$ and $\xi\in X,$ the function $G\times G\to X,\ (s,t)\mapsto T_{f(s,ts)}\xi,$ is continuous and has compact support.
Then $f\cdot_T\xi\colon G\to X,$ defined by
\begin{equation*}
	f\cdot_T \xi(t) = \int_G T_{f(s,ts)}\xi\, ds,
\end{equation*}
is a continuous function of compact support (and so an element of $L^2(G,X)$).

Fix $f,g\in \Gamma$ and $\xi,\eta\in X.$
To compute the inner product $\langle f|\delta,g|\delta\rangle_{C^*(\cB)} $ it is convenient to think $f|\delta$ and $g|\delta$ as elements of $C_c(\bL(\cL^2\cB))\subset C^*(\bL(\cL^2\cB)).$
Then $\langle f|\delta,g|\delta\rangle_{C^*(\cB)}\in C_c(\cB)$ is the convolution $(f|\delta)^**(g|\delta)$ computed in $C_c(\bL(\cL^2\cB)):$
\begin{equation*}
\langle f|\delta,g|\delta\rangle_{C^*(\cB)}(t)
=\int_G \langle f|_s\delta_s,g|_{st}\delta_{st}\rangle_\cB \, ds
=\int_G\int_G  f(s,rs)^*g(st,rst) \, drds
\end{equation*}

By the definition of the inner products of $L^2(G,X)$ and $C^*(\cL^2\cB)\otimes_{\intform{T}}X$ we have
\begin{multline*}
   \langle (f|\delta)\otimes_{\intform{T}}\xi,(g|\delta)\otimes_{\intform{T}}\eta\rangle
   = \langle \xi,\intform{T}_{\langle f|\delta,g|\delta\rangle_{C^*(\cB)}}\rangle
     = \int_G\int_G\int_G \langle T_{ f(s,rs)}\xi,T_{g(st,rst)} \eta\rangle\, drdsdt\\
     = \int_G\int_G\int_G \langle T_{ f(s,rs)}\xi,T_{g(t,rt)} \eta\rangle\, dtdsdr
     = \langle f\cdot_T\xi,g\cdot_T\eta\rangle.
\end{multline*}
Then there exists a unique linear isometry
\begin{equation*}
	V\colon C^*(\cL^2\cB)\otimes_{\intform{T}}X\to \bB(L^2(G,X))
\end{equation*}
mapping $(f|\delta)\otimes_{\intform{T}}\xi$ to $f\otimes_T\xi,$ for all $f\in \Gamma$ and $\xi\in X.$

The canonical representation $\regrep{\bL(\cL^2\cB)}\colon \bL(\cL^2\cB)\to \bB(C^*(\bL(\cL^2\cB)))$ can be restricted to get a non degenerate *-representation 
\begin{equation*}
  \Omega\colon \cB_\beta\to \bB(C^*(\cL^2\cB)) \qquad \Omega_u g=\regrep{\bL(\cL^2\cB)}_ug.
\end{equation*}
Naturally, the integrated form of $\Omega$ can be computed using the integrated form of $\regrep{\bL(\cL^2\cB)}$ (i.e. $q^{\bL(\cL^2\cB)}$).
Then $\intform{\Omega}_fg=f*g$ for all $f\in C_c(\cB_\beta)$ and $g\in C_c(\cL^2\cB).$

The integrated form of the tensor product representation 
\begin{equation*}
\Omega\otimes_{\intform{T}}1\colon \cB_\beta\to \bB(C^*(\cL^2\cB)\otimes_{\intform{T}}X)
\end{equation*}
is precisely $\Ind_{C^*(\cL^2\cB)}(\intform{T}).$
Hence, to show 
\begin{equation}\label{equ:integrated form of V ind equals S V}
	V\Ind_{C^*(\cL^2\cB)}(\intform{T})_f = \intform{S}_fV
\end{equation}
for all $f\in C^*(\cB_\beta)\equiv \bk(\cB)\rtimes_\beta G,$ it suffices to prove  that 
\begin{equation}\label{equ:V ind equals S V}
V(\Omega\otimes_{\intform{T}}1)_{k\delta_t}(  (g|\delta)\otimes_{\intform{T}}\xi)=  S_{k\delta_t}V(  (g|\delta)\otimes_{\intform{T}}\xi)
\end{equation}
for all $k\in \bk_c(\cB),$ $t\in G,$ $g\in \Gamma$ and $\xi\in X.$

Fix $k,t,g$ and $\xi.$
Then $(\Omega\otimes_{\intform{T}}1)_{k\delta_t}(  (g|\delta)\otimes_{\intform{T}}\xi)=(\Omega_{k\delta_t}  (g|\delta))\otimes_{\intform{T}}\xi.$
Since $g|\delta\in C_c(\cL^2\cB),$ we can compute $\Omega_{k\delta_t}  (g|\delta)\in C_c(\cL^2\cB)$ as
\begin{equation*}
\Omega_{k\delta_t}  (g|\delta)(s)
= (k\delta_t)(g|_{t^{-1}s}\delta_{t^{-1}s})
=\Delta(t)^{1/2} \beta_{s^{-1}}(k)g|_{t^{-1}s}\delta_s.
\end{equation*} 

By the definition of the right $\cB_\beta-$bundle structure of $\cL^2\cB,$ the expression $\beta_{s^{-1}}(k)g|_{t^{-1}s}$ refers to the natural representation of $\bk(\cB)$ on $L^2_e(\cB).$
Then 
\begin{equation*}
	[\beta_{s^{-1}}(k)g|_{t^{-1}s}](r)
	=\int_G \Delta(s)^{-1}k(rs^{-1},us^{-1})g(t^{-1}s,u)\, du
	=\int_G k(rs^{-1},u)g(t^{-1}s,us)\, du
\end{equation*}

One way of showing that the function 
\begin{align*}
   k\diamond_t g &\colon G\times G\to \cB & k\diamond_t g(s,r):=\Delta(t)^{1/2} \int_G k(rs^{-1},u)g(t^{-1}s,us)\, du
\end{align*}
is continuous and has compact support is as follows.
Let $A$ be a compact subset of $G$ such that $A\times A $ contains the supports of $k$ and $g.$
Then $\nu\colon G\times G\times G\to \cB,$ $\nu(r,s,u)=k(rs^{-1},u)g(t^{-1}s,us),$ is continuous and its support is contained in the compact $B:=(AtA)\times (tA)\times A.$
Besides, $\nu(r,s,t)\in B_r.$
Let $\cD$ be the retraction of $\cB$ by $G\times G\to G,(r,s)\mapsto r,$ and $E$ the (Banach) space of continuous cross sections of $\cD$ supported in $(AtA)\times (tA);$ the norm of $E$ being $\|\ \|_\infty.$
Then $\nu|\colon G\to E,$ $\nu|(u)(r,s)=\nu(r,s,u),$ is continuous and has compact support, so its integral $\int_G \nu|(u)\, du$ is an element of $E$ such that, for all $(r,s)\in G\times G:$
\begin{equation*}
\Delta(t)^{1/2}\left(\int_G \nu|(u)\, du\right) (r,s)=\Delta(t)^{1/2}\int_G \nu(r,s,u)\, du=k\diamond_tg(r,s).
\end{equation*}
By the definition of the retraction of a bundle, $k\diamond_tg$ is continuous (and has compact support).

Our definition of $k\diamond_tg$ gives the identity $\Omega_{k\delta_t}  (g|\delta)= (k\diamond_tg)|\delta,$ so
\begin{equation*}
V(\Omega\otimes_{\intform{T}}1)_{k\delta_t}(  (g|\delta)\otimes_{\intform{T}}\xi)
= V(\Omega_{k\delta_t}  (g|\delta))\otimes_{\intform{T}}\xi
=(k\diamond_tg)\cdot_T\xi.
\end{equation*}
By construction,
\begin{multline*}
(k\diamond_tg)\cdot_T\xi(r)
=\int_G T_{(k\diamond_tg)(s,rs)}\xi\, ds
=\Delta(t)^{1/2}\int_G\int_G  T_{k(r,u)g(t^{-1}s,us)}\xi\, duds=\\
=\Delta(t)^{1/2}\int_GT_{k(r,u)} \int_G  T_{g(t^{-1}s,us)}\xi\, dsdu
=\int_GT_{k(r,u)} \Delta(t)^{1/2}\int_G  T_{g(s,uts)}\xi\, dsdu =\\
=\int_GT_{k(r,u)} [U_t(g\cdot_T\xi)](u) du
= \psi(k)U_t(g\cdot_T\xi)
=S_{k\delta_t}V(g|\delta)\otimes_{\intform{T}}\xi.
\end{multline*}

This completes the proof of \eqref{equ:V ind equals S V}, which is equivalent to \eqref{equ:integrated form of V ind equals S V} and thus we obtain \eqref{equ:inclusion of induced kernels}, completing the proof.
\end{proof}

The result below is a (weak) combination of  Fell's Imprimitivity Theorem for saturated bundles \cite[XI 14.18]{FlDr88}, Theorem 2.14 and Corollary 2.15 from \cite{ExNg} and, finally, the last theorem we proved above.

\begin{corollary}\label{cor:reps that are part of s system of imprimitivity}
 Let $\cB$ be a Fell bundle over $G,$ $H \sbgp G$ and $T\colon \cB\to \bB(X_A)$ a *-representation which is part of a system of imprimitivity $\langle T,\varphi\rangle$ for $\cB$ over $G/H.$
 If $\cB$ is saturated (and/)or every Fell bundle over $G$ satisfies Derighetti's strong condition with respect to $H,$ then there exists a unique *-representation $\pi_T$ making 
 \begin{equation*}
 \xymatrix{ C^*(\cB)\ar[rr]^{\intform{T}}\ar[rd]_{q^\cB_H} & & \bB(X_A)\\
 	& C^*_H(\cB)\ar[ur]_{\pi_T} & }
 \end{equation*}
 a commutative diagram.
\end{corollary}
\begin{proof}
Take a faithful and non degenerate *-representation $\rho\colon A\to \bB(Y).$
Then $\langle T\otimes_\rho 1,\varphi\otimes_\rho 1\rangle $ is a system of imprimivitivity for $\cB$ over $G/H.$
If $\cB$ is saturated, then Fell's Imprimitivity Theorem for saturated bundles implies $T\otimes_\rho 1$ is unitary equivalent to $\Ind_H^\cB(R)=\regrep{H\cB}\otimes_R 1$ for some *-representation $R\colon \cB_H\to \bB(Y).$
In case $\cB$ is not saturated, we are under the hypotheses of Theorem~\ref{thm:main about weak containment of systems of imprimitivity}.
In any of the two situations we can assume $T\otimes_\rho 1$ is weakly contained in $\Ind_H^\cB(R)=\regrep{H\cB}\otimes_R 1$ for some *-representation $R\colon \cB_H\to \bB(Y).$
For all $f\in C^*(\cB)$ we have $\|\intform{T}_f\|\leq \|\intform{\Ind}_H^\cB(R)_f\|\leq \|\intregrep{H\cB}_f\|=\|q^\cB_H(f)\|.$
Then we can define $\pi_T$ by setting $\pi_T(q^\cB_H(f)):=\intform{T}_f.$
\end{proof}

\section{C*-completions of Banach *-algebraic bundles}\label{sec:completions}

In this section we relate the induction of representations with the construction of C*-completions of Banach *-algebraic bundles.
To do this we introduce two definitions.

\begin{definition}
 A collection of maps $\rho=\{\rho_t\colon A_t\to B_t\}_{t\in G}$ is a morphism of Banach *-algebraic bundles from $\cA=\{A_t\}_{t\in G}$ to $\cB=\{B_t\}_{t\in G}$ if
 \begin{enumerate}
  \item For all $t\in G,$ $\rho_t$ is linear.
  \item The disjoint union $\rho\colon \cA\to \cB$ of the maps $\{\rho_t\}_{t\in G}$ is continuous.
  \item For all $a,b\in \cA,$ $\rho(ab)=\rho(a)\rho(b)$ and $\rho(a^*)=\rho(a)^*.$
  \item There exists $M\in [0,+\infty)$ such that $\|\rho(a)\|\leq M\|a\|$ for all $a\in \cA.$
 \end{enumerate}
\end{definition}

If the bundle $\cB$ above is a Fell bundle, then $\rho_e\colon A_e\to B_e$ is contractive because it is a map from a Banach *-algebra to a C*-algebra.
Hence, for all $a\in \cA$ we have $\|\rho(a)\|=\|\rho(a)^*\rho(a)\|^{1/2}=\|\rho(a^*a)\|^{1/2}\leq \|a^*a\|^{1/2}=\|a\|.$
Thus condition (4) holds for $M=1.$

\begin{definition}
 A C*-completion of a Banach *-algebraic bundle $\cA=\{A_t\}_{t\in G}$ is a pair $(\cB,\rho)$ such that 
  $\cB=\{B_t\}_{t\in G}$ is a Fell bundle; $\rho=\{\rho_t\colon A_t\to B_t\}_{t\in G}$ a morphism of Banach *-algebraic bundles and $\rho_t(A_t)$ is dense in $B_t$ for all $t\in G.$
 Given another C*-completion of $\cA,$ $(\cC,\gamma),$ a morphism from $(\cB,\rho)$ to $(\cC,\gamma)$ is a morphism of Banach *-algebraic bundles $\nu=\{\nu_t\colon B_t\to C_t\}_{t\in G}$ such that $\nu\circ \rho = \gamma$ (i.e. $\nu_t\circ \rho_t=\gamma_t$ for all $t\in G$). 
\end{definition}

\begin{example}
 Let $\cB$ be a Banach *-algebraic bundle and $\cC$ the bundle C*-completion of $\cB,$ as constructed in \cite[VIII 16.7]{FlDr88}.
 We also use the map $\rho\colon \cB\to \cC$ constructed there to conclude $(\cC,\rho)$ is a C*-completion.
\end{example}

We now use the ideas of \cite[VIII 16.7]{FlDr88} to get the following.

\begin{proposition}\label{prop:completions and representations}
 Let $\cB=\{B_t\}_{t\in G}$ be a Banach *-algebraic bundle and $T\colon \cB\to \bB(Y)$ a *-representation.
 Then there exists a C*-completion $(\cB^T,\rho^T)$ of $\cB$ such that $\|\rho^T(b)\|=\|T_b\|$ for all $b\in \cB.$
 Moreover,
 \begin{enumerate}
  \item Any C*-completion of $\cB$ is isomorphic to one of the form $(\cB^T,\rho^T).$
  \item Given another *-representation $S\colon \cB\to \bB(Z),$ there exists a morphism 
  \begin{equation*}
   \rho^{TS}\colon (\cB^T,\rho^T)\to (\cB^S,\rho^S)
  \end{equation*}
  if and only if $\|S_b\|\leq \|T_b\|$ for all $b\in B_e.$
  If this last condition holds, then there is a unique morphism from $ (\cB^T,\rho^T)$ to $ (\cB^S,\rho^S).$
  \item  $ (\cB^T,\rho^T)$ is isomorphic to $ (\cB^S,\rho^S)$ if and only if $\|T_b\|=\|S_b\|$ for all $b\in B_e.$
 \end{enumerate}
 \begin{proof}
  The construction of $(\cB^T,\rho^T)$ can be performed following \cite[VIII 16.7]{FlDr88} but using the function $\cB\to [0,+\infty)$ given by $b\mapsto \|T_b\|$ instead of the function $\|\ \|_c$ used in \cite[VIII 16.7]{FlDr88}.
  It is implicit in this construction that the set of sections $\Gamma_T:=\{\rho^T\circ f\colon f\in C_c(\cB)\} \subset C_c(\cB^T)$ determines the topology of $\cB^T$ in the sense of \cite[II 13.18]{FlDr88}.
  Moreover, that set of sections is dense in the inductive limit topology in $C_c(\cB^T)$ by \cite[II 14.6]{FlDr88}.
  
  Take a C*-completion $(\cC,\gamma)$ of $\cB.$
  Then $\cC$ is it's own bundle C*-completion by \cite[VIII 16.10]{FlDr88}; this means there exists a *-representation $R\colon \cC\to \bB(Z)$ such that $\|R_c\|=\|c\| $ for all $c\in C_e.$
  Thus, for all $c\in \cC,$ $\|c\|= \|c^*c\|^{1/2}=\|R_c^*R_c\|^{1/2}=\|R_c\|.$
  
  Set $R':=R\circ \gamma\colon \cB\to \bB(Y)$ and form the C*-completion $(\cB^{R'},\rho^{R'}).$
  For every $t\in G$ the fibre $B^{R'}_t$ over $t$ of $\cB^{R'}$ is constructed as the completion of $B_t$ with respect to the seminorm $B_t\to [ 0,+\infty),$ $b\mapsto \|R'_b\|=\|R_{\rho(b)}\|=\|\rho(b)\|.$ 
  Since $\rho(B_t)$ is dense in $C_t,$ there exists a unique surjective linear isometry $\nu_t\colon B^{R'}_t\to C_t$ such that $\nu_t\circ \rho^{R'}_t = \gamma.$
  We claim that $\nu:=\{\nu_t\colon B^{R'}_t\to C_t\}_{t\in G}$ is an isomorphism of C*-completions. 
  
  Notice that for all $f=\rho^{R'}\circ g\in \Gamma_{R'}$ we have $\nu\circ f = \gamma\circ g\in C_c(\cC).$
  Since each $\nu_t$ is a bijective linear isometry, the two Propositions of \cite[II 13.16]{FlDr88} imply both $\nu\colon \cB^{R'}\to \cC$ and it's inverse $\nu^{-1}\colon \cC\to \cB^{R'}$ are continuous.
  We have, for all $a,b\in \cB,$ $\nu( \rho^{R'}(a)\rho^{R'}(b) ) = \gamma(ab)=\gamma(a)\gamma(b) = \nu( \rho^{R'}(a))\nu(\rho^{R'}(b) ).$
  Since $\rho(R')(\cB)$ is dense in $\cB^{R'}$ and $\nu$ is continuous, it follows that $\nu$ preserves multiplication. 
  A similar argument implies $\nu(a)^*=\nu(a)$ for all $a\in \cB^{R'}.$
  Then $\nu$ is an isomorphism, which shows (1).
  
  Assume, as in (2), $S\colon \cB\to \bB(Z)$ is a *-representation and that there exists a morphism $\rho^{TS}\colon  \cB^T\to \cB^{S}.$
  Such a morphism must be contractive and, consequently, for all $b\in \cB $ we have $\|S_b\|= \|\rho^S(b)\|=\|\rho^{TS}(\rho^T(b))\|\leq \|\rho^T(b)\|=\|T_b\|.$
  
  To prove the converse in (2) we first notice that for all $b\in \cB,$ $\|S_b\|= \|S_{b^*b}\|^{1/2}\leq \|T_{b^*b}\|^{1/2} = \|T_b\|.$
  Then the morphism $\rho^{TS}\colon \cB^T\to \cB^S$ can be constructed using the same arguments we used to construct the isomorphism $\nu.$
  When doing so notice that to show $\nu$ is continuous we don't really need it to be isometric, just contractive suffices.
  
  Finally, claim (3) follows easily from (2).
  \end{proof} 
\end{proposition}

\begin{corollary}\label{cor:uniqueness of completion}
 Let $\cB=\{B_t\}_{t\in G}$ be a Banach *-algebraic bundle and consider two of its  C*-completions, $(\cA,\iota)$ and $(\cC,\kappa).$
 Then there exists a morphism $\rho\colon (\cA,\iota)\to (\cC,\kappa)$ if and only if $\|\kappa(b)\|\leq \|\iota(b)\|$ for all $b\in B_e.$
 In fact the morphism is unique (if it exists) and it is an isomorphism if and only if $\|\iota(b)\|=\|\kappa(b)\|$ for all $b\in B_e.$
\begin{proof}
 By Proposition~\ref{prop:completions and representations} we may assume $(\cA,\iota)=(\cB^S,\rho^S)$ and $(\cC,\kappa)=(\cB^T,\rho^T)$ for some representations $S$ and $T$ of $\cB.$
 After this identification the statement becomes claim (3) of Proposition~\ref{prop:completions and representations}.
\end{proof}
\end{corollary}

Notice that the Corollary above implies that any C*-completion $(\cA,\iota)$ of $\cB$ is completely determined by the completion $(A_e,\iota_e)$ of $B_e.$
We want to know which are the C*-completions of $B_e$ arising in this way.
This problem can be thought of a generalisation of the problem considered in \cite[XI 11.2]{FlDr88} of determining if the unit fibre of the bundle C*-completion of $\cB$ equals the enveloping C*-algebra of $B_e.$

\begin{definition}\label{defi:rho completion of bundle}
 Given a Banach *-algebraic bundle $\cB=\{B_t\}_{t\in G}$ and a C*-completion $\rho\colon B_e\to A,$ a $\rho-$completion of $\cB$ is a C*-completion $(\cA,\iota)$ such that $\iota_e=\rho.$ 
\end{definition}

We now combine the induction process with the existence of particular C*-completions of Banach *-algebraic bundles.

\begin{theorem}\label{thm:equivalent conditions for the existence of completions}
 Let $\cB= \{B_t\}_{t\in G}$ be a Banach *-algebraic bundle and let $\pi\colon B_e\to A$ be a C*-completion, which we regard as a *-representation $\pi\colon B_e\to A\subset \bB(A).$
 Then the following are equivalent:
 \begin{enumerate}
  \item There exists a $\pi-$completion of $\cB.$
  \item $\pi(b^*b)\geq 0$ and and $\pi(c^*b^*bc)\leq \|\pi(b^*b)\|\pi(c^*c)$ for all $b,c\in \cB.$
  \item There exists a $\cB-$positive *-representation $T\colon B_e\to \bB(Y)$ such that for all $b\in B_e$ we have $\|\Ind_{\{e\}}^\cB(T)_b\|\leq \|T_b\|=\|\pi(b)\|.$
  \item There exists a $\cB-$positive *-representation $T\colon B_e\to \bB(Y)$ such that for all $b\in B_e$ we have $\|\Ind_{\{e\}}^\cB(T)_b\|= \|T_b\|=\|\pi(b)\|.$
  \item There exists a non degenerate *-representation $T\colon \cB\to \bB(Y)$ such that $\|\pi(b)\|=\|T_b\|$ for all $b\in B_e.$
 \end{enumerate}
Moreover, all the conditions above imply 
 \begin{enumerate}
  \item[(6)] $0\leq \pi(c^*b^*bc)\leq \|\pi(b^*b)\|\pi(c^*c)$ for all $b,c\in \cB$
 \end{enumerate}
 and the converse holds provided that $\cB$ has a strong approximate unit.
\begin{proof}
 Assume (1) holds and take a C*-completion $(\cA,\iota)$ such that $\pi=\iota_e.$
 Then for all $b\in \cB$ we have $\pi(b^*b)=\iota(b)^*\iota(b)\geq 0$ because $\cA$ is a Fell bundle.
 If we replace the topology of $G$ by the discrete one, then $\cA$ becomes a Fell bundle over a discrete group and using \cite[Lemma 17.2]{Exlibro} we deduce that $a^*b^*ba\leq \|b^*b\|a^*a$ (in $A_e$) for all $a,b\in \cA$ (no matter which topology we consider on $G$).
 Thus for all $b,c\in \cB$ we have $\pi(c^*b^*bc)=\iota(c)^*\iota(b)^*\iota(b)\iota(c)\leq \|\iota(b^*b)\|\iota(c)^*\iota(c)=\|\pi(b^*b)\|\pi(c^*c).$
 Hence (2) holds.
 
 Assume (2) and take a non non degenerate and faithful representation $\rho\colon A\to \bB(Y).$
 If we set $T:=\rho\circ \pi\colon B_e\to \bB(Y),$ then $\|T_b\|=\|\pi(b)\|$ for all $b\in B_e.$
 Theorem~\ref{thm:characterization of positive representations} implies $T$ is $\cB-$positive.
 To prove the inequality in (3) we may use the abstract induction process.
 
 The Hilbert space induced by $T,$ $Z,$ is obtained by completing the algebraic tensor product $C_c(\cB)\odot Y$ with respect to the pre-inner product given by
 \begin{equation*}
  \langle f\odot \xi,g\odot \eta\rangle =\langle \xi,T_{f^**g(e)}\eta\rangle =\int_G \langle \xi,T_{f(t)^*g(t)}\eta\rangle\, dt=
  \int_G \langle \xi,\rho\circ \pi(f(t)^*g(t))\eta\rangle\, dt.
 \end{equation*}
 We denote $f\otimes_T \xi$ the image of $f\odot \xi$ in $Z.$ 

 Notice that for all $f,g\in C_c(\cB), $ $b\in B_e$ and $\xi,\eta\in Y$ it follows that 
 \begin{equation*}
  \langle f\otimes_T T_b\xi,g\otimes_T\eta\rangle 
  =\int_G \langle \xi,\rho\circ \pi((f(t)b)^*g(t))\eta\rangle\, dt
  = \langle fb\otimes_T \xi,g\otimes_T\eta\rangle.
 \end{equation*}
 Thus $f\otimes_T T_b\xi = fb\otimes_T \xi.$
  
 Take a decomposition of $\rho$ into cyclic subrepresentations $\rho=\bigoplus_{i\in I} \rho^i,$ with $\{\xi_i\}_{i\in I}$ being the corresponding family of cyclic vectors.
 Since the vectors of the form $f\otimes_T T_b \xi_i = fb\otimes_T\xi_i$ span a dense subset of $Z;$ the finite direct sums $\sum_{j=1}^n f_j\otimes_T \xi_{i_j}$ with $i_j\neq i_k$ if $j\neq k$ form a dense subset of $Z.$
 If $\zeta = \sum_{j=1}^n f_j\otimes_T \xi_{i_j}$ is one of such sums and $b\in \cB,$ then condition (2) implies
 \begin{align*}
  \|\Ind_e^\cB(T)_b \zeta \|^2
  & =  \int_G \sum_{j,k=1}^n \langle \xi_{i_j},\rho\circ \pi(f_j(t)^* b^*b f_k(t))\xi_{i_k}\rangle \, dt\\
  & =  \int_G \sum_{j=1}^n \langle \xi_{i_j},\rho\circ \pi(f_j(t)^* b^*b f_j(t))\xi_{i_j}\rangle \, dt\\
  & \leq  \|\pi(b)\|^2 \int_G \sum_{j=1}^n \langle \xi_{i_j},\rho\circ \pi(f_j(t)^* f_j(t))\xi_{i_j}\rangle \, dt
  \leq \|T_b\|^2  \|\zeta\|^2. 
 \end{align*}
 It thus follows that $\|\Ind_e^\cB(T)_b\|\leq \|T_b\|=\|\pi(b)\|$ for all $b\in B_e.$

 Claim (3) implies (4) because \cite[XI 11.3]{FlDr88} guarantees that $\|T_b\|\leq \|\Ind_e^\cB(T)_b\|$ for all $b\in B_e.$
 Notice (4) implies (5) taking $\Ind_{\{e\}}^\cB(T)$ as the representation of (5).
 
 Assume (5) holds and construct the C*-completion $(\cB^T,\rho^T)$ as in Proposition~\ref{prop:completions and representations}.
 Then the unit fibre of $B^T$ is C*-isomorphic to $A$ and by replacing this fibre with $A$ (and $\rho^T_e$ with $\pi$) we get a C*-completion $(\cA,\iota)$ with $\iota_e=\pi.$
 Thus (5) implies (1).
 
 Notice that (2) implies (6).
 Conversely, if $\cB$ has a strong approximate unite $\{u_i\}_{i\in I},$  then for all $b\in \cB$ we have $ \pi(b^*b)=\lim_i \pi(b^* u_i^* u_i b)\geq 0.$
\end{proof}
\end{theorem}

\begin{remark}[Universal property of the bundle C*-completion]\label{rem:universal property of bundle C*-completion}
 Let $\cB=\{B_t\}_{t\in G}$ be a Banach *-algebraic bundle and consider it's bundle C*-completion $(\cC,\iota).$
 For any other C*-completion $(\cA,\kappa)$ there exists a *-representation $T\colon \cA\to \bB(Z)$ such that $\|T_a\|=\|a\|$ for all $a\in \cA.$
 Consequently, $T\circ \kappa$ is a *-representation of $\cB$ and the construction of $\cC$ implies $\|\kappa(b)\|=\|T\circ\kappa(b)\|\leq \|\iota(b)\|$ for all $b\in B_e.$
 Then Corollary~\ref{cor:uniqueness of completion} implies the existence of a unique morphism $\rho\colon \cC\to \cA.$
 So the bundle C*-completion is the universal C*-completion of $\cB.$
\end{remark}

The following result is a generalisation of Proposition~\ref{prop:can change a banach algebraic bundle by its completion}.

\begin{proposition}\label{prop:induciton and completions}
Let $\cB$ be a Banach *-algebraic bundle over $G,$ $H \sbgp G$ and $(\cC,\rho)$ a C*-completion of $\cB.$
We say a *-representation $T\colon \cB_H\to \bB(X)$ is dominated by $\rho$ if $\|T_b\|\leq \|\rho(b)\|$ for all $b\in B_e.$
Then

\begin{enumerate}
 \item Every *-representation of $\cB_H$ dominated by $\rho$ is $\cB-$positive.
 \item\label{item:characterization of dominated reps} For every *-representation $T\colon \cB_H\to \bB(X)$ dominated by $\rho$ there exists a unique *-representation $T^\rho\colon \cC_H\to \bB(X)$ such that $T=T^\rho\circ (\rho|_{\cB_H}).$
 Reciprocally, for every *-representation $S\colon \cC_H\to \bB(Z)$ the composition $T:=S\circ (\rho|_{\cB_H})\colon \cB_H\to \bB(Z)$ is dominated by $\rho$ and $T^\rho =S.$
 \item For every *-representation $T\colon \cB_H\to \bB(X)$ dominated by $\rho,$ $\Ind_H^\cB(T)$ is dominated by $\rho$ and we have a unitary equivalence ($\approx$) of *-representations
 \begin{equation*}
 (\Ind_H^\cB(T))^\rho \approx  \Ind_H^\cC(T^\rho).
 \end{equation*}
\end{enumerate}
\end{proposition}
\begin{proof}
   Let $T\colon \cB_H\to \bB(X)$ be a $*-$representation dominated by $\rho.$
   To prove that $T$ is $\cB-$positive we may assume, without loss of generality, that $T$ is non degenerate.
   
   Take a non degenerate *-representation $R\colon \cC\to \bB(Y)$ such that $R|_{B_e}$ is faithful.
   Then, for all $b\in B_e,$ $\|T_b\|\leq \|\rho(b)\|=\|R\circ \rho (b)\|.$
   This means $T|_{B_e}\colon B_e\to \bB(X)$ is weakly contained in $R\circ \rho|_{B_e}\colon B_e\to \bB(Y),$ the last being a $\cB-$positive *-representation of $\cB$ (because it is a restriction of a *-representation of $\cB$).
   By \cite[XI 8.19]{FlDr88}  $T|_{B_e}$ is $\cB-$positive and this implies $T$ is $\cB-$positive (Theorem~\ref{thm:characterization of positive representations}).
   
   For every $t\in G$ and $b\in B_t$ we have $\|T_b\|=\|T_{b^*b}\|^{1/2}\leq \|\rho(b^* b)\|^{1/2}=\|\rho(b)\|.$
   Since $\rho(B_t)$ is a dense subspace of $C_t,$ there exists a unique linear function $T^{\rho t}\colon C_t\to \bB(X)$ such that $T^{\rho t}\circ \rho|_{B_t}=T|_{B_t}.$
   Let $T\colon \cB_H\to \bB(X)$ be the union of the maps $\{T^{\rho t}\}_{t\in H}.$
   A straight forward density + continuity argument shows that $T^\rho$ is multiplicative and preserves adjoints.
   
   We now want to show that for all $\xi,\eta\in X$ the map $\cC_H\to \bC,\ c\mapsto \langle \xi,T_c\eta\rangle,$ is continuous.
   To do this we use \cite[II 13.16]{FlDr88}.
   Let $\cD$ be the trivial Banach bundle over $H$ with constant fibre $\bC$ and define $\psi\colon \cC_H\to \cD$ by $\psi(c)=(t,\langle \xi,T^\rho_c\eta\rangle),$ for all $c\in C_t$ and $t\in H.$
   Then $\psi$ is linear on each fibre and $|\psi(c)|\leq \|\xi\|\|\eta\|\|c\|$ for all $c\in \cC_H.$
   The family $\Gamma:=\{\rho\circ f\colon f\in C_c(\cB_H)\}$ satisfies condition (iii) of \cite[II 13.16]{FlDr88}; it also satisfies (iv) because for all $\rho\circ f \in \Gamma$ the composition $\psi( \rho\circ f(t) )=(t,\langle \xi,T_{f(t)}\eta\rangle)$ is continuous.
   Thus $\psi$ is continuous, implying that $T^\rho$ is a *-representation of $\cC_H.$
   
   If $R\colon \cC_H\to \bB(X)$ is a *-representation such that $R\circ (\rho|_{\cB_H})=T,$ then $R$ and $T^\rho$ agree on the dense set $\rho(\cB_H).$
   Hence $R=T^\rho.$
   This shows the first half of claim (\ref{item:characterization of dominated reps}); the second half is left to the reader.
   
   To prove the last claim we fix a *-representation $T\colon \cB_H\to \bB(X)$ dominated by $\rho.$
   We denote $X^p_T$ the Hilbert space induced by $T$ via the generalized restriction $p=p^\cB_H\colon C_c(\cB)\to C_c(\cB_H).$
   The Hilbert space induced by $T^\rho$ should be denoted $X^p_{T^\rho}$ but we prefer the more specific term $L^2_H(\cC)\otimes_{T^\rho}X$ instead.
   
   To prove the existence of a bounded operator $U\colon X^p_T\to L^2_H(\cC)\otimes_{T^\rho}X$
    such that $U(f\otimes_T\xi)=\rho\circ f\otimes_{T^\rho}\xi$ (for all $f\in C_c(\cB)$ and $\xi\in X$) it suffices to show that 
    \begin{equation*}
    	\langle f\otimes_{T}\xi,g\otimes_T\eta\rangle = \langle (\rho\circ f)\otimes_{T^\rho}\xi,(\rho\circ g)\otimes_{T^\rho}\eta\rangle;
    \end{equation*}
    which will also show that $U$ is an isometry.
    
    By definition and the construction of $T^\rho:$
    \begin{multline*}
    	\langle f\otimes_{T}\xi,g\otimes_T\eta\rangle 
    	= \int_H \Delta_H(t)^{-1/2}\Delta_G(t)^{1/2}\int_G \langle \xi,T_{f(s)^*g(st)}\eta\rangle \, dsdt  =\\
    	  	= \int_H \Delta_H(t)^{-1/2}\Delta_G(t)^{1/2}\int_G \langle \xi,T^\rho_{(\rho\circ f)(s)^*(\rho\circ g)(st)}\eta\rangle \, dsdt  = \langle (\rho\circ f)\otimes_{T^\rho}\xi,(\rho\circ g)\otimes_{T^\rho}\eta\rangle.
    \end{multline*}
   
   The set of sections $\Gamma$ defined above is dense in $C_c(\cC_H)$ in the inductive limit topology and the image of $U$ contains all the tensor $\{u\otimes_{T^\rho}\xi\colon u\in \Gamma,\xi\in X\},$ then $U$ has dense range and it must be a unitary operator.
   
   We claim that $\Ind_H^\cB(T)_b = U^* \Ind_H^\cC(T^\rho)_{\rho(b)}U$ for all $b\in \cB_H.$
   Indeed, for all $f\in C_c(\cB_H)$ and $\xi\in X$ we have
   \begin{multline*}
      U^* \Ind_H^\cC(T^\rho)_{\rho(b)}U(f\otimes_T\xi)
      =
      U^*((\rho(b)(\rho\circ f))\otimes_{T^\rho}\xi)
      =\\
      =      U^*((\rho\circ bf)\otimes_{T^\rho}\xi)
      = bf\otimes_T\xi
      = \Ind_H^\cB(T)_b(f\otimes_T\xi).
   \end{multline*}
   Consequently, for all  $b\in B_e$
   \begin{equation*}
   	\| \Ind_H^\cB(T)_b \|
   	= \| \Ind_H^\cC(T^\rho)_{\rho(b)} \|\leq \|\rho(b)\|,
   \end{equation*}
   and $\Ind_H^\cB(T)_b$ is dominated by $\rho.$
   
   Let $R\colon \cC\to \bB(X^p_T)$ be given by $R_b= U^* \Ind_H^\cC(T^\rho)_{c}U.$
   Then $R$ is a *-representation of $\cC$ such that $R\circ \rho=\Ind_H^\cB(T),$ so $R=\Ind_H^\cB(T)^\rho$ is unitary equivalent to $\Ind_H^\cC(T^\rho).$
\end{proof}
\subsection{Cross sectional bundles, C*-completions and induction}
We fix, for the rest to this section, a Fell bundle $\cB=\{B_t\}_{t\in G}$ and a closed normal subgroup $N$ of $G.$
We will briefly recall the main properties of the $L^1-$cross sectional bundle over $G/N$ derived from $\cB,$ as constructed in \cite[VIII 6]{FlDr88}.

We regard the quotient group $G/N=\{tN\colon t\in G\}$ as a LCH topological group with the quotient topology.
Given a coset $\alpha=tN$ ($t\in G$) let $\nu_\alpha$ be the regular Borel measure on $tN$ such that $\int_{\alpha} f(x)\, d\nu_\alpha(x)= \int_N f(tx)\, d_Nx$ for all $f\in C_c(tN).$
Here $d_Nx$ is the integration with respect to a (fixed) left invariant Haar Measure of $N.$
There is no ambiguity in the definition of $\nu_\alpha$ because the left invariance of $d_N$ implies the function $tN\to \mathbb{R},$ $r\mapsto \int_N f(rx)\, d_Nx,$ is constant.

For each $f\in C_c(G)$ we define $f^0\colon G\to \mathbb{C}$ as $f^0(t):=\int_Nf(tx)\, d_Nx=\int_{tN} f(x)d\nu_{tN}(x).$
It can be shown that $f^0$ is continuous and constant in the cosets, so it defines a function $f^{00}\in C(G/N)$ that vanishes outside the projection of $\supp(f)$ on $G/N.$
Then, by construction,
\begin{equation*}
 \int_{G/N}f^{00}(x)\, d_{G/N}x = \int_{G/N}d_{G/N}tN \int_N f(ts)\, d_Ns.
\end{equation*}

Throughout this work we assume the left invariant Haar measures of $G,\ N$ and $G/N$ are normalized in such a way that for all $f\in C_c(G)$
\begin{equation}\label{equ:normalization of Haar measures}
 \int_G f(t)\, d_Gt = \int_{G/N} d_{G/N} tN \int_{N} f(ts)\, d_Ns,
\end{equation}
exactly as in \cite[VIII 6.7]{FlDr88}.
In case $G$ is a product $H\times K,$ $N=H$ and $K=(H\times K)/H$; we meet this requirement by considering the product measure $d_G(r,s)=d_Hr \times d_Ks.$

By \cite[VIII 6.5]{FlDr88} there exists a unique continuous homomorphism $\Gamma\colon G\to (0,+\infty)$ such that
\begin{equation*}
 \int_N f(xyx^{-1})\, d_Ny = \Gamma(x) \int_N f(y)\, d_Ny,\ \forall \ x\in G,\ f\in C_c(N).
\end{equation*}
In case $G$ is a product group $H\times K$ and $N=H$ one has $\Gamma(r,s)=\Delta_H(r).$

The $L^1-$partial cross sectional bundle over $G/N$ derived from $\cB,$ $\cC=\{C_\alpha\}_{\alpha\in G/N},$ is determined by the following properties:
\begin{itemize}
 \item For every $\alpha\in  G/N,$ if $B_{\alpha}=\{B_t\}_{t\in \alpha}$ is the reduction of $\cB$ to $\alpha,$ then $C_\alpha$ is the completion of $C_c(\cB_\alpha)$ with respect to the norm $\|f\|_1=\int_\alpha \|f(t)\|\, d\nu_\alpha(t).$
 \item For every $r,s\in G,$ $f\in C_c(\cB_{rN})$ and $g\in C_c(\cB_{sN}),$ the product $f*g\in C_c(\cB_{rsN})\subset C_{rsN}$ and the involution $f^*\in C_c(\cB_{\rmu N})$ are determined by 
 \begin{align*}
  f*g(x)& = \int_{rN} f(y)g(y^{-1}x)\, d\nu_{rN}(y) & f^*(z)&=\Gamma(z)^{-1}f(z^{-1})^*,
 \end{align*}
 for all $x\in rsN$ and $z\in \rmu N.$
 \item Given $f\in C_c(\cB),$ if $f|\colon G/N \to \cC$ is given by the restriction $f|(\alpha):=f|_{\alpha},$ then $f|$ is a continuous cross section.
\end{itemize}

\begin{notation}
 The $L^1-$partial cross sectional bundle over $G/N$ derived from $\cB$ will be denoted $\quotient{L^1}{\cB}{N}=\{L^1(\cB_\alpha)\}_{\alpha\in G/N}.$
 This is more than just a notation because if given the coset $\alpha\in G/H$ we denote by $\cB_\alpha$ the reduction of $\cB$ to $\alpha$, then $\cB_\alpha$ is a Banach bundle and $L^1(\cB_\alpha)$ is (constructively and symbolically) the completion of $C_c(\cB_\alpha)$ with respect to $\|\ \|_1.$
 The usual $L^1-$cross sectional algebra of $\cB$ may be regarded (notationally and concretely) as $L^1[\cB/G]$ and the bundle $\cB$ itself as $L^1[\cB/\{e\}].$
\end{notation}

\begin{remark}
Recall that $\cB$ is a Fell bundle, so it has a strong approximate unit by \cite[VIII 16.3]{FlDr88}.
This is also the case for $\quotient{L^1}{\cB}{N}$ by~\cite[VIII 6.9]{FlDr88}. 
\end{remark}

\begin{remark}\label{rem:ilt dense set of L1 partial cross sectional bundle}
 The set $\{f|\colon f\in C_c(\cB)\}\subset C_c(\quotient{L^1}{\cB}{N})$ is dense in $C_c(\quotient{L^1}{\cB}{N})$ with respect to the inductive limit topology by \cite[II 14.6]{FlDr88} and, consequently, it is dense in $L^1( \quotient{L^1}{\cB}{N} ).$
\end{remark}

\begin{remark}\label{rem: iso of L1 algebras of L1 partial cross sectional bundle}
It is shown in \cite[VIII 6.7]{FlDr88} that there exists a unique isometric isomorphism of Banach *-algebras $\Phi\colon L^1(\cB)\to L^1(\quotient{L^1}{\cB}{N})$ such that $\Phi(f)=f|,$ for all $f\in C_c(\cB).$
\end{remark}

Given a *-representation $T\colon \cB\to \bB(Y_A)$ we can follow \cite[VIII 15.9]{FlDr88} and construct a  *-representation 
\begin{equation}\label{equ:L1[T/N]}
 \quotient{L^1}{T}{N}\colon \quotient{L^1}{\cB}{N}\to \bB(Y_A)
\end{equation}
such that for every coset $\alpha\in G/N,$ $f\in C_c(\cB_\alpha)$ and $\xi\in Y_A,$
\begin{equation*}
 \quotient{L^1}{T}{N}_f\xi = \int_G T_{f(t)}\xi\, d\nu_\alpha(t).
\end{equation*}
Note that $T$ and $\quotient{L^1}{T}{N}$ determine each other because
\begin{equation}\label{equ:integrated form and L1 form}
\quotient{\intform{L^1}}{T}{N}\circ \Phi = \intform{T}.
\end{equation}

In case we are given a non degenerate *-representation $S\colon \quotient{L^1}{\cB}{N}\to \bB(Y),$ the composition $\intform{S}\circ \Phi\colon L^1(\cB)\to \bB(Y_A)$ is a non degenerate *-representation that can be disintegrated to a non degenerate *-representation $T\colon \cB\to \bB(Y_A).$
We have $\quotient{\intform{L^1}}{T}{N}\circ \Phi = \intform{T} = \intform{S}\circ \Phi$ and it follows that  $\quotient{L^1}{T}{N}=S.$

\subsection{Weak containment for cross sectional bundles}
Any subgroup of $G/N$ can be expressed as $H/N$ for a unique subgroup $H$ of $G$ containing $N.$
We fix one such $H$ and make the identification 
\begin{equation*}
 \quotient{L^1}{\cB_H}{N}\equiv \{L^1(\cB_\alpha)\}_{\alpha\in H/N} \equiv \quotient{L^1}{\cB}{N}_{H/N};
\end{equation*}
which yields an isometric isomorphism of Banach *-algebras
\begin{equation}\label{equ:iso L1BHN and L1BH}
 \Phi^H\colon L^1(\cB_H)\to L^1(\quotient{L^1}{\cB_H}{N})\equiv L^1(\quotient{L^1}{\cB}{N}_{H/N}).
\end{equation}
For $H=G$ this is just the isomorphism $L^1(\cB)\cong L^1(\quotient{L^1}{\cB}{N})$ we have considered before.
Besides, we may think of $\Phi^H$ as the restriction of a C*-isomorphism 
\begin{equation}\label{equ:C*B_H as C*something}
 C^*(\cB_H)\cong C^*(\quotient{L^1}{\cB}{N}_{H/N})
\end{equation}
between the corresponding enveloping C*-algebras.

The equivalence of the representations theories of $\cB$ and $\quotient{L^1}{\cB}{N}$ now reduces to a correspondence $T\leftrightsquigarrow \quotient{L^1}{T}{N}$ between non degenerate *-representations of $\cB_H$ and of $\quotient{L^1}{\cB_H}{N}\equiv \quotient{L^1}{\cB}{N}_{H/N}.$

\begin{remark}\label{rem:every rep of L1BN is positive}
 Given a *-representation $T\colon \cB_H\to \bB(Y),$ the Proposition in \cite[XI 12.7]{FlDr88} implies $\quotient{L^1}{T}{N}$ is $\quotient{L^1}{\cB}{N}-$positive if and only if $T$ is $\cB-$positive.
Thus every *-representation of $\quotient{L^1}{\cB}{N}_{H/N}$ is $\quotient{L^1}{\cB}{N}-$positive by Theorem~\ref{thm:characterization of positive representations}.
\end{remark}

There is still another thing we want to extract from \cite[XI 12.7]{FlDr88}: equation (6) from page 1164.
In our notation it becomes
\begin{equation}\label{equ:restrictions, integrations and integrated forms}
 \quotient{\intform{L^1}}{T}{N}_{p^{G/N}_{H/N}(\Phi(f)^**\Phi(g))}=\quotient{\intform{L^1}}{T}{N}_{p^{G/N}_{H/N}(\Phi(f^**g))} = \intform{T}_{p^G_H(f^**g)};
\end{equation}
Here the $p$ functions are the generalized restrictions for the groups indicated in the super and sub indexes.

Identity \eqref{equ:restrictions, integrations and integrated forms} is used in \cite[XI 12.7]{FlDr88} to establish a unitary equivalence
\begin{equation}\label{equ:L1[T/N] and induction}
 \quotient{L^1}{\Ind_{H}^\cB(T)}{N} \cong \Ind_{H/N}^{\quotient{L^1}{\cB}{N}}(\quotient{L^1}{T}{N}).
\end{equation}

\begin{proposition}\label{prop:iso of H corss sectional algebra and cross sectional bundle}
 There exists a unique morphism of C*-algebras $\pi_H$ such that
 \begin{equation*}
 \xymatrix{L^1(\cB)\ar[rr]^{\Phi}\ar[rrd]_{\iota_H}& &L^1( \quotient{L^1}{\cB}{N} ) \ar[rr]^{\iota_{H/N}} & & C^*_{H/N}(\quotient{L^1}{\cB}{N})\\
& & C^*_H(\cB) \ar[rru]_{\pi_H} & & }
 \end{equation*}
 commutes; where $\Phi$ is the map from Remark~\ref{rem: iso of L1 algebras of L1 partial cross sectional bundle} and the $\iota$ maps are those of Definition~\ref{defi:of H cross sectional algebra}.
 Moreover, $\pi_H$ is a C*-isomorphism.
 \begin{proof}
  Equation \eqref{equ:L1[T/N] and induction} and the definition of the norm used to construct the cross sectional C*-algebras imply that for all $f\in L^1(\cB)$ we have $\|f\|^\cB_H = \|\Phi(f)\|^{\quotient{L^1}{\cB}{N}}_{H/N}.$
  Then the proof follows directly from Definition~\ref{defi:of H cross sectional algebra}.
 \end{proof}
\end{proposition}

The next proposition is the Corollary in \cite[XI 12.9]{FlDr88} stated for general Fell bundles, and not just for saturated ones.

\begin{proposition}\label{prop:universal completion of cross sectional bundle}
 The bundle C*-completion of $\quotient{L^1}{\cB}{N}=\{L^1(\cB_\alpha)\}_{\alpha\in G/N}$ is the C*-completion of $\quotient{L^1}{\cB}{N}$ with respect to the universal C*-completion $L^1(\cB_N)\to C^*(\cB_N)$ of the unit fibre of $\quotient{L^1}{\cB}{N}.$
\begin{proof}
 Let $\cC=\{C_\alpha\}_{\alpha\in G/N}$ be the bundle C*-completion of $\quotient{L^1}{\cB}{N}$ and 
 \begin{equation*}
  \rho=\{\rho_\alpha\}_{\alpha\in G/N}\colon \quotient{L^1}{\cB}{N}\to \cC
 \end{equation*}
 the canonical morphism.
 We clearly have $\|\rho_H(f)\|\leq \|f\|_{C^*(\cB_N)}$ for all $f\in L^1(\cB_N).$
 
 Take a faithful and non degenerate *-representation $T\colon C^*(\cB_N)\to \bB(Y).$
 The restriction $R:=T|_{L^1(\cB_N)}$ is $\quotient{L^1}{\cB}{N}-$positive by Remark~\ref{rem:every rep of L1BN is positive}.
 By \cite[XI 11.3]{FlDr88} we know that for all $f\in L^1(\cB_N),$ $\|R_f\|\leq \|\Ind_{\{H\}}^{\quotient{L^1}{\cB}{N}}(R)_f\|\leq \|f\|_{C^*(\cB_N)}=\|R_f\|.$
 Then there exists a C*-completion of $\quotient{L^1}{\cB}{N}$ with respect to $L^1(\cB_N)\to C^*(\cB_N).$ The rest of the proof follows from the universal property of $(\cC,\rho)$ of Remark~\ref{rem:universal property of bundle C*-completion} and Corollary~\ref{cor:uniqueness of completion}.
\end{proof}
\end{proposition}

\begin{notation}
 In the situation of the Proposition above, the bundle C*-completion of $\quotient{L^1}{\cB}{N}$ will be denoted $\quotient{C^*}{\cB}{N}=\{C^*(\cB_\alpha)\}_{\alpha\in G/N}.$
 This makes sense because the unit fibre of $\quotient{C^*}{\cB}{N}$ is, both symbolically and concretely, $C^*(\cB_N).$
\end{notation}

With the notation above, Propositions~\ref{prop:universal completion of cross sectional bundle} and~\ref{prop:iso of H corss sectional algebra and cross sectional bundle} and Remark~\ref{rem:H cross sectional algebras of bundle an completion} imply we have a C*-isomorphism
\begin{equation*}
 \psi_H\colon C^*_H(\cB)\to C^*_{H/N}( \quotient{C^*}{\cB}{N} )
\end{equation*}
that extends $\Phi\colon L^1(\cB)\to L^1(\quotient{L^1}{\cB}{N})$ (for every $H \sbgp G$ containing $N$).
In particular, for $H=N,$
\begin{equation*}
 \psi_N\colon C^*_N(\cB)\to C^*_\red( \quotient{C^*}{\cB}{N} ).
\end{equation*}

The identity above should not surprise anybody who is familiar with Chapter XII of \cite{FlDr88}.
In fact the proof (for saturated bundles) can be inferred out of \cite[pp 1264]{FlDr88}.

\begin{corollary}\label{cor:subgroups containing N}
 Given subgroups $H$ and $K$ such that $N \sbgp H \sbgp K \sbgp G,$ the diagram
 \begin{equation}\label{equ:q and psi maps for cross sectional bundles}
  \xymatrix{
  C^*_K(\cB)\ar[rr]^{q^\cB_{KH}} \ar[d]_{\psi_K} & & C^*_H(\cB)\ar[d]^{\psi_H}\\
  C^*_{K/N}(\quotient{C^*}{\cB}{N})\ar[rr]_{q^{\quotient{C^*}{\cB}{N}}_{(K/N) (H/N)}}& & C^*_{H/N}(\quotient{C^*}{\cB}{N})}
 \end{equation}
 commutes.
 Thus $\cB$ has the $HK$-WCP if and only if $\quotient{C^*}{\cB}{N}$ has the $(H/N)(K/N)-$WCP.
 In case both $H$ and $K$ are normal in $G,$ the preceding claims are equivalent to say $\cB_K$ has the $H-$WCP.
 In particular, for $H=N$ and $K=G$ this gives the equivalence of the following claims:
 \begin{enumerate}
  \item $\cB$ has the $N-$WCP.
  \item $\quotient{C^*}{\cB}{N}$ has WCP.
 \end{enumerate}
 \begin{proof}
  The commutativity of the diagram can be proved by computing the arrows in the dense set $L^1(\cB).$
  The rest is straightforward from the comments preceding the statement and Corollaries~\ref{cor:isomorphisms of q maps} and~\ref{cor:qBKH iso iff qBKH iso}.
 \end{proof}
\end{corollary}

\begin{corollary}\label{cor:qBH iso and amenability}
 If $H \sbgp N$ is normal in $G,$ then $\cB$ has the $H-$WCP if and only if $\quotient{C^*}{\cB}{N}$ has the WCP and $\cB_N$ has the $H-$WCP.
 In particular, for $H=\{e\}$ this says $\cB$ has the WCP if and only if both $\quotient{C^*}{\cB}{N}$ and $\cB_N$ have the WCP.
 \begin{proof}
  By Remark~\ref{rmk:composition of q maps},  $q^\cB_H=q^\cB_{NH}\circ q^\cB_H.$
  Hence, $q^\cB_H$ is a C*-isomorphism if and only if both $q^\cB_{NH}$ and $q^\cB_N$ are C*-isomorphism.
  By Corollary~\ref{cor:subgroups containing N}, $q^\cB_{N}$ is a C*-isomorphism if and only if $\quotient{C^*}{\cB}{N}$ has the WCP and Corollaries~\ref{cor:qBKH iso iff qBKH iso} and~\ref{cor:isomorphisms of q maps} imply $q^{\cB}_{NH}$ is a C*-isomorphism if and only if $\cB_N$ has the WCP.
  \end{proof}
\end{corollary}

\begin{corollary}\label{cor:G/N inner amenable and otimesmax=otimesmin}
 If $G/N$ is inner amenable and the canonical map
 \begin{equation*}
  C^*_\red(G/N)\otimes_{\max}C^*_N(\cB)\to C^*_\red(G/N)\otimes C^*_N(\cB)
 \end{equation*}
 is faithful, then $\cB$ has the $N-$WCP.
 \begin{proof}
  By Corollary~\ref{cor:subgroups containing N}, it suffices to show that the left regular representation of $C^*(\quotient{C^*}{\cB}{N})$ is faithful (i.e. $\quotient{C^*}{\cB}{N}$ has the WCP).
  Corollary~\ref{cor:inner amenability and nuclearity} gives the desired result when applied to the bundle $\quotient{C^*}{\cB}{N}.$
 \end{proof}
\end{corollary}

\begin{corollary}\label{cor:G/N amenable implies NWCP}
If $G/N$ is amenable, then every Fell bundle over $G$ has the $N-$WCP.
\end{corollary}

Every subgroup $H \sbgp N$ gives a C*-completion $\iota_H\colon L^1(\cB_N)\to C^*_H(\cB_N),$ $f\mapsto q^{\cB_N}_H(f),$ of the unit fibre of $\quotient{L^1}{\cB}{N}.$ 
If $H$ is normal in $G$ this completion gives rise to a C*-completion with particularly nice properties, as we show below.

\begin{theorem}\label{thm:HWCP equivalent conditions H normal included in N}
 Let $\cB=\{B_t\}_{t\in G}$ be a Fell bundle and consider normal subgroups of $G,$ $H \sbgp  N \sbgp G.$
 Then the C*-completion $\iota_H\colon L^1(\cB_N)\to C^*_H(\cB_N)$ of the unit fibre of $\quotient{L^1}{\cB}{N}$ satisfies the equivalent conditions of Theorem~\ref{thm:equivalent conditions for the existence of completions}.
 For the $\iota_H-$completion of $\quotient{L^1}{\cB}{N},$ $(\quotient{C^*_H}{\cB}{N},\kappa),$ there exists a unique *-homomorphism $\pi\colon C^*_H(\cB)\to C^*_\red(\quotient{C^*_H}{\cB}{N}) $ such that, for all $f\in C_c(\cB),$ $\pi(f)= \kappa\circ (\Phi(f));$  with $\Phi$ being that of Remark~\ref{rem: iso of L1 algebras of L1 partial cross sectional bundle}.
 Moreover, $\pi$ is a C*-isomorphism and the following are equivalent:
 \begin{enumerate}
  \item\label{item:B has the HWCP} $\cB$ has the $H-$WCP.
  \item\label{item:CBN has the WCP and BN HWCP} $\quotient{C^*}{\cB}{N}$ has the WCP and $\cB_N$ has the $H-$WCP.
  \item\label{item:CHBN has the WCP and BN HWCP} $\quotient{C^*_H}{\cB}{N}$ has the WCP and $\cB_N$ has the $H-$WCP.
  \item\label{item:CBN has the WCP and iso} $\quotient{C^*}{\cB}{N}$ has the WCP and the morphism  $\rho\colon \quotient{C^*}{\cB}{N}\to \quotient{C^*_H}{\cB}{N}$ provided by Corollary~\ref{cor:uniqueness of completion}, is an isomorphism.
 \end{enumerate}
 \begin{proof}  
  Let $T\colon \cB_H\to \bB(Y)$ be a non degenerate *-representation with faithful integrated form.
  Then the integrated form of $R:=\Ind_H^{\cB_N}(T)\colon \cB_N\to \bB(Z)$ factors via $q^{\cB_N}_H$ through a faithful *-representation $\intform{R}^q\colon C^*_H(\cB_N)\to \bB(Z).$
  It will be convenient to make a notational distinction between the integrated form of $\intform{R}$ and its restriction to $L^1(\cB_N),$ we denote $\intform{R}|$ the latter.
    
  By the induction in stages and \eqref{equ:L1[T/N] and induction} (for the subgroup $\{N\}\subset G/N$) we have the unitary equivalences
  \begin{equation}\label{equ:l1 induced}
    \quotient{L^1}{\Ind_H^\cB(T)}{N}\cong \quotient{L^1}{\Ind_N^\cB(R)}{N} \cong \Ind_{\{N\}}^{\quotient{L^1}{\cB}{N}}(\intform{R}|). 
  \end{equation}
  
  By construction $\Ind_H^\cB(T) = \regrep{H\cB}\otimes_{\intform{T}} 1.$
  Then the restriction of $\Ind_{\{N\}}^{\quotient{L^1}{\cB}{N}}(\intform{R}|)$ to the unit fibre $L^1(\cB_N)$ of $\quotient{L^1}{\cB}{N}$ is the integrated form of $ (\regrep{H\cB}|_{\cB_N}) \otimes_{\intform{T}} 1. $
  Theorem~\ref{thm:inclusion CHBK in BCHB} implies that for all $f\in L^1(\cB_N)$
  \begin{equation*}
   \| \Ind_{\{N\}}^{\quotient{L^1}{\cB}{N}}(\intform{R}|)_f \| = \| \iota_H(f)  \| =\| \intform{R}|_f \|;
  \end{equation*}
  which shows that $\iota_H$ satisfies the equivalent conditions of Theorem~\ref{thm:equivalent conditions for the existence of completions}

  The map $\pi_0\colon C_c(\cB)\to C_c(\quotient{C^*_H}{\cB}{N}),\ f\mapsto \kappa\circ (\Phi(f)),$ is a *-homomorphism of $*-$algebras, which we want to extend to construct the map $\pi$ of the thesis.
  
  In the notation of Proposition  \eqref{prop:induciton and completions}, $(\intform{R}|)^\kappa =\intform{R}^q$ and we have a unitary equivalence
  \begin{equation*}
    \Ind_{\{N\}}^{\quotient{L^1}{\cB}{N}}(\intform{R}|) \approx \Ind_{\{N\}}^{\quotient{C^*_H}{\cB}{N}}(\intform{R}^q)\circ\kappa.
  \end{equation*}
  Combining this with  \eqref{equ:integrated form and L1 form} and the fact that $\intform{R}^q$ is faithful we get that for all $f\in C_c(\cB),$
  \begin{multline*}
     \| f \|_H^\cB 
      = \| \intform{\Ind}_H^\cB(T)_f \|
      = \| \quotient{\intform{L^1}}{\Ind_N^\cB(R)}{N}_{\Phi(f)} \|
      =  \| \Ind_{\{N\}}^{\quotient{L^1}{\cB}{N}}(\intform{R})_{\Phi(f)} \|
      =\\
      =\|\Ind_{\{N\}}^{\quotient{C^*_H}{\cB}{N}}(S)_{\kappa\circ (\Phi(f))} \|
      = \|\pi_0(f)\|_{\{N\}}^{\quotient{C^*_H}{\cB}{N}}.
  \end{multline*}
  Hence, $\pi_0$ has a unique continuous extension $\pi\colon C^*_H(\cB)\to C^*_\red(\quotient{C^*_H}{\cB}{N}).$
  In fact $\pi$ is isometric and has dense range because the range of $\pi_0$ is dense in $C_c(\quotient{C^*_H}{\cB}{N})$ with respect to the inductive limit topology. 
  This shows $\pi$ is a C*-isomorphism.

  Claims (\ref{item:B has the HWCP}) and (\ref{item:CBN has the WCP and BN HWCP}) are equivalent by Corollary~\ref{cor:qBH iso and amenability}.
  By Corollary~\ref{cor:uniqueness of completion}, the $\rho$ of (\ref{item:CBN has the WCP and iso}) is an isomorphism if and only if   $\rho|_{\quotient{C^*}{\cB}{N}_e}\equiv q^{\cB_N}_H\colon C^*(\cB_N)\to C^*_H(\cB_N)$ is a C*-isomorphism; in which case $\quotient{C^*_H}{\cB}{N}$ has the WCP if and only if $\quotient{C^*}{\cB}{N}$ has the WCP.
  This gives the equivalence between (\ref{item:CBN has the WCP and BN HWCP}), (\ref{item:CHBN has the WCP and BN HWCP}) and (\ref{item:CBN has the WCP and iso}).
 \end{proof}
\end{theorem}

For the trivial group $H=\{e\}$ the Theorem above produces a C*-completion $\quotient{C^*_\red}{\cB}{N}$ of $\quotient{L^1}{\cB}{N}$ with respect to the regular representation $L^1(\cB_N)\to C^*_\red(\cB_N).$
Then we obtain the following immediate consequence.

\begin{corollary}
 Given a Fell bundle $\cB=\{B_t\}_{t\in G}$ and a normal subgroup $N \sbgp G,$ the following are equivalent:
 \begin{enumerate}
  \item $\cB$ has the WCP.
  \item Both $\quotient{C^*}{\cB}{N}$ and $\cB_N$ have the WCP.
  \item Both $\quotient{C^*_\red}{\cB}{N}$ and $\cB_N$ have the WCP.
  \item $\quotient{C^*}{\cB}{N}$ has the WCP and the C*-completions $(\quotient{C^*}{\cB}{N},\iota_N)$ and $(\quotient{C^*_\red}{\cB}{N},\iota_e)$ are isomorphic as C*-completions of $\quotient{L^1}{\cB}{N}.$
 \end{enumerate}
\end{corollary}

Our last result is similar to Theorem~\ref{thm:contiditional expectation open subgroup}, but it is not exactly  the same (consider the hypotheses in the case $H=N$).

\begin{corollary}
	Given a Fell bundle $\cB=\{B_t\}_{t\in G},$ an open subgroup $H \sbgp G$ and a normal subgroup $N \sbgp G$ such that $N \sbgp H \sbgp G,$ the inclusion of *-algebras $C_c(\cB_H)\subset C_c(\cB)$ can be extended to an inclusion of C*-algebras $C^*_N(\cB_H)\subset C^*_N(\cB).$
	Moreover, there is a unique conditional expectation $P_N\colon C^*_N(\cB)\to C^*_N(\cB_H)$ extending the restriction map $C_c(\cB)\to C_c(\cB_H),$ $f\mapsto f|_H.$	
\end{corollary}
\begin{proof}
 Let $K:=H/N\subset G/N$ be the projection of $H.$
 Then $K$ is an open subgroup of $G/N$ and, by Theorem~\ref{thm:contiditional expectation open subgroup}, there is a unique conditional expectation
 \begin{equation*}
P\colon C^*_\red(\quotient{C^*}{\cB}{N})\to C^*_\red( \quotient{C^*}{\cB}{N}_{K} )
 \end{equation*}
 extending the restriction map $p\colon C_c(\quotient{C^*}{\cB}{N})\to C_c(\quotient{C^*}{\cB}{N}_K).$
 
 Theorem~\ref{thm:HWCP equivalent conditions H normal included in N} gives a C*-isomorphism $\pi\colon C^*_N(\cB)\to C^*_\red(\quotient{C^*}{\cB}{N})\equiv C^*_\red(\quotient{C^*_N}{\cB}{N})$
 and we want to identify $C^*_N(\cB_H)$ with $C^*_\red( \quotient{C^*}{\cB}{N}_{K} )$ using $\pi.$
 On one hand, we may think of $\quotient{C^*}{\cB}{N}_K= \{C^*(\cB_\alpha)  \}_{\alpha \in K}\equiv \{C^*(\cB_\alpha)  \}_{\alpha \in H/N}$ as the completion of $\quotient{L^1}{\cB}{N}_K=\{L^1(\cB_\alpha)\}_{\alpha\in H/N}$ with respect to the  (universal) completion $\regrep{\cB_N}\colon L^1(\cB_N)\to C^*(\cB_N)$ of its unit fibre.
 On the other hand, $\quotient{L^1}{\cB}{N}_K=\{L^1((\cB_H)_\alpha)\}_{\alpha\in H/N}$ may be seen as $\quotient{L^1}{\cB_H}{N}$ (considering $N$ as a normal subgroup of $H$).
 Thus $\quotient{C^*}{\cB}{N}_K$ is the bundle C*-completion of $\quotient{L^1}{\cB_H}{N},$ $\quotient{C^*}{\cB_H}{N}.$
 Theorem~\ref{thm:HWCP equivalent conditions H normal included in N} now gives a C*-isomorphism $\pi_H\colon C^*_N(\cB_H)\to C^*_\red(\quotient{C^*}{\cB}{N})$ that we use to construct the injective morphism of C*-algebras $\iota \colon C^*_N(\cB_H)\to C^*_N(\cB), f\mapsto \pi^{-1}(\pi_H(f)).$ 
 
 We leave to the reader the verification of the fact that $\iota$ is the unique *-homomorphism extending the natural inclusion $C_c(\cB_H)\subset C_c(\cB).$ 
 The reader can also check that the natural candidate for $P_N,$ the map $f\mapsto \pi_H^{-1}(P(\pi(f))),$ does indeed satisfy the desired conditions.
\end{proof}

We close this article leaving a problem we would like to solve in the future.
Consider a group $G$ and an open subgroup $H \sbgp G$.
Is it true that every Fell bundle $\cB$ over $G$ with $C^*_H(\cB)$ nuclear has the $H-$WCP? 
This is the case if $H$ is normal (Corollary~\ref{cor:G/N inner amenable and otimesmax=otimesmin}). 
If, in addition, $H$ is inner amenable, then Proposition~\ref{prop:inner amenable subgroup and downward WCP} implies $\cB$ has the WCP (which is stronger than the $H-$WCP).

\bibliography{/home/damian/Nextcloud/Bibliografia/FerraroBiblio}
\bibliographystyle{plain}
\end{document}